\documentclass[11pt]{amsart}

\usepackage{color}
\usepackage{hyperref}
 
\usepackage[all]{hypcap}
\hypersetup{colorlinks,bookmarksopen,bookmarksnumbered,citecolor=blue,pdfstartview=FitH}
\usepackage{xcolor}
\hypersetup{
    colorlinks,
    linkcolor={red!50!black},
    citecolor={blue!50!black},
    urlcolor={blue!80!black}
}
\usepackage{amsmath,amsfonts,amssymb,amscd,amsthm,amsxtra}
\usepackage{graphicx}
\usepackage{xypic}
\usepackage[margin=3cm]{geometry} 
\newtheorem{theorem}[equation]{Theorem}
\newtheorem*{theorem*}{Theorem}
\newtheorem{proposition}[equation]{Proposition}
\newtheorem*{proposition*}{Proposition}
\newtheorem{lemma}[equation]{Lemma}
\newtheorem{corollary}[equation]{Corollary}
\newtheorem*{corollary*}{Corollary}

\newtheorem{conjecture}[equation]{Conjecture} 

\theoremstyle{definition}

\newtheorem{definition}[equation]{Definition}

\theoremstyle{remark}
\newtheorem{remark}[equation]{Remark}

\makeatletter 
\@addtoreset{equation}{subsection} 
\makeatother

\makeatletter
\newcommand{\symbitem}[1]{\item[#1]%
\renewcommand{\@currentlabel}{#1}\ignorespaces}
\makeatother
\def \AA {\mathcal{A}}
\def \CC {\mathcal{C}}
\def \D {\mathcal{D}}
\def \G {\mathbb{G}}
\def \L {\mathcal{L}}
\def \N {\mathcal{N}}
\def \O {\mathcal{O}}
 \def \cS {\mathcal{S}} 
\def \T {\mathcal{T}}
\def \U {\mathcal{U}}
\def \V {\mathcal{V}}

\def \C {\mathbb{C}}
\def \F {\mathbb{F}}
\def \P {\mathbb{P}}
\def \Q {\mathbb{Q}}
\def \R {\mathbb{R}}
\def \Z {\mathbb{Z}}

\def \k {\C}
\def \u {\mathbf{u}}
\def \iu {\mathtt{i}}

\def \Gg {\widehat{\Gamma}} % \Gg = Gamma-class  \Gamma(1+x)
\DeclareMathOperator{\ch}{ch} % usual Chern character \sum e^{q_i}
\DeclareMathOperator{\Ch}{Ch} % $\Ch = (2 \pi i)^{Gr} \ch$

\def \im {\operatorname{Im}}
\def \ge {\geqslant}

\def \le {\leqslant}
\def \leq {\leqslant}

\newcommand{\Ang}[1]{\left\langle #1 \right\rangle} 
\newcommand{\Ker}{\operatorname{Ker}}
\newcommand{\Spec}{\operatorname{Spec}}
 
\newcommand{\udot}{{\:\raisebox{3pt}{\text{\circle*{1.5}}}}}
\newcommand{\ldot}{{\:\raisebox{2pt}{\text{\circle*{1.5}}\:\!}}}

\DeclareMathOperator{\Hom}{Hom}

\DeclareMathOperator{\Ext}{Ext}
\DeclareMathOperator{\Res}{Res}

\DeclareMathOperator{\rank}{rank}
\DeclareMathOperator{\End}{End} 
\DeclareMathOperator{\id}{id}

\def\rev-v{\color{red}}

\def\be{\begin{equation}}
\def\ee{\end{equation}}

\def\parfrac#1#2{{\frac{\partial #1}{\partial #2}}}

\newcommand{\Eff}{\operatorname{Eff}} 
\newcommand{\diag}{\operatorname{diag}} 
\newcommand{\MRS}{\mathfrak{Ref}} 
\newcommand{\re}{\operatorname{Re}} 
\newcommand{\LA}{\operatorname{Lead}} 
\newcommand{\Td}{\operatorname{Td}} 
\newcommand{\td}{\operatorname{td}} 
\newcommand{\pt}{\operatorname{pt}} 
\newcommand{\Sat}{\operatorname{Sat}} 
\newcommand{\Sym}{\operatorname{Sym}} 
\newcommand{\Elem}{\operatorname{Elem}} 
\newcommand{\Euler}{\operatorname{Euler}} 
\newcommand{\sgn}{\operatorname{sgn}} 
 
\newcommand{\Auteq}{\operatorname{Auteq}} 
\newcommand{\Stab}{\operatorname{Stab}}

\newcommand{\hnabla}{\widehat{\nabla}} 
\newcommand{\tnabla}{\widetilde{\nabla}} 
\newcommand{\hy}{\hat{y}} 
\newcommand{\hs}{\hat{s}} 
\newcommand{\my}{y^-} 
\newcommand{\mY}{Y^-}

\newcommand{\hA}{\widehat{A}}
\newcommand{\tchi}{\tilde{\chi}} 

\newcommand{\tS}{\widetilde{S}} 

\newcommand{\hG}{\widehat{G}} 

\newcommand{\bx}{\mathbf{x}} 

\newcommand{\frM}{\mathfrak{M}} 
\newcommand{\tfrM}{\widetilde{\mathfrak{M}}}
\newcommand{\frS}{\mathfrak{S}}

\title[Gamma Conjectures for Fano Manifolds]
{Gamma classes and quantum cohomology of Fano manifolds: 
Gamma Conjectures} 

\author[Galkin]{Sergey Galkin} 
\address{National Research University Higher School of Economics \\
Faculty of Mathematics and Laboratory of Algebraic Geometry \\ 
Vavilova str. 7 \\ 
Moscow 117312 \\ Russia}
\email{sergey.galkin@phystech.edu} 
\author[Golyshev]{Vasily Golyshev}
\address{Algebra and Number Theory Sector \\ Institute 
for Information Transmission Problems \\ 
Bolshoy Karetny per.19 \\
Moscow 127994 \\ Russia} 
\email{golyshev@mccme.ru} 
\author[Iritani]{Hiroshi Iritani}
\address{Department of Mathematics \\ 
Graduate School of Science \\ 
Kyoto University \\  
Kitashirakawa-Oiwake-cho \\ Sakyo-ku \\ 
Kyoto \\ 606-8502 \\ Japan} 
\email{iritani@math.kyoto-u.ac.jp} 
\begin{document}
\begin{abstract} 
We propose \emph{Gamma Conjectures} for Fano manifolds 
which can be thought of as a square root of the index theorem. 
Studying the exponential asymptotics of solutions to the quantum differential 
equation, we associate a \emph{principal asymptotic class} $A_F$ to 
a Fano manifold $F$. We say that $F$ satisfies \emph{Gamma Conjecture I} if 
$A_F$ equals the Gamma class $\Gg_F$. 
When the quantum cohomology of $F$ is semisimple, we say that $F$ satisfies 
\emph{Gamma Conjecture II} if the columns of the central connection 
matrix of the quantum cohomology are formed by $\Gg_F \Ch(E_i)$ 
for an exceptional collection $\{E_i\}$ in the derived category of coherent sheaves
$\D^b_{\rm coh}(F)$. Gamma Conjecture II refines part (3) of 
Dubrovin's conjecture \cite{Dubrovin98}. 
We prove Gamma Conjectures 
for projective spaces 
% toric manifolds, certain toric complete intersections 
and Grassmannians.  
\end{abstract} 
\maketitle 

\tableofcontents

\section{Introduction} 
\subsection{Gamma class} 
\label{subsec:Gamma} 
The \emph{Gamma class} of a complex manifold $X$ is the cohomology class 
\[
\Gg_X = \prod_{i=1}^n \Gamma(1+ \delta_i) \in H^\udot(X,\R)
\]
where $\delta_1,\dots,\delta_n$ are the Chern roots of the tangent 
bundle $TX$ and $\Gamma(x)$ is Euler's Gamma function. 
A well-known Taylor expansion for the Gamma function implies 
that it is expanded in the Euler constant $C_{\rm eu}$ and 
the Riemann zeta values $\zeta(k)$, $k=2,3,\dots$:  
\[
\Gg_X = \exp\left(- C_{\rm eu} c_1(X) + 
\sum_{k\ge 2} (-1)^k (k-1)! \zeta(k) \ch_k(TX) \right). 
\]
The Gamma class $\Gg_X$ has a loop space 
interpretation \cite{Lu, Iritani:ttstar}. 
Let $\L X$ denote the free loop space of $X$ and consider 
the locus $X\subset \L X$ of constant loops. The normal 
bundle $\N$ of $X$ in $\L X$ has a natural $S^1$-action 
(by loop rotation) and splits into the sum $\N_+ \oplus \N_-$ 
of positive and negative representations. 
By the $\zeta$-function regularization (see \cite{Lu} 
and Appendix \ref{app:zeta}), we obtain  
\begin{equation} 
\label{eq:regularization}
\frac{1}{e_{S^1}(\N_+)} = \frac{1}{
\prod_{k=1}^\infty e_{S^1}(TX \otimes L^k)}  
\sim (2\pi)^{-\frac{\dim X}{2}} z^{-\mu} z^{c_1(X)} \Gg_X  
\end{equation} 
where $L$ denotes the $S^1$-representation of weight one, 
$z$ denotes a generator of $H_{S^1}^2(\pt)$ such that 
$\ch(L) = e^z$, $\mu\in \End(H^\udot(X))$ denotes 
the grading operator defined by $\mu(\phi) = (p-\frac{\dim X}{2}) 
\phi$ for $\phi \in H^{2p}(X)$ 
and $z^{-\mu} =e^{-(\log z) \mu}$. 
Therefore $\Gg_X$ can be regarded as 
a localization contribution from constant maps in Floer theory, 
cf.~Givental's equivariant Floer theory \cite{Givental:ICM}.  

The Gamma class can be also regarded as a `square root' of the 
Todd class (or $\hA$-class). The Gamma function identity: 
\[
\Gamma(1-z) \Gamma(1+z) = \frac{\pi z}{\sin \pi z} 
= \frac{2\pi\iu z}{1- e^{-2\pi\iu z}} e^{-\pi \iu z} 
\]
implies that we can factorize the Todd class in 
the Hirzebruch-Riemann--Roch (HRR) formula 
as follows: 
\begin{align}
\label{eq:HRR_Gamma} 
\begin{split}  
\chi(E_1,E_2) & = \int_F \ch(E_1^\vee) \cup \ch(E_2) \cup \td_X \\
& = \left [ \Gg_X \Ch(E_1), \Gg_X \Ch(E_2) \right) 
% & = \frac{1}{(2\pi)^{\dim X}} 
% \int_F \left(e^{\pi\iu c_1(X)} e^{\pi \iu \mu} \Gg_X \Ch(E_1)\right) 
% \cup \Gg_X \Ch(E_2)  
\end{split} 
\end{align} 
where $\chi(E_1,E_2) = \sum_{p=0}^{\dim X} (-1)^p 
\dim \Ext^p(E_1,E_2)$ is the Euler pairing of vector bundles $E_1$, $E_2$, 
$\Ch(E_i) = \sum_{p=0}^{\dim X} (2\pi\iu)^p \ch_p(E_i)$ is the 
modified Chern character and 
\[
[A,B) := \frac{1}{(2\pi)^{\dim X}} \int_X 
(e^{\pi\iu c_1(X)} e^{\pi\iu \mu} A) \cup B
\]
is a non-symmetric pairing on $H^\udot(X)$. 
% where $\mu \in \End(H^\udot(X))$ is the grading operator defined 
% by $\mu(\phi) = (p - \frac{\dim X}{2}) \phi$ for $\phi \in H^{2p}(X)$. 
Geometrically this factorization corresponds to the decomposition 
$\N = \N_+ \oplus \N_-$ of the normal bundle. 
Recall that Witten and Atiyah \cite{Atiyah:circular} derived 
heuristically the index 
theorem by identifying the $\hA$-class with $1/e_{S^1}(\N)$. 
In this sense, the Gamma conjectures can be regarded as 
\emph{a square root of the index theorem}. 

\subsection{Gamma conjectures} 
The Gamma conjectures relate the quantum cohomology 
of a Fano manifold and the Gamma class in terms of differential 
equations. For a Fano manifold $F$, 
the quantum cohomology algebra $(H^\udot(F), \star_0)$ 
(at the origin $\tau=0$) defines the \emph{quantum connection} 
\cite{Dubrovin94}: 
\[
\nabla_{z\partial_z} = z \parfrac{}{z} - \frac{1}{z} (c_1(F)\star_0) + \mu 
\]
acting on $H^\udot(F) \otimes \C[z,z^{-1}]$. 
It has a regular singularity at $z=\infty$ and 
an irregular singularity at $z=0$. Flat sections near $z=\infty$ are 
constructed by the so-called Frobenius method and 
can be put into correspondence with cohomology classes 
in a natural way. 
Flat sections near $z=0$ are classified by their 
exponential growth order (along a sector). 
Our underlying assumption is Property $\O$ 
(Definition \ref{def:conjO}) which roughly says that 
$c_1(F)\star_0$ has a simple eigenvalue $T>0$ of the 
biggest norm. Under Property $\O$, we can single out 
a flat section $s_1(z)$ with the smallest asymptotics 
$\sim e^{-T/z}$ as $z\to +0$ along $\R_{>0}$; then we 
transport the flat section $s_1(z)$ to $z=\infty$ and identify the 
corresponding cohomology class $A_F$. 
We call $A_F$ the \emph{principal asymptotic class} of a 
Fano manifold. We say that $F$ satisfies 
\emph{Gamma Conjecture I} (Conjecture \ref{conj:GammaI}) 
if $A_F$ equals the Gamma class: 
\[
A_F = \Gg_F. 
\] 
More generally (under semisimplicity assumption), we can identify 
a cohomology class $A_i$ such that the corresponding flat section 
has an exponential asymptotics $\sim e^{-u_i/z}$ as $z\to 0$ 
along a fixed sector (of angle bigger than $\pi$) for each 
eigenvalue $u_i$ of $(c_1(F)\star_0)$, $i=1,\dots,N$. 
These classes $A_1,\dots,A_N\in H^\udot(F)$ form a basis 
which we call the \emph{asymptotic basis} of $F$. 
We say that $F$ satisfies \emph{Gamma Conjecture II} 
(Conjecture \ref{conj:conj-g2}) 
if the basis can be written as: 
\[
A_i = \Gg_F \Ch(E_i) 
\]
for a certain exceptional collection $\{E_1,\dots,E_N\}$ of 
the derived category $\D^b_{\rm coh}(F)$. 
Gamma Conjecture I says that the exceptional object $\O_F$ 
corresponds to the biggest real positive eigenvalue $T$ 
of $(c_1(F) \star_0)$. 

The quantum connection has an isomonodromic deformation 
over the cohomology group $H^\udot(F)$. 
By Dubrovin's theory \cite{Dubrovin94, Dubrovin98a}, 
the asymptotic basis $\{A_i\}$ changes by mutation 
\[
(A_1,\dots,A_i, A_{i+1},\dots, A_N) 
\longrightarrow (A_1,\dots, A_{i+1}, A_i - [A_i,A_{i+1}) A_{i+1}, 
\dots, A_N) 
\]
when the eigenvalues $u_i$ and $u_{i+1}$ are interchanged 
(see Figure \ref{fig:braid}). 
Via the HRR formula \eqref{eq:HRR_Gamma} this corresponds 
to a mutation of the exceptional collection $\{E_i\}$. 
The braid group acts on the set of asymptotic bases by mutation 
and we formulate this in terms of a \emph{marked reflection 
system} (MRS) in \S \ref{subsec:MRS}. 
Note that Gamma Conjecture II implies (part (2) of) 
Dubrovin's conjecture \cite{Dubrovin98} (see \S \ref{subsec:GammaII}), 
which says that  
the Stokes matrix $S_{ij} = [A_i, A_j)$ of the quantum connection equals 
the Euler pairing $\chi(E_i,E_j)$. 
While we were writing this paper, we were informed that 
Dubrovin \cite{Dubrovin:Strasbourg} gave a new formulation 
of his conjecture that includes Gamma Conjecture II above. 

\subsection{Gamma Conjectures for Grassmannians} 
In this paper, we establish Gamma conjectures for projective 
spaces and for Grassmannians. 
The Gamma Conjectures for $\P^n$ were implicit but 
essentially shown in the work of Dubrovin \cite{Dubrovin98a}; 
they also follow from mirror symmetry computations 
in \cite{Iritani07,Iritani09,KKP08}. 
In \S \ref{sec:Gamma_P}, we give an elementary proof of 
the following theorem. 

\begin{theorem}[Theorem \ref{thm:Gamma_P}]
Gamma Conjectures I and II hold for the projective space 
$\P=\P^{N-1}$. An asymptotic basis of $\P$ is formed by mutations 
of the Gamma basis $\Gg_\P \Ch(\O(i))$ associated to 
Beilinson's exceptional collection 
$\{\O(i) : 0\le i\le N-1\}$. 
\end{theorem} 

%%%%%%%%%%%%%%%%%%%%%%%% 
%%%%%%%  previous text  
%%%%%%%%%%%%%%%%%%%%%%%%%
% More generally one can deduce the Gamma Conjectures 
% for toric manifolds or toric complete intersections from 
% the mirror theorem \cite{Givental:toric} and 
% the calculation of mirror periods \cite{Iritani07,Iritani09,Iritani11}. 

% \begin{theorem}[see Theorems \ref{thm:toric_GammaI}, 
% \ref{thm:toric_GammaII}, \ref{thm:toriccomp_GammaI} 
% for precise statements] 
% When certain conditions (Conditions \ref{cond:toric_O}, \ref{cond:compint_O}) 
% ensuring Property $\O$ are satisfied, a Fano toric complete intersection 
% satisfies Gamma Conjecture I. For a Fano toric manifold, 
% a $K$-theoretic analogue of Gamma Conjecture II holds. 
% \end{theorem} 

% We also give a sketch of argument in \S \ref{subsec:Lefschetz} showing 
% that Gamma conjecture I should be compatible with quantum Lefschetz principle 
% \cite{Coates-Givental}, i.e.~deducing Gamma Conjecture I for 
% a hypersurface assuming that it holds for an ambient space. 
%%%%%%%%%%%%%%%%%%%%%%%%%%%%
%%%%%%%%%%%%%%%%%%%%%%%

We deduce the Gamma Conjectures for Grassmannians $G(r,N)$ 
from the truth of the Gamma Conjectures for projective spaces. 
The main ingredient in the proof is 
quantum Satake principle \cite{GMa} or abelian/non-abelian
correspondence \cite{BCFK,CFKS}, which says that the quantum 
connection of $G(r,N)$ is the $r$-th wedge of the quantum connection 
of $\P^{N-1}$. 
\begin{theorem}[Theorem \ref{thm:Gamma_G}] 
Gamma Conjectures I and II hold for Grassmannians $\G= G(r,N)$. 
An asymptotic basis of $\G$ is formed by mutations of the Gamma basis 
$\Gg_\G \Ch(S^\nu V^*)$ associated to 
Kapranov's exceptional collection 
$\{ S^\nu V^* : \nu \subset \text{$r\times (N-r)$-box} \}$, 
where $V$ is the tautological bundle and $S^\nu$ is the 
Schur functor. 
\end{theorem} 

The grounds are not sufficient at the moment to claim
the Gamma Conjectures for \emph{all} Fano varieties. 
However the following cases are established. 
Golyshev and Zagier have announced
a proof of Gamma Conjecture I for Fano 3-folds of Picard rank one.  
In a separate paper \cite{GI:mirror}, we use mirror symmetry to 
show Gamma Conjectures for certain toric manifolds or 
toric complete intersections; we will also discuss 
the compatibility of Gamma Conjecture I 
with taking hyperplane sections (quantum Lefschetz). 
Together with the method in the present paper 
(i.e.~compatibility with abelian/non-abelian correspondence), 
the current techniques would allow us to prove Gamma Conjectures for 
a wide class of Fano manifolds. 

\subsection{Limit formula and Apery's irrationality} 
Let $J(t)=J(-\log(t)K_F)$ denote Givental's $J$-function \eqref{eq:J} restricted to 
the anti-canonical line $\C c_1(F)$. This is a cohomology-valued 
solution to the quantum differential equation. 
We show the following  (continuous and discrete) limit formulae: 
\begin{theorem}[Corollary \ref{cor:limit_J}]
Suppose that a Fano manifold $F$ satisfies Property $\O$ and 
Gamma Conjecture I. 
Then the Gamma class of $F$ can be obtained as the limit of the ratio 
of the $J$-function
\[
\Gg_F = \lim_{t\to + \infty} \frac{J(t)}{\langle [\pt], J(t)\rangle}.  
\]
\end{theorem} 
\begin{theorem}[see Theorem \ref{thm:Apery} for precise statements]  
Under the same assumptions as above,  
the primitive part of the Gamma class can be obtained as 
the following discrete limit (assuming the 
limit exists): 
\begin{equation} 
\label{eq:Apery_limit} 
\langle \gamma, \Gg_F\rangle = 
\lim_{n\to \infty} \frac{\langle \gamma, J_{rn} \rangle}
{\langle [\pt], J_{rn}\rangle} 
\end{equation} 
for every $\gamma \in H_{\ldot}(F)$ with $c_1(F) \cap \gamma =0$. 
Here we write 
$J(t) = e^{c_1(F) \log t} \sum_{n=0}^\infty J_{rn} t^{rn}$ 
with $r$ the Fano index. 
\end{theorem} 
Discrete limits similar to \eqref{eq:Apery_limit} were  
studied by Almkvist-van Straten-Zudilin \cite{AvSZ} in 
the context of Calabi--Yau differential equations 
and are called \emph{Apery limits} (or Apery constants). 
Golyshev \cite{Golyshev08a} and Galkin \cite{Galkin:Apery} 
studied the limits \eqref{eq:Apery_limit} for Fano manifolds. 
In fact, these limits are related to famous Apery's proof of the irrationality of 
$\zeta(2)$ and $\zeta(3)$. 
An Apery limit of the Grassmannian $G(2,5)$ 
gives a fast approximation of $\zeta(2)$ 
and an Apery limit of the orthogonal Grassmannian $OGr(5,10)$ 
gives a fast approximation of $\zeta(3)$; 
they prove the irrationality of 
$\zeta(2)$ and $\zeta(3)$. 
It would be extremely interesting to find 
a Fano manifold whose Apery limits give fast approximations 
of other zeta values. 

\subsection{Mirror symmetry} 
Gamma Conjectures are closely related to mirror symmetry 
and relevant phenomena have been observed since its early days. 
The following references serve as motivation to the Gamma conjectures. 
\begin{itemize} 
\item 
The Gamma class of a Calabi--Yau threefold $X$ is given by 
$\Gg_X = 1- \frac{\pi^2}{6}c_2(X) - \zeta(3)c_3(X)$; 
the number $\zeta(3)\chi(X)$ appeared in the computation 
of mirror periods by Candelas et al.~\cite{CDGP}; 
it also appeared in the conifold period formula of 
van Enckevort--van Straten \cite{vEnckevort-vStraten}. 
\item 
Libgober \cite{Lib99} found the (inverse) Gamma class from 
hypergeometric solutions to the Picard--Fuchs equation 
of the mirror, inspired by the work of Hosono et al.~\cite{HKTY}.  
\item 
Kontsevich's homological mirror symmetry 
suggests that the monodromy of the Picard--Fuchs equation of mirrors 
should be related to $\Auteq(\D^b_{\rm coh}(X))$. 
In the related works of Horja \cite{Horja}, 
Borisov--Horja \cite{Borisov-Horja} and Hosono \cite{Hosono}, 
Gamma/hypergeometric series play an important role. 
\item 
In the context of Fano/Landau--Ginzburg mirror symmetry, 
Iritani \cite{Iritani07,Iritani09} and Katzarkov-Kontsevich-Pantev 
\cite{KKP08} introduced a rational or integral structure of 
quantum connection in terms of the Gamma class, by shifting 
the natural integral structure of the Landau--Ginzburg model 
given by Lefschetz thimbles. 
\end{itemize} 
Although not directly related to mirror symmetry, we remark that 
the Gamma class also appears in a recent progress 
\cite{Hori-Romo,HJLM} in physics on the sphere/hemisphere 
partition functions. 
\begin{remark} 
In their studies of a `generalized' Hodge structure 
(\emph{TERP/nc-Hodge} structure), 
Hertling-Sevenheck \cite{Hertling-Sevenheck:nilpotent} 
and Katzarkov-Kontsevich-Pantev \cite{KKP08}  
discussed the compatibility of the Stokes structure with the real/rational structure. 
Our Gamma conjecture can be regarded as 
an explicit Fano-version of 
the following more general conjecture: 
\emph{the $\Gg$-integral structure in quantum cohomology should 
be compatible with the Stokes structure}. 
\end{remark}

\subsection{Plan of the paper} 
In \S \ref{sec:qcoh}, we review quantum cohomology 
and quantum differential equation for Fano manifolds. 
We explain asymptotically exponential flat sections, 
their mutation and Stokes matrices. 
The content in this section is mostly taken from 
Dubrovin's work \cite{Dubrovin94,Dubrovin98a}, 
but we make the following technical refinement: we carefully deal with 
the case where the quantum cohomology 
is semisimple but the Euler multiplication $(E\star_\tau)$ 
has repeated eigenvalues 
(see Propositions \ref{prop:repeated_eigenvalues} and \ref{prop:Stokes}). 
In particular, we show that a semisimple point $\tau$ is 
not a turning point even when $E\star_\tau$ has multiple eigenvalues. 
We need this case since it happens for Grassmannians. 
In \S \ref{sec:GammaI}, we formulate Property $\O$ 
and Gamma Conjecture I. We also prove limit formulae 
for the principal asymptotic class. 
In \S \ref{sec:GammaII}, we formulate Gamma Conjecture II 
and explain a relationship to (original) Dubrovin's conjecture. 
% In \S \ref{sec:toric}, we prove 
% Gamma conjectures for toric complete intersections 
% (under certain conditions) 
% and explain the compatibility of Gamma Conjecture I 
% with quantum Lefschetz principle. 
In \S \ref{sec:Gamma_P}, we prove 
Gamma Conjectures for projective spaces. 
In \S \ref{sec:Gamma_G}, we deduce the Gamma Conjectures 
for Grassmannians from the truth of the Gamma Conjectures 
for projective spaces. Main tools in the proof are 
isomonodromic deformation and quantum Satake principle. 
For this purpose we extend quantum Satake principle 
\cite{GMa} to big quantum cohomology 
(Theorem \ref{thm:wedge_conn}) using abelian/non-abelian 
correspondence \cite{BCFK, CFKS,KimSab}.

\section{Quantum cohomology, quantum connection and solutions.} 
\label{sec:qcoh} 
In this section we discuss background material on quantum cohomology 
and quantum connection of a Fano manifold. Quantum connection 
is defined as a meromorphic flat connection of the trivial cohomology 
bundle over the $z$-plane, with singularities at $z=0$ and $z=\infty$. 
We discuss two fundamental solutions associated to the regular 
singularity ($z=\infty$) and the irregular singularity ($z=0$). 
Under the semisimplicity assumption, we discuss mutations and 
Stokes matrices.  

% \pph {\bf Quantum cohomology.} 
\subsection{Quantum cohomology} 
\label{subsec:qcoh} 
Let $F$ be a Fano manifold, i.e.~ a smooth projective variety 
such that the anticanonical line bundle $\omega_F^{-1} = \det(TF)$ 
is ample. 
Let $H^\udot(F)= H^{\rm even}(F;\C)$ denote the even part of the 
Betti cohomology group over $\C$. 
For $\alpha_1,\alpha_2,\dots, \alpha_n \in H^\udot(F)$, 
let $\Ang{\alpha_1,\alpha_2,\dots,\alpha_n}_{0,n,d}^F$ 
denote the genus-zero $n$ points Gromov--Witten invariant 
of degree $d\in H_2(F;\Z)$, see e.g.~\cite{Man99}. 
Informally speaking, this counts the (virtual) number of rational curves in 
$F$ which intersect the Poincar\'{e} dual cycles of 
$\alpha_1,\dots,\alpha_n$. 
It is a rational number when $\alpha_1,\dots,\alpha_n \in H^\udot(F;\Q)$. 
The quantum product $\alpha_1 \star_\tau \alpha_2\in H^\udot(F)$ 
of two classes $\alpha_1,\alpha_2\in H^\udot(F)$ 
with parameter $\tau \in H^\udot(F)$ 
is given by 
\begin{equation} 
\label{eq:quantumproduct}
 (\alpha_1\star_\tau \alpha_2, \alpha_3)_F 
= \sum_{d\in \Eff(F)} \sum_{n=0}^\infty 
\frac{1}{n!} \Ang{\alpha_1,\alpha_2,
\alpha_3,\tau,\dots,\tau}_{0,3+n,d}^F 
\end{equation} 
where $(\alpha,\beta)_F= \int_F \alpha \cup \beta$ 
is the Poincar\'{e} pairing 
and $\Eff(F) \subset H_2(F;\Z)$ is the set of effective curve classes. 
The quantum product is associative and commutative, 
and recovers the cup product in the limit where 
$\re\left( \int_d \tau\right) \to -\infty$ for all non-zero effective curve classes $d$. 
It is not known if the quantum product $\star_\tau$ 
converges in general, however it does for all the examples in 
this paper (see also Remark \ref{rem:qc_formalseries}). 
The quantum product $\star_\tau$ with $\tau\in H^2(F)$ is 
called the \emph{small} quantum product; for general $\tau\in H^\udot(F)$ 
it is called the \emph{big} quantum product. 

We are particularly interested in the quantum product $\star_0$ 
specialized to $\tau=0$. 
An effective class $d$ contributing to the sum 
\[
\sum_{d\in \Eff(F)} \Ang{\alpha_1,\alpha_2,\alpha_3}_{0,3,d}^F
\]
has to satisfy $\frac{1}{2}\sum_{i=1}^3\deg \alpha_i = 
\dim F + c_1(F) \cdot d$ and there are only finitely many 
such $d$ when $F$ is Fano. Therefore the specialization at $\tau=0$ 
makes sense. 

\begin{remark} 
\label{rem:qc_formalseries}
Writing $\tau = h+ \tau'$ with $h \in H^2(F)$ 
and $\tau' \in \bigoplus_{p\neq 1} H^{2p}(F)$ and using 
the divisor axiom in Gromov--Witten theory, we 
have: 
\[
(\alpha_1\star_\tau\alpha_2, \alpha_3)_F  = 
\sum_{d\in \Eff(F)} \sum_{n=0}^\infty \frac{1}{n!} 
\Ang{\alpha_1,\alpha_2,\alpha_3, \tau' ,\dots,\tau' }_{0,n+3,d}^F 
e^{\int_{d} h}. 
\]
Therefore the quantum product 
\eqref{eq:quantumproduct} makes sense as a formal power series 
in $\tau'$ and the exponentiated $H^2$-variables $e^{h_1},\dots, 
e^{h_r}$, where we write $h = h_1 p_1 + \cdots + h_r p_r$ 
by choosing a nef basis $\{p_1,\dots,p_r\}$ of $H^2(X;\Z)$. 
\end{remark} 
\subsection{Quantum connection} 
\label{subsec:q_conn} 
Following Dubrovin \cite{Dubrovin94, Dubrovin98, 
Dubrovin98a}, we introduce a 
meromorphic flat connection 
associated to the quantum product. 
Consider a trivial vector bundle $H^\udot(F) \times \P^1 
\to \P^1$ over $\P^1$ and fix an 
inhomogeneous co-ordinate $z$ on $\P^1$. 
Define the quantum connection $\nabla$ on the trivial bundle 
by the formula 
\begin{equation}
\label{eq:nabla_tau_zero}
\nabla_{z\partial_z} = z\parfrac{}{z} - \frac{1}{z} (c_1(F) \star_0) 
+ \mu 
\end{equation} 
where $\mu\in \End(H^\udot(F))$ is the grading operator 
defined by $\mu|_{H^{2p}(F)} = (p-\frac{\dim F}{2})\id_{H^{2p}(F)}$ 
($\dim F$ denotes the complex dimension of $F$). 
The connection is smooth away from $\{0,\infty\}$; the 
singularity at $z=\infty$ is regular (or more precisely logarithmic) 
and the singularity at $z=0$ is irregular. 
The quantum connection preserves the Poincar\'{e} pairing 
in the following sense: we have 
\begin{equation} 
\label{eq:nabla_preserves_pairing}
z\parfrac{}{z} (s_1(-z), s_2(z))  =  
((\nabla_{z\partial_z} s_1)(-z), s_2(z)) + 
(s_1(-z), \nabla_{z\partial_z} s_2 (z)) 
\end{equation} 
for $s_1,s_2 \in H^\udot(F) \otimes \C[z,z^{-1}]$. 
Here we need to flip the sign of $z$ for the first 
entry $s_1$. 

What is important in Dubrovin's theory is the fact 
that $\nabla$ admits an 
\emph{isomonodromic deformation} over $H^\udot(F)$. 
Suppose that $\star_\tau$ converges on a region 
$B\subset H^\udot(F)$. Then the above 
connection is extended to a meromorphic flat connection 
on $H^\udot(F) \times (B \times \P^1) 
\to (B \times \P^1)$ as follows: 
\begin{align} 
\label{eq:big_q_conn}
\begin{split}
\nabla_\alpha & = \partial_\alpha + \frac{1}{z} (\alpha\star_\tau)
\qquad \alpha \in H^\udot(F) 
\\ 
\nabla_{z\partial_z} & = z\parfrac{}{z} 
- \frac{1}{z} (E \star_\tau) 
+ \mu 
\end{split} 
\end{align} 
with $(\tau,z)$ a point on the base $B \times \P^1$. 
Here $\partial_\alpha$ denotes 
the directional derivative in the direction of $\alpha$ 
and 
\[
E = c_1(F) + \sum_{i=1}^N 
\left(1-\frac{1}{2}\deg \phi_i\right) \tau^i \phi_i 
\]
is the Euler vector field, 
where we write $\tau = \sum_{i=1}^N \tau^i \phi_i$ 
by choosing a homogeneous basis $\{\phi_1,\dots,\phi_N\}$ 
of $H^\udot(F)$. 
We refer to this extension as the \emph{big} quantum connection. 
The Poincar\'{e} pairing $(\cdot,\cdot)_F$ 
is flat with respect to the big quantum connection $\nabla$, 
i.e.~$\partial_\alpha(s_1(\tau,-z),s_2(\tau,z)) = 
((\nabla_\alpha s_1)(\tau,-z), s_2(\tau,z)) + 
(s_1(\tau,-z),\nabla_\alpha s_2(\tau,z))$. 

\begin{remark} 
\label{rem:anticanonical}
The connection in the $z$-direction can be identified 
with the connection in the anticanonical direction 
after an appropriate rescaling. 
Consider the quantum product $\star_\tau$ 
restricted to the anticanonical line 
$\tau = c_1(F) \log t$, $t\in \C^\times$:  
\begin{equation}
\label{eq:qprod_antican}
(\alpha_1\star_{c_1(F) \log t} \alpha_2, \alpha_3)_F
= \sum_{d\in \Eff(F)} \Ang{\alpha_1,\alpha_2,\alpha_3}_{0,3,d} 
t^{c_1(F)\cdot d}. 
\end{equation} 
This is a polynomial in $t$ since $F$ is Fano, and 
coincides with $\star_0$ when $t=1$. 
It can be recovered from the product $\star_0$ by the 
formula: 
\begin{equation} 
\label{eq:qprod_scaling} 
(\alpha \star_{c_1(F) \log t})= t^{\deg \alpha/2} 
t^{-\mu} (\alpha \star_0) t^\mu. 
\end{equation} 
The quantum connection restricted to the anticanonical line 
and $z=1$ is 
\[
\nabla_{c_1(F)}\Big|_{z=1} = t\parfrac{}{t} +  
(c_1(F) \star_{c_1(F) \log t}).  
\]
On the other hand we have
\[
z^\mu \Big[ \nabla_{z\partial_z} \Big]_{\tau=0} z^{-\mu}
= z\parfrac{}{z} - (c_1(F)\star_{-c_1(F)\log z})  
\]
Therefore, the connections $\nabla_{c_1(F)}|_{z=1}$ 
and $\nabla_{z\partial_z}|_{\tau=0}$ 
are gauge equivalent via $z^\mu$ 
under the change of variables $t= z^{-1}$. 
\end{remark}

\subsection{Canonical fundamental solution around $z=\infty$} 
\label{subsec:fundsol_regular}
We consider the connection $\nabla$ 
\eqref{eq:nabla_tau_zero} defined by the quantum product 
$\star_0$ at $\tau=0$. We have a (well-known) canonical 
fundamental solution $S(z) z^{-\mu} z^\rho$ for $\nabla$ 
associated to the regular singular point $z=\infty$. 

\begin{proposition}
\label{prop:fundsol} 
There exists a unique holomorphic function 
$S \colon \P^1\setminus \{0\} 
\to \End(H^\udot(F))$ with $S(\infty) = \id_{H^\udot(F)}$ 
such that 
\begin{align*} 
&\nabla ( S(z)z^{-\mu}z^{\rho} \alpha) = 0
\qquad 
\text{for all $\alpha \in H^\udot(F)$}; \\ 
&\text{$T(z)= z^\mu S(z) z^{-\mu}$ is regular at $z=\infty$ 
and $T(\infty) = \id_{H^\udot(F)}$},   
\end{align*} 
where $\rho= (c_1(F)\cup)\in \End(H^\udot(F))$ and we define 
$z^{-\mu} = \exp(-\mu\log z)$, 
$z^{\rho} = \exp(\rho \log z)$. 
Moreover we have 
\[
(S(-z)\alpha, S(z) \beta )_F= (\alpha,\beta)_F 
\qquad \alpha,\beta\in H^\udot(F). 
\] 
\end{proposition} 
\begin{proof} 
The endomorphism $S(z)$ is a gauge transformation giving 
the Levelt normal form (see e.g.~\cite[Exercise 2.20]{Sabbah:Frobenius_manifold}) 
of the connection $\nabla$ near $z=\infty$. 
A similar fundamental solution was given by Dubrovin for 
a general Frobenius manifold \cite[Lemma 2.5, 2.6]{Dubrovin98a}; 
we also have an explicit formula in terms of Gromov-Witten invariants  
(see Remark \ref{rem:S_general_tau} below). 
Note that $S(z)$ in our case satisfies the additional properties  
that $T(z) = z^\mu S(z) z^{-\mu}$ is regular at $z=\infty$ and that 
$T(\infty) = \id$, which ensure the uniqueness of $S$. 
We give a construction of $S(z)$ to verify these points. 

Consider the equivalent differential equation 
$\nabla (z^{-\mu} T(z) z^\rho \alpha)=0$ 
for $T(z) = z^\mu S(z) z^{-\mu}$  
with the initial condition $T(\infty) = \id$.  
The differential equation for $T$ reads: 
\[
z\parfrac{}{z} T(z) - \frac{1}{z} 
z^{\mu} (c_1(F)\star_0) z^{-\mu} T(z) + T(z) \rho =0.   
\]
Expand: 
\begin{align*} 
T(z) & = \id + T_1 z^{-1}+ T_2 z^{-2} + T_3 z^{-3} + 
\cdots  \\ 
(c_1(F)\star_0) & = G_0+G_1 + G_2 +  
\cdots \quad \text{(finite sum)} 
\end{align*} 
where $G_k\in \End(H^\udot(F))$ is an endomorphism of
degree $1-k$, i.e.~$z^{\mu} G_k z^{-\mu} = z^{1-k} G_k$ 
and $G_0 = c_1(F) \cup = \rho$. 
The above equation 
is equivalent to the system of equations:  
\begin{align*} 
0 & = \rho - \rho  \\ 
0 & =T_1 + G_1 + [\rho, T_1]  \\ 
&\ \  \vdots  \\
0 & = m T_m + G_m  + G_{m-1} T_1 + \cdots + G_1 T_{m-1} 
+ [\rho, T_m].  
\end{align*} 
These equations can be solved recursively for $T_1,T_2,T_3,\dots$ 
because the map $X \mapsto m X + [\rho, X]$ is invertible 
(since $\rho$ is nilpotent). 
One can easily show the convergence of $T(z)$. 

By construction, $T_k$ is an endomorphism of degree $\ge (1-k)$ 
and hence $z^{-k} z^{-\mu} T_k z^\mu$ contains only negative 
powers in $z$ for $k\ge 1$. Therefore 
$S(z) = z^{-\mu} T(z) z^\mu$ is regular at $z=\infty$ 
and satisfies $S(\infty) = \id$. 

% the following text is for journal version 
% The property $(S(-z)\alpha, S(z) \beta) = (\alpha,\beta)$ is 
% well-known, see \cite[\S 1]{Givental:elliptic}, 
% \cite[Proposition 2.4]{Iritani09}. 
Finally we show $(S(-z)\alpha,S(z) \beta)_F = (\alpha,\beta)_F$. 
We claim that 
\[
\left(S(-z) (e^{\pi\iu} z)^{-\mu} (e^{\pi\iu} z)^\rho\alpha, 
S(z) z^{-\mu} z^\rho \beta \right)_F
= (e^{-\pi\iu\mu} e^{\pi\iu\rho} \alpha, \beta)_F. 
\]
Because $\nabla$ preserves the Poincar\'{e} pairing 
\eqref{eq:nabla_preserves_pairing}, the left-hand side 
is independent of $z$. 
On the other hand, the left-hand side equals 
\[
\left(e^{-\pi\iu \mu} T(-z) (e^{\pi\iu}z)^\rho \alpha, T(z) z^\rho \beta\right)_F 
\]
and can be expanded in $\C[\![z^{-1}]\!][\log z]$. 
The constant term of 
the Taylor expansion in $z^{-1}$ and $\log z$ 
equals the right-hand side. The claim follows. 
Replacing $\alpha, \beta$ with $(e^{\pi\iu} z)^{-\rho}
(e^{\pi\iu}z)^{\mu}\alpha$ and $z^{-\rho} z^\mu \beta$, 
we arrive at the conclusion. 
\end{proof} 

\begin{remark}[\cite{Dubrovin94, Givental:elliptic, Pan98, Iritani09}] 
\label{rem:S_general_tau}
The fundamental solution $S(z)$ is given by 
descendant Gromov--Witten invariants. 
Let $\psi$ denote the first Chern class of the 
universal cotangent line bundle over $\overline{M}_{0,2}(F,d)$ 
at the first marking. Then we have: 
\[
(S(z) \alpha, \beta)_F 
= (\alpha, \beta)_F  
+ 
\sum_{m\ge 0} 
\frac{1}{z^{m+1}} 
\sum_{d\in \Eff(F)\setminus \{ 0 \}} 
(-1)^{m+1} 
\Ang{\alpha \psi^{m}, \beta}_{0,2,d}^F.  
\]
Here again the summation in $d$ is finite (for a fixed $m$) 
because $F$ is Fano. 
A similar fundamental solution exists 
for the isomonodromic deformation of $\nabla$ 
associated to the big quantum cohomology.  
The big quantum connection $\nabla$ over $H^\udot(F)\times \P^1$ 
admits a fundamental solution of 
the form $S(\tau,z) z^{-\mu} z^\rho$ 
extending the one in Proposition \ref{prop:fundsol} 
such that 
\[
\nabla (S(\tau,z) z^{-\mu} z^\rho \alpha)=0, \quad 
S(\tau,\infty) = \id, \quad 
(S(\tau,-z) \alpha, S(\tau,z) \beta)_F = (\alpha,\beta)_F. 
\]
Here $z^\mu S(\tau,z) z^{-\mu}$ is not necessarily 
regular at $z=\infty$ for $\tau\notin H^2(F)$ 
(cf.~Lemma \ref{lem:S_shift}). 
\end{remark}

\subsection{$UV$-system: semisimple case} 
\label{subsec:UV}
Suppose that the quantum product $\star_\tau$ is convergent 
and semisimple at some $\tau \in H^\udot(F)$. 
The semisimplicity means that the algebra $(H^\udot(F), \star_{\tau})$ 
is isomorphic to a direct sum of $\C$ as a ring. 
Let $\psi_1,\dots,\psi_N$ denote the idempotent basis 
of $H^\udot(F)$ such that 
\[
\psi_i \star_\tau \psi_j = \delta_{i,j} \psi_i 
\]
where $N = \dim H^\udot(F)$. 
Let $\Psi_i := \psi_i/\sqrt{(\psi_i,\psi_i)_F}$, $i=1,\dots,N$ 
be the \emph{normalized} idempotents. 
They form an orthonormal basis and are unique up to sign. 
We write 
\[
\Psi = \begin{pmatrix} \Psi_1, \dots, \Psi_N \end{pmatrix} 
\]
for the matrix with the column vectors $\Psi_1,\dots,\Psi_N$. 
We regard $\Psi$ as a homomorphism $\C^N \to H^\udot(F)$. 
Let $u_1,\dots,u_N$ be the eigenvalues of $(E\star_\tau)$ 
given by $E\star_\tau \Psi_i = u_i \Psi_i$. 
Define $U$ to be the diagonal matrix with entries $u_1,\dots,u_N$: 
\begin{equation}
\label{eq:U} 
U = \begin{pmatrix} 
u_1 &  &  &  \\
 & u_2 &  &  \\ 
 & & \ddots & \\
&  & & u_N
\end{pmatrix}.  
\end{equation} 
We allow $E\star_\tau$ (or $U$) to have repeated eigenvalues. 
By the constant gauge transformation $\Psi$, the quantum connection 
$\nabla_{z\partial_z}$ \eqref{eq:big_q_conn} is transformed 
to the connection 
\begin{equation}
\label{eq:UV}
\Psi^* \nabla_{z\partial_z} = z\parfrac{}{z} - \frac{1}{z} U + V 
\end{equation} 
with $V = \Psi^{-1} \mu \Psi$. 
We call this the \emph{$UV$-system}, cf.~\cite[Lecture 3]{Dubrovin94}. 
Notice that the semisimplicity is an open condition: when $\tau$ 
varies\footnote{If quantum cohomology is not known to converge 
except at $\tau$, we can work with the formal neighbourhood of 
$\tau$ in $H^\udot(F)$ in the following discussion.}, 
the matrices $\Psi$ and $U$ 
depends analytically on $\tau$ as far as $\star_\tau$ is 
semisimple. 
Moreover $\tau\mapsto (u_1,\dots,u_N)$ 
gives a local co-ordinate system on $H^\udot(F)$ 
and one has $\psi_i = \parfrac{\tau}{u_i}$ 
(see \cite[Lecture 3]{Dubrovin98a}). 
The $UV$-system is extended in the $H^\udot(F)$-direction 
as follows \cite[Lemma 3.2]{Dubrovin98a}: 
\begin{equation}
\label{eq:UV_extension}
\Psi^* \nabla_{\partial_{u_i}} = 
\parfrac{}{u_i} + \frac{1}{z} E_i + V_i 
\end{equation}
where $E_i = \diag[0,\dots,0,\overset{\text{$i$th}}{1},0,\dots,0]$ 
and $V_i = \Psi^{-1} \partial_{u_i} \Psi$.  
\begin{lemma} 
\label{lem:V} 
The matrix $V = (V_{ij})$ is anti-symmetric $V_{ij} = - V_{ji}$. 
Moreover, $V_{ij} = 0$ whenever $u_i = u_j$. 
\end{lemma} 
\begin{proof} 
The anti-symmetricity of $V$ follows from the fact that 
$\mu$ is skew-adjoint: $(\mu \alpha, \beta)_F= - (\alpha, \mu \beta)_F$ 
and that $\Psi_1,\dots,\Psi_N$ are orthonormal. 
To see the latter statement, we use the isomonodromic deformation 
\eqref{eq:UV}--\eqref{eq:UV_extension}.   
The flatness of $\Psi^*\nabla$ implies: 
\begin{equation} 
\label{eq:EiV}
[E_i,V] = [V_i, U] 
\end{equation} 
and it follows that $V_{ij}= (u_j-u_i) (V_i)_{ij}$ if $i\neq j$. 
Therefore $V_{ij} =0$ if $i\neq j$ and $u_i= u_j$. 
If $i=j$, $V_{ij}=0$ by the anti-symmetricity. 
\end{proof} 

\subsection{Asymptotically exponential flat sections} 
\label{subsec:asymptotic_solution}
Under semisimplicity assumption, we will construct a basis of flat sections 
for the quantum connection near the irregular singular point $z=0$ 
which have exponential asymptotics $\sim e^{-u_i/z}$ as $z\to 0$ 
along an angular sector. 

The so-called \emph{Hukuhara-Turrittin theorem}  
\cite{Trjitzinsky,Hukuhara_II_III,Malmquist_I_II_III,Turrittin}, 
\cite[Theorem 19.1]{Wasow}, 
\cite[II, 5.d]{Sabbah:Frobenius_manifold}, 
\cite[\S 8]{Hertling-Sevenheck:nilpotent}   
says that $\nabla_{z\partial_z}$ admits, after a change\footnote
{In view of mirror symmetry, we might hope that 
the order $r$ of ramification equals $1$ for a wide class of Fano manifolds; 
this is referred to as $\nabla$ being of 
``exponential type'' in \cite[Definition 2.12]{KKP08}.} 
of variables $z = w^r$ with $r\in \Z_{>0}$, 
a fundamental matrix solution around $z=0$ of the form:  
\begin{equation*}
% %\label{eq:Hukuhara-Turrittin}
P(w) w^C e^{\Lambda(w^{-1})} 
\end{equation*}
where $\Lambda(w^{-1})$ is a diagonal 
matrix whose entries are polynomials in $w^{-1}$, 
$C$ is a constant matrix, and $P(w)$ is an invertible matrix-valued 
function having an asymptotic expansion 
$P(w) \sim P_0 + P_1 w + P_2 w^2 + \cdots$ 
as $|w| \to 0$ in an angular sector. 

When the eigenvalues $u_1,\dots,u_N$ of $E\star_\tau$ are 
pairwise distinct, the fundamental solution takes 
a simpler form: we have $r=1$, $z=w$ and 
$\Lambda(w^{-1}) = -U z^{-1}$ with $U = \diag[u_1,\dots,u_N]$. 
This case has been studied by many people, including 
Wasow \cite[Theorem 12.3]{Wasow}, 
Balser-Jurkat-Lutz \cite[Theorem A, Proposition 7]{BJL79}, 
\cite{BJL81} (for general irregular connections) 
and Dubrovin \cite[Lectures 4, 5]{Dubrovin98a} 
(in the context of Frobenius manifolds). 
Bridgeland--Toledano-Laredo \cite{BTL} studied 
the case where $E\star_\tau$ is semisimple but has repeated 
eigenvalues. Under a certain condition \cite[\S 8 (F)]{BTL}, 
they showed that one can take $r=1$ and $\Lambda(w^{-1}) = - U/z$; 
their condition (F) is ensured, in our setting, by Lemma \ref{lem:V}. 
We extend the results in \cite{BJL79,BJL81,Dubrovin98a, BTL} to 
isomonodromic deformation (where $u_1,\dots,u_N$ are not necessarily
distinct) 
and show that both formal and actual solutions depend
analytically on $\tau$. 

We say that a phase $\phi\in\R$ (or a direction $e^{\iu\phi}\in S^1$) 
is \emph{admissible}\footnote
{Our admissible direction is perpendicular to the one 
in \cite[Definition 4.2]{Dubrovin98a}.}  for a multiset 
$\{u_1,\dots,u_N\} \subset \C$ if $e^{-\iu\phi}(u_i-u_j) 
\notin \R_{>0}$ holds for every $(i,j)$, i.e.~
$e^{\iu\phi}$ is not 
parallel to any non-zero difference $u_i-u_j$.

\begin{proposition}  
\label{prop:repeated_eigenvalues}
Assume that the quantum product $\star_{\tau}$ 
is analytic and semisimple in a neighbourhood $B$ of 
$\tau_0\in H^\udot(F)$. 
Consider the big quantum connection $\nabla$ \eqref{eq:big_q_conn} 
over $B \times \P^1$. 
Let $\phi\in \R$ be an admissible phase for the spectrum 
$\{u_{1,0},\dots,u_{N,0}\}$ 
of $(E\star_{\tau_0})$.  
Then, shrinking $B$ if necessary, we have an analytic fundamental 
solution $Y_\tau(z) = (y_1(\tau,z),\dots,y_N(\tau,z))$ for $\nabla$ 
and $\epsilon>0$ such that 
\begin{equation} 
\label{eq:Y_asymp} 
Y_\tau(z) e^{U/z} \to  \Psi \qquad
\text{as $z\to 0$ in the sector $|\arg z -\phi| < \frac{\pi}{2} + \epsilon$,}
\end{equation} 
where $U = \diag[u_1,\dots,u_N]$ 
and $\Psi=(\Psi_1,\dots,\Psi_N)$ are as in \S\ref{subsec:UV}. 
Moreover, we have:  
\begin{enumerate} 
\item A fundamental solution $Y_\tau(z)$ satisfying this 
asymptotic condition is unique; 
we call it the \emph{asymptotically exponential fundamental solution 
associated to $e^{\iu\phi}$}. 

\item Let $Y_\tau^-(z)=(\my_1(\tau,z),\dots,\my_N(\tau,z))$ 
be the fundamental solution associated to 
$-e^{\iu\phi}$. Then we have $(y_i(\tau,-z), y_j(\tau,z))_F= \delta_{ij}$. 
\end{enumerate} 
\end{proposition} 

Note that the fundamental solution $Y_\tau(z)$ depends on 
the choice of sign and ordering of the normalized idempotents 
$\Psi_1,\dots,\Psi_N$. 

We construct $Y_\tau$ using Laplace transformation. 
Under the formal substitution 
$z^{-1} \to \partial_\lambda$, $\partial_{z^{-1}} \to -\lambda$, 
the differential equation $\nabla (z^{-1}y(\tau,z)) = 0$ is transformed 
to the equations: 
\begin{align}
\label{eq:2nd_conn}
\begin{split}  
\hnabla_\alpha \hy(\tau,\lambda) & := \left(\partial_\alpha - 
(\alpha\star_\tau)(\lambda - E\star_\tau)^{-1} \right) \hy(\tau,\lambda) = 0, \\  
\hnabla_{\partial_\lambda} \hy(\tau,\lambda) & := 
\left(\partial_\lambda + (\lambda - E \star_\tau)^{-1} \mu \right) 
\hy(\tau,\lambda) =0, 
\end{split}
\end{align} 
where $\alpha \in H^\udot(F)$. 
See e.g.~\cite[Lecture 5]{Dubrovin98a}, \cite{Man99}. 
The connection $\hnabla$ is flat, and has only logarithmic 
singularities at $\lambda = u_1,\dots,u_N, \infty$ under 
the semisimplicity assumption. In fact, by the gauge 
transformation by $\Psi$, $\hnabla$ is transformed into 
the following form\footnote{Here we used 
$\sum_{i=1}^N V_i = \Psi^{-1} 
(\partial_{u_1} + \cdots + \partial_{u_N}) \Psi = 0$.
}: 
\[
\Psi^*\hnabla = d + \sum_{j=1}^N 
\left(\frac{E_j V}{\lambda -u_j} - V_j \right) 
d (\lambda -u_j) 
\]
where $E_j$, $V$, $V_j$ are as in \S \ref{subsec:UV}. 
This has logarithmic singularities along the normal crossing divisors 
$\prod_{j=1}^N (\lambda - u_j) =0$ in $B\times \C_\lambda$. 
\begin{lemma} 
\label{lem:hnabla_flat} 
Let $(\tau_0,\lambda_0) \in B\times \C_\lambda$ be a point 
on the singularity of $\hnabla$. 
\begin{enumerate} 
\item  
For every divisor $\{\lambda = u_i\}$ passing through $(\tau_0,\lambda_0)$, 
there exists a $\hnabla$-flat section $\hy_i(\tau,\lambda)$ 
which is holomorphic near $(\tau,\lambda) = (\tau_0,\lambda_0)$ 
and $\hy_i(\tau,u_i(\tau)) = \Psi_i$. 
\item 
Let $M$ be the monodromy transformation with respect to 
an anti-clockwise loop around $\lambda = \lambda_0$ in 
the $\lambda$-plane. For a $\hnabla$-flat section $\hy$, 
$M \hy - \hy$ is a linear combination of the flat sections 
$\hy_i$ in (1) such that $\{\lambda = u_i\}$ passes through 
$(\tau_0,\lambda_0)$. 
\end{enumerate}
\end{lemma} 
\begin{proof} 
(1) It suffices to find a $\Psi^*\hnabla$-flat section $\hs_i(\tau,\lambda)$ 
with the property $\hs_i(\tau,u_i) = e_i$. 
Since we do not assume that $u_1,\dots,u_N$ are pairwise 
distinct, we can have several singularity divisors 
passing through the point $(\tau_0,\lambda_0)$; 
let $\{\lambda = u_j\}$ be one of them. 
The residue $R_j = E_jV|_{\tau=\tau_0}$ 
of $\Psi^*\hnabla$ at $(\tau_0, \lambda_0)$ 
\emph{along} the divisor $\{\lambda = u_j\}$ 
is nilpotent by Lemma \ref{lem:V}. 
Thus, in a neighbourhood of 
$(\tau,\lambda) = (\tau_0, \lambda_0)$, we have a fundamental solution 
for $\hnabla$ of the form 
(see, e.g.~\cite[Theorem 2, Remark 2]{Yoshida-Takano}): 
\begin{equation} 
\label{eq:fundsol_hnabla}
U(\tau,\lambda) 
\exp \left( 
- \sum_{j : u_{j}(\tau_0)= \lambda_0} R_j \log (\lambda - u_j)
\right) 
\end{equation} 
where $U(\tau,\lambda)$ is a matrix-valued 
holomorphic function defined near $(\tau,\lambda)= (\tau_0, \lambda_0)$ 
such that $U(\tau_0,\lambda_0)= \id$. 
We define a $\Psi^*\hnabla$-flat section $\hs_i(\tau,\lambda)$ by 
applying the above fundamental solution to the $i$th basis vector $e_i$. 
Note that $R_j e_i = 0$ 
whenever $u_j(\tau_0) = u_i(\tau_0) =\lambda_0$ by Lemma \ref{lem:V}, 
therefore $\hs_i(\tau,\lambda ) = U(\tau,\lambda) e_i$ 
is holomorphic near $(\tau_0,\lambda_0)$.  
We claim that $\hs_i(\tau,u_i) = e_i$. 
By definition we have $\hs_i(\tau_0, \lambda_0)= e_i$. 
On the other hand, the residual connection $\hnabla^{(i)}$ 
on the divisor $\{\lambda =u_i\}$ induced from $\Psi^*\hnabla$ reads: 
\[
\hnabla^{(i)} = d + \sum_{j :j\neq i} 
\left( \frac{E_j V}{u_i - u_j} - V_j \right) d(u_i-u_j). 
\]
Using the formula \eqref{eq:EiV}, one finds that 
$E_j V e_i = (u_i-u_j)V_je_i$ for $i\neq j$; hence 
$\hnabla^{(i)} e_i =0$. Since $\hs_i(\tau,u_i)$ is flat 
with respect to $\hnabla^{(i)}$, the claim follows. 

(2) This follows from the form \eqref{eq:fundsol_hnabla} 
of the fundamental solution 
and the fact that $\im R_j \subset \C e_j$. 
\end{proof}

\begin{proof}[Proof of Proposition \ref{prop:repeated_eigenvalues}]
We closely follow the method of Balser-Jurkat-Lutz \cite[Theorem 2]{BJL81}
and Bridgeland--Toledano-Laredo \cite[\S8.4]{BTL}. 
Using the flat section $\hy_i(\tau,\lambda)$ from 
Lemma \ref{lem:hnabla_flat}, we define 
\begin{equation}
\label{eq:yi_Laplace}
y_i(\tau,z) = 
\frac{1}{z}
\int_{u_i+ \R_{\ge 0} e^{\iu\phi}} \hy_i(\tau,\lambda) e^{-\lambda/z} d\lambda.
\end{equation} 
Shrinking $B$ if necessary, we may assume that 
$e^{-\iu\phi}(u_i(\tau) - u_j(\tau))\notin \R$ 
for all $\tau\in B$ and for all $(i,j)$ with $u_{i,0} \neq u_{j,0}$.  
Then the flat section $\hy_i(\tau,\lambda)$ can be analytically 
continued along the contour $u_i + \R_{\ge 0} e^{\iu\phi}$ 
in the $\lambda$-plane when $\tau \in B$. 
Indeed, by assumption, the contour $u_i + \R_{\ge 0} e^{\iu\phi}$ 
can only contain singular points $u_j$ (of $\hnabla$) 
such that $u_{i,0} = u_{j,0}$; 
but we know from Lemma \ref{lem:hnabla_flat} that 
$\hy_i(\tau,\lambda)$ is holomorphic in a neighbourhood of 
$(\tau_0,u_{i,0})$
and thus it is regular at $(\tau,u_j)$ whenever $u_{i,0} = u_{j,0}$. 
Because $\hnabla$ is regular singular at $\lambda = \infty$, 
$\hy_i(\tau,\lambda)$ grows at most polynomially 
as $\lambda \to \infty$. Thus the integral \eqref{eq:yi_Laplace} converges 
if $|\arg z - \phi|<\frac{\pi}{2}$; by changing the slope of the contour a little,  
we can analytically continue $y_i(\tau,z)$ to a bigger sector 
$|\arg z -\phi|<\frac{\pi}{2}+ \epsilon$ with $\epsilon>0$ 
sufficiently small. 
By an elementary calculation using integration by parts, we can show 
that $y_i(\tau,z)$ is $\nabla$-flat, where we need the fact that 
$\hy_i(\tau,u_i)$ is a $u_i$-eigenvector of $E\star_\tau$, 
see~\cite[\S 8.4]{BTL}. 
Watson's lemma \cite[6.2.2]{Ablowitz-Fokas} shows that 
$y_i(\tau,z) \to \Psi_i$ as $z\to 0$ in the sector 
$|\arg z - \phi|<\frac{\pi}{2} + \epsilon$. 

The uniqueness of $Y_\tau$ 
follows from the fact that the angle of the sector 
is bigger than $\pi$ \cite[Remark 1.4]{BJL79}. 
Indeed, suppose that we have two solutions $Y_1$, $Y_2$ subject to  
\eqref{eq:Y_asymp}.  Then there exists a constant matrix $C$ 
such that $Y_1 = Y_2 C$. We have $e^{-U/z} C e^{U/z}\to \id$ 
as $z\to 0$ in the sector. Since the angle of the sector 
is bigger than $\pi$, it happens only when $C=\id$. 

Finally we show Part (2). We omit $\tau$ from the notation. 
Since $\my_i(z)$ and $y_j(z)$ are flat, the pairing 
$(\my_i(-z), y_j(z))_F$ does not depend on $z$ 
by \eqref{eq:nabla_preserves_pairing}. 
By the asymptotic condition, we have 
$e^{(u_j-u_i)/z}(\my_i(-z), y_j(z))_F 
= (e^{-u_i/z} \my_i(-z), e^{u_j/z} y_j(z))_F 
\to (\Psi_i, \Psi_j)_F = \delta_{ij}$ 
as $|z|\to 0$ in the sector $|\arg z - \phi|<\frac{\pi}{2}+\epsilon$. 
Thus $(\my_i(-z),y_j(z))_F = \delta_{ij}$ if $u_i = u_j$. 
If $u_i \neq u_j$, we have $(\my_i(-z), y_j(z))_F=0$ 
since the angle of the sector is bigger than $\pi$. 
\end{proof}

\begin{remark} 
Applying Watson's lemma to \eqref{eq:yi_Laplace}, we obtain 
the asymptotic expansion
\[
Y_\tau(z) e^{U/z} \sim \Psi (\id + R_1 z + R_2 z^2+ \cdots) 
\]
as $z\to 0$ in the sector $|\arg z - \phi|<\frac{\pi}{2}+ \epsilon$. 
The right-hand side is called a formal solution which is typically 
divergent. 
The existence of a formal solution in the case 
where $u_1,\dots,u_N$ are not distinct was also remarked 
by Teleman \cite[Theorem 8.15]{Tel10}. Our construction 
shows that each $R_k$ depends analytically on $\tau$, which 
is not clear from the standard recursive construction of a formal solution. 
In other words, a semisimple point of a Frobenius manifold 
is never a turning point. 
\end{remark} 

\begin{remark} 
\label{rem:bigger_sector} 
Each flat section $y_i(\tau, z)$ in \eqref{eq:yi_Laplace}
can have the asymptotic expansion $y_i(\tau, z) \sim e^{-u_i/z} \Psi_i$ 
in a bigger sector. 
Set $\Sigma_i = \{(u_j-u_i)/|u_j-u_i| : u_j \neq u_i \}$ 
and let $\{e^{\iu\theta} : \xi_1 < \theta< \xi_2\}$ be the connected 
component of $S^1 \setminus \Sigma_i$ which contains the 
admissible direction $e^{\iu\phi}$. 
By construction, the flat section 
$y_i(\tau,z)$ has the asymptotic expansion in the sector 
$\xi_1-\frac{\pi}{2} < \arg z <\xi_2 + \frac{\pi}{2}$ 
as in Figure \ref{fig:sector_from_u_i}. 
(The sector here can be bigger than $2\pi$.)
\end{remark} 

\begin{figure}[htbp]
\begin{center} 
\includegraphics[bb=176 535 476 720]{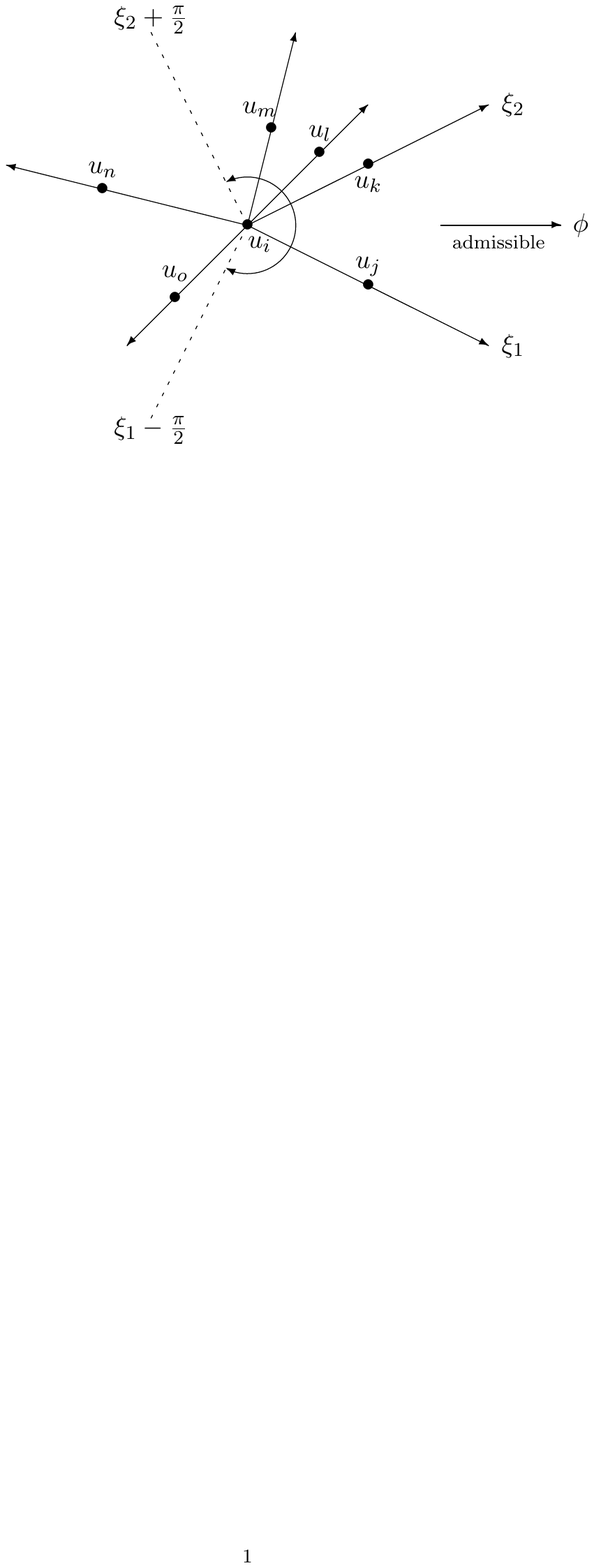}
\end{center} 
\caption{The flat section $y_i(\tau,z)$ 
has the asymptotics $y_i(\tau,z) \sim e^{-u_i/z} \Psi_i$ in the 
sector $(\xi_1-\frac{\pi}{2}, \xi_2+ \frac{\pi}{2})$.}
\label{fig:sector_from_u_i} 
\end{figure}

\subsection{Mutation and Stokes matrix} 
\label{subsec:mutation}
In this section, we discuss mutation of flat sections 
and Stokes matrices. The braid group action on 
the irregular monodromy data $(Y_\tau,S)$ via mutation 
was discussed by Dubrovin \cite[Lecture 4]{Dubrovin98a}. 
We use the idea of Balser-Jurkat-Lutz \cite{BJL81} expressing 
Stokes data in terms of monodromy of the Laplace-dual 
connection $\hnabla$ and extend the result of Dubrovin 
to the case where some of $u_1,\dots,u_N$ may coincide. 
The results here lead us to the formulation of 
a marked reflection system in \S\ref{subsec:MRS}-\ref{subsec:MRS_mutation}, 
\S\ref{subsec:MRS_Fano}.

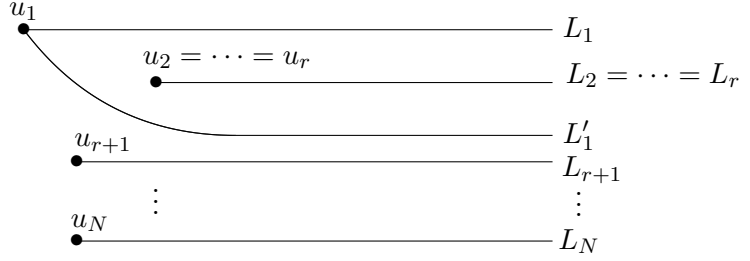
\begin{figure}[htbp]
\begin{center} 
\begin{picture}(300,85)(0,-20)
\put(100,60){\makebox(0,0){$\bullet$}} 
\put(100,67){\makebox(0,0){$u_1$}} 
\put(150,40){\makebox(0,0){$\bullet$}} 
\put(145,47){{$u_2=\cdots = u_r$}} 
\put(120,10){\makebox(0,0){$\bullet$}} 
\put(130,17){\makebox(0,0){$u_{r+1}$}}
\put(150,-2){\makebox(0,0){$\vdots$}} 
\put(310,-3){\makebox(0,0){$\vdots$}} 
\put(125,-13){\makebox(0,0){$u_N$}} 
\put(120,-20){\makebox(0,0){$\bullet$}} 
\put(120,-20){\line(1,0){180}} 
\put(310,-20){\makebox(0,0){$L_N$}} 
\put(100,60){\line(1,0){200}}
\put(150,40){\line(1,0){150}} 
\put(310,60){\makebox(0,0){$L_1$}} 
\put(305,40){$L_2 = \cdots =L_r$} 
\put(310,20){\makebox(0,0){$L_1'$}}
\put(180,20){\line(1,0){120}}
\qbezier(100,60)(130,20)(180,20)
\put(120,10){\line(1,0){180}} 
\put(315,7){\makebox(0,0){$L_{r+1}$}}
\end{picture} 
\end{center} 
\caption{Right mutation of $L_1$ (where $e^{\iu\phi}=1$ is admissible)} 
\label{fig:mutation} 
\end{figure}
Recall that we constructed the asymptotically exponential flat section 
$y_i(\tau,z)$ as the Laplace transform \eqref{eq:yi_Laplace} 
of the $\hnabla$-flat section $\hy_i(\tau,\lambda)$ 
over a straight half-line $L_i =u_i +\R_{\ge 0} e^{\iu\phi}$. 
We study the change of flat sections under a change of integration paths. 
To illustrate, we consider the change of paths from $L_1$ to $L_1'$ 
depicted in Figure \ref{fig:mutation}. 
In the passage from $L_1$ to $L_1'$, the path is assumed to 
cross only one eigenvalue $u_2=\cdots = u_{r}$ of multiplicity $r-1$. 
We use the straight paths $L_1,\dots, L_N$ as 
branch cuts for $\hnabla$-flat sections and regard 
$\hy_i(\lambda)$ as a single-valued analytic function 
on $\C\setminus \bigcup_{j\neq i} L_j$. 
Let $M$ denote the anti-clockwise monodromy transformation 
around $\lambda =u_2$ acting on the space of $\hnabla$-flat sections. 
% The description of the fundamental solution in \eqref{eq:fundsol_hnabla} 
% shows that $\im(M-\id)$ is spanned by $\hy_2,\dots,\hy_r$. 
% Hence the monodromy transform of $\hy_1$ can be written as: 
By Lemma \ref{lem:hnabla_flat} (2), the monodromy transform of $\hy_1$ 
can be written as: 
\[
M \hy_1 = \hy_1 - c_{12} \hy_2 - \cdots - c_{1r} \hy_r 
\]
for some coefficients $c_{12},\dots,c_{1r}\in \C$. 
From this it follows that the flat section $y_1'(z)$ defined 
by the integral over $L_1'$ is given by: 
\begin{equation} 
\label{eq:y1_mutation} 
y_1'(z) = y_1(z) - c_{12} y_2(z) - \cdots - c_{1r} y_r(z). 
\end{equation}
We call the flat section $y_1'(z)$ (or the path $L_1'$) 
the \emph{right mutation} of $y_1(z)$ (resp.~$L_1$) 
with respect to $u_2 = \cdots = u_r$. 
The \emph{left mutation} of a flat section or a path 
is the inverse operation: see Figure \ref{fig:left_mutation}. 
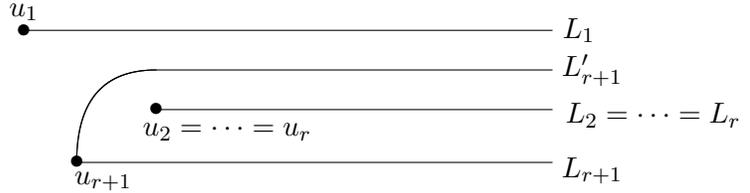
\begin{figure}[htbp]
\begin{center} 
\begin{picture}(300,70)
\put(100,60){\makebox(0,0){$\bullet$}} 
\put(100,67){\makebox(0,0){$u_1$}} 
\put(150,30){\makebox(0,0){$\bullet$}} 
\put(145,20){{$u_2=\cdots = u_r$}} 
\put(120,10){\makebox(0,0){$\bullet$}} 
\put(130,3){\makebox(0,0){$u_{r+1}$}}
\put(100,60){\line(1,0){200}}
\put(150,30){\line(1,0){150}} 
\put(310,60){\makebox(0,0){$L_1$}} 
\put(305,25){$L_2 = \cdots =L_r$} 
\qbezier(120,10)(120,45)(150,45)
\put(150,45){\line(1,0){150}} 
\put(315,45){\makebox(0,0){$L_{r+1}'$}}
\put(315,7){\makebox(0,0){$L_{r+1}$}} 
\put(120,10){\line(1,0){180}} 
\end{picture} 
\end{center} 
\caption{Left mutation of $L_{r+1}$} 
\label{fig:left_mutation} 
\end{figure}

A mutation occurs when we vary the direction $e^{\iu\phi}$ 
and $e^{\iu\phi}$ becomes non-admissible. 
We now let the phase $\phi$ decrease by the angle $\pi$ continuously. 
Then the asymptotically exponential flat sections 
$y_1,\dots,y_N$ undergo a sequence of right mutations. 
Let $\my_1,\dots,\my_N$ be the basis of asymptotically 
exponential flat sections associated to $-e^{\iu\phi}$ 
as in Proposition \ref{prop:repeated_eigenvalues}. 
In the situation of Figure \ref{fig:mutation}, we successively 
right-mutate $L_1'$ across $u_{r+1}, u_{r+2}, \dots, u_N$, 
arriving at: 
\begin{equation} 
\label{eq:Stokes_example}
\my_1(z) = y_1(z) - c_{12} y_2(z) - \cdots - c_{1r}y_r(z) 
- \left( 
\substack{\text{\scriptsize linear combinations of $y_j(z)$} \\  
\text{with $\im(e^{-\iu\phi}u_j)< \im(e^{-\iu\phi}u_2)$}}  
\right). 
\end{equation} 
In this way we can write $\my_i$ as a linear combination of 
$y_j$'s and vice versa. The transition matrix is called 
the Stokes matrix. 

\begin{figure}[htbp]  
\begin{center} 
\includegraphics[bb=206 560 444 718]{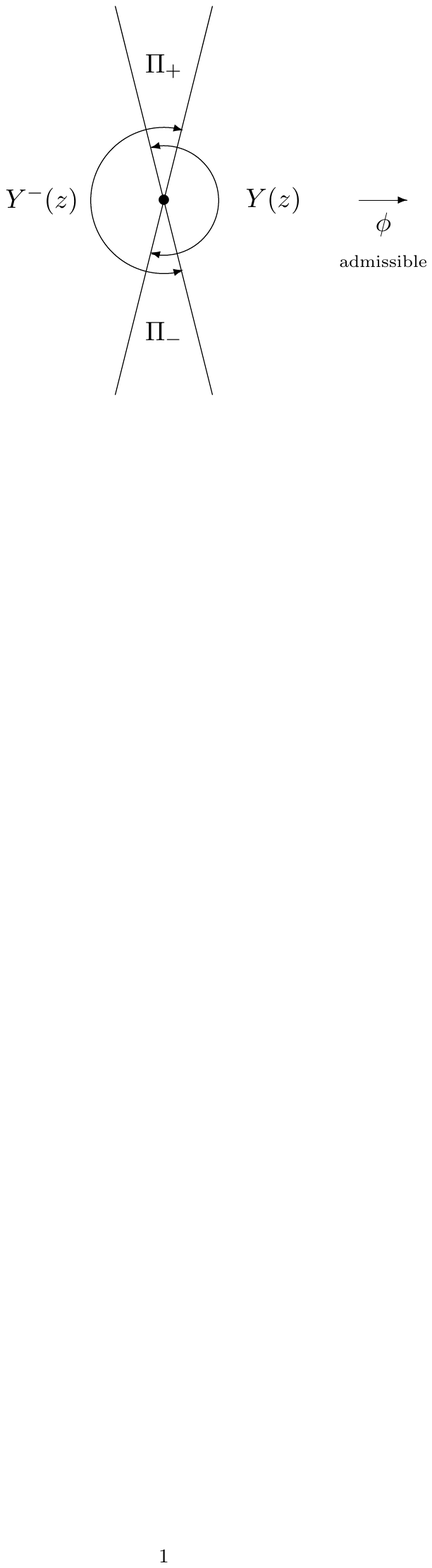}
\end{center} 
\caption{Domains of the two solutions $Y$ and $\mY$}
\label{fig:Stokes} 
\end{figure} 
\begin{definition}[{\cite[Remark 1.6]{BJL79}, \cite[Definition 4.3]{Dubrovin98a}}] 
\label{def:Stokes} 
Let $Y=(y_1,\dots,y_N)$, $\mY=(\my_1,\dots,\my_N)$ 
be the fundamental solution 
associated to an admissible direction 
$e^{\iu\phi}$ and $-e^{\iu\phi}$ respectively 
as in Proposition \ref{prop:repeated_eigenvalues}.  
Let $\Pi_+ \cup \Pi_-$ be the domain in Figure \ref{fig:Stokes} 
where both $Y$ and $\mY$ are defined. 
The \emph{Stokes matrices} (or Stokes multipliers) are 
the constant matrices\footnote
{Our Stokes matrices are inverse to the ones in 
\cite[Definition 4.3]{Dubrovin98a}.} 
$S$ and $S_-$ satisfying  
\begin{alignat*}{2} 
Y(z) & = \mY(z) S   \quad &  \text{for $z\in \Pi_+$} \\
Y(z) & = \mY(z) S_-  & \text{for $z\in \Pi_-$} 
\end{alignat*} 
\end{definition} 

\begin{proposition}[{\cite[Theorem 4.3, (4.39)]{Dubrovin98a}}] 
\label{prop:Stokes} 
Let $(y_1,\dots,y_N)$ and $(\my_1,\dots,\my_N)$ be the 
asymptotically exponential 
fundamental solution associated to admissible directions 
$e^{\iu\phi}$ and $-e^{\iu\phi}$ respectively and 
let $S=(S_{ij}), S_-=(S_{-,ij})$ be the Stokes matrices. 
We have 
\begin{enumerate} 
\item $S_{ij} = S_{-,ji} = (y_i(\tau,e^{-\pi\iu} z),y_j(\tau,z))_F$ 
for $z\in \Pi_+$. 
Here we write $e^{-\pi\iu}z$ instead of $-z$ to specify the path of 
analytic continuation. Similarly, $(\my_i(\tau,e^{\pi\iu}z), \my_j(\tau,z))_F$ 
gives the coefficient $(S^{-1})_{ij}$ of the inverse matrix of $S$. 

\item The Stokes matrix $S$ is a triangular matrix 
with diagonal entries all equal to one. 
More precisely, we have $S_{ii} = 1$ for all $i$ and 
\begin{align*} 
& S_{ij} = 0 \qquad \text{if $i\neq j$ and $u_i = u_j$;} \\
& S_{ij} = 0 \qquad \text{if $\im(e^{-\iu\phi} u_i) < \im(e^{-\iu\phi} 
u_j)$.}
\end{align*} 
\end{enumerate} 
In particular, $S$, $S_-$ do not depend on $\tau$ 
and $y_1,\dots,y_N$ are semiorthonormal. 
\end{proposition} 
\begin{proof} 
Dubrovin \cite{Dubrovin98a} discussed the case 
where $u_1,\dots,u_N$ are distinct. 
We have 
$y_j(z) = \sum_{k=1}^N \my_k(z) S_{kj}$ 
for $z\in \Pi_+$ and 
$y_j(z) = \sum_{k=1}^N \my_k(z) S_{-,kj}$ 
for $z\in \Pi_-$. 
Taking the pairing with $y_i(-z)$ and 
using the property $(\my_k(z), y_i(-z))_F= \delta_{ki}$ 
from Proposition \ref{prop:repeated_eigenvalues}, we 
obtain 
\begin{align*} 
& (y_i(-z), y_j(z))_F  = S_{ij} \quad \text{for $z\in \Pi_+$} \\ 
& (y_i(-z), y_j(z))_F = S_{-,ij} \quad \text{for $z\in \Pi_-$} 
\end{align*} 
Replacing $z$ with $-z$ in the second formula and taking into 
account the direction of analytic continuation, we see that (1) and (2) hold. 
(The discussion for $(\my_i(e^{\pi\iu}z), \my_j(z))_F$ is similar.) 
As already discussed, the coefficients of the Stokes matrix 
arise from a sequence of 
mutations. Part (3) is obvious from this. 
\end{proof} 
\begin{remark} 
We often choose an ordering of normalized idempotents 
$\Psi_1,\dots,\Psi_N$ 
so that the corresponding eigenvalues satisfy 
$\im(e^{-\iu\phi}u_1) \ge \cdots \ge \im(e^{-\iu\phi}u_N)$. 
Then the Stokes matrix becomes \emph{upper}-triangular. 
Note also that flat sections $y_i$ corresponding 
to the same eigenvalue are mutually \emph{orthogonal}. 
\end{remark} 

Let us go back to the situation of Figure \ref{fig:mutation}. 
The coefficient $c_{1i}$ appearing in \eqref{eq:Stokes_example} 
coincides with the coefficient $S_{1i}$ of the Stokes matrix, 
because by inverting \eqref{eq:Stokes_example} we obtain 
\[
y_1(z) = \my_1(z) + c_{12} \my_2(z) + \cdots + c_{1r} \my_r(z) 
+ \left( 
\substack{\text{\scriptsize linear combinations of $\my_j(z)$} \\  
\text{with $\im(e^{-\iu\phi}u_j)< \im(e^{-\iu\phi}u_2)$}}  
\right)
\]
for $z\in \Pi_-$ and therefore $c_{1i} = S_{-,i1}= S_{1i}$. 
Recall that $c_{1i}$ arises as the coefficient  
of the right mutation \eqref{eq:y1_mutation}. 
Therefore we obtain the following corollary: 

\begin{corollary}[{\cite[Theorem 4.6]{Dubrovin98a}}] 
\label{cor:mutation} 
Let $e^{\iu\phi}$ be an admissible direction and let 
$(y_1,\dots,y_N)$ be the asymptotically exponential 
fundamental solution associated to $e^{\iu\phi}$ 
in Proposition \ref{prop:repeated_eigenvalues}. 
Let $u_a$, $u_b$ be distinct eigenvalues such that 
there are no eigenvalues $u_j$ with 
$\im(e^{-\iu\phi}u_a) > \im(e^{-\iu\phi} u_j)> \im(e^{-\iu\phi}u_b)$. 
Then the right mutation of the flat section $y_a(z)$ with respect to $u_b$ is 
given by 
\[
y_a \mapsto y_a - \sum_{j : u_j = u_b} S_{aj} y_j. 
\]
Similarly, the left mutation of $y_b(z)$ with respect to $u_a$ is given by: 
\[
y_b \mapsto y_b - \sum_{j:u_j = u_a} S_{jb} y_j 
\]
where $S_{ij} = (y_i(\tau,e^{-\pi\iu}z), y_j(\tau,z))_F$ are the coefficients 
of the Stokes matrix. 
\end{corollary}

\subsection{Isomonodromic deformation}  
\label{subsec:isomonodromy} 
In semisimple case, the base of the big quantum 
connection \eqref{eq:big_q_conn} can be extended to the 
universal covering of a configuration space. 

Suppose that the quantum product is convergent and 
semisimple in a neighbourhood of $\tau=0\in H^\udot(F)$. 
Since the eigenvalues $u_1,\dots,u_N$ of $E\star_\tau$ form 
a local co-ordinate system $H^\udot(F)$, we can make 
$u_1,\dots,u_N$ pairwise distinct by a small deformation 
of $\tau$. We take a base point $\tau^\circ\in H^\udot(F)$ 
such that the corresponding eigenvalues 
$\u^\circ = \{u_1^\circ, \dots, u_N^\circ\}$ 
are pairwise distinct. 
The quantum connection $\nabla|_{\tau^\circ}$ then admits a 
unique isomonodromic deformation 
over the universal cover $C_N(\C)\sptilde$ of the 
configuration space 
\begin{equation} 
\label{eq:config} 
C_N(\C) = \{(u_1,\dots,u_N) \in \C^N: u_i \neq u_j \ 
\text{for all $i\neq j$}\} \big/ \frS_N.  
\end{equation} 
\begin{proposition}[{\cite[Lemma 3.2, Exercise 3.3, Lemma 3.3]{Dubrovin98a}}] 
\label{prop:isomonodromic_deformation}
We have a unique meromorphic flat connection 
on the trivial $H^\udot(F)$-bundle over $C_N(\C)\sptilde \times \P^1$ 
of the form:  
\begin{align*} 
\nabla_{\parfrac{}{u_i}} & = \parfrac{}{u_i} + \frac{1}{z} \CC_i \\ 
\nabla_{z\parfrac{}{z}}  &=z  \parfrac{}{z} - \frac{1}{z} \U + \mu
\end{align*} 
where $\CC_i$ and $\U$ are $\End(H^\udot(F))$-valued holomorphic functions  
on $C_N(\C)\sptilde$, 
such that it restricts to the big quantum connection \eqref{eq:big_q_conn} 
in a neighbourhood of the base point $\u^\circ$. 
Here the eigenvalues of $\U$ give the co-ordinates 
$u_1,\dots,u_N$ on the base. 
\end{proposition} 

\begin{remark}[{\cite[Lemma 3.3]{Dubrovin98a}}] 
This isomonodromic deformation defines a Frobenius manifold structure 
on an open dense subset of $C_N(\C)\sptilde$. 
\end{remark}

\section{Gamma Conjecture I} 
\label{sec:GammaI} 
In this section we formulate Gamma conjecture I 
for arbitrary Fano manifolds.  
We do not need to assume the semisimplicity 
of the quantum cohomology algebra. 

\subsection{Property $\O$}
We introduce Property $\O$ for a Fano manifold $F$. 
\begin{definition}
\label{def:conjO}
Let $F$ be a Fano manifold and let $c_1(F)\star_0 \in \End(H^\udot(F))$ 
be the quantum product at $\tau=0$ (see \S \ref{subsec:qcoh}). 
Set 
\begin{equation*} 
%% \label{eq:T} 
T := \max\{ |u| : \text{$u$ is an eigenvalue of $(c_1(F)\star_0)$}\}. 
\end{equation*} 
We say that $F$ satisfies \emph{Property $\O$} if 
\begin{enumerate} 
\item $T$ is an eigenvalue of $(c_1(F)\star_0)$. 

\item If $u$ is an eigenvalue of $(c_1(F)\star_0)$ with $|u| =T$, 
we have $u = T \zeta$ for some $r$-th root of unity 
$\zeta \in \C$, where $r$ is the Fano index of $F$. 

\item  The multiplicity of the eigenvalue $T$ is one. 
\end{enumerate} 
\end{definition} 
\begin{conjecture}[Conjecture $\O$]
Every Fano manifold satisfies Property $\O$. 
\end{conjecture} 

The number $T\in \R_{\ge 0}$ above is an algebraic number.  
This is an invariant of a monotone \emph{symplectic} 
manifold.  Conjecture $\O$ says that $T$ is non-zero unless $F$ is a point.  
We do not have a general argument to show Conjecture $\O$, but 
there are a few weak evidences as follows. 

\begin{remark} 
Let $r$ be the Fano index. The formula \eqref{eq:qprod_antican} 
shows that we have $\star_0 = \star_\tau$ with $\tau = c_1(F) (2\pi\iu/r)$. 
This together with \eqref{eq:qprod_scaling} gives the symmetry: 
\begin{equation} 
\label{eq:symmetry_c1star}
(c_1(F) \star_0) = e^{2\pi\iu/r} e^{-2\pi\iu \mu/r} (c_1(F)\star_0) 
e^{2\pi\iu \mu/r}   
\end{equation} 
Therefore the spectrum of $(c_1(F)\star_0)$ is invariant under 
multiplication by $e^{2\pi\iu/r}$. 
\end{remark} 

\begin{remark}[$\O$ is the structure sheaf] 
Mirror symmetry for Fano manifolds claims that 
$F$ is mirror to a holomorphic function $f \colon Y \to \C$ 
on a complex manifold $Y$. 
Under mirror symmetry it is expected that 
the eigenvalues of $(c_1(F)\star_0)$ should coincide with 
the critical values of $f$. 
Each Morse critical point of $f$ associates a vanishing 
cycle (or Lefschetz thimble) which gives an object of the 
Fukaya-Seidel category of $f$.  
Under homological mirror symmetry, 
the vanishing cycle associated to $T$ should correspond 
to the structure sheaf $\O$ of $F$. 
This explains the name `Conjecture $\O$'. 
Our Gamma conjectures give a direct link between $T$ and $\O$ 
in terms of differential equations (without referring to mirror 
symmetry): the asymptotically exponential 
flat section associated to $T$ should correspond to 
$\Gg_F = \Gg_F \Ch(\O)$; 
more generally a simple eigenvalue should correspond 
to $\Gg_F \Ch(E)$ for an exceptional object $E \in \D^b_{\rm coh}(F)$ 
(see Gamma Conjecture I, II in \S \ref{subsec:GammaI}, \S \ref{subsec:GammaII}). 
\end{remark} 

\begin{remark}[Conifold point \cite{Galkin:conifoldpoint}] 
\label{rem:conifold} 
Let $f\colon (\C^\times)^n \to \C$ 
be a convenient\footnote{A Laurent polynomial is said to be 
\emph{convenient} if the Newton polytope of $f$ contains the origin 
in its interior.} Laurent polynomial mirror to a Fano manifold $F$. 
Suppose that all the coefficients 
of $f$ are positive real. Then the Hessian 
\[
\parfrac{^2f}{\log x_i \partial \log x_j}(x_1,\dots,x_n) 
\]
is positive definite on the subspace $(\R_{>0})^n$. 
Thus $f|_{(\R_{>0})^n}$ attains a global minimum at a unique 
critical point $x_{\rm con}$ in the domain $(\R_{>0})^n$; 
also the critical point $x_{\rm con}$ is non-degenerate. 
We call the point $x_{\rm con}$ the \emph{conifold point}. 
Many examples suggest that $T$ equals 
$T_{\rm con} := f(x_{\rm con})$. 
\end{remark} 

\begin{remark}[Perron-Frobenius eigenvalue] 
This remark is due to Kaoru Ono. We say that a square matrix $M$ of 
size $n$ is \emph{irreducible} if the linear map $\C^n \to \C^n$, 
$v\mapsto M v$ admits no invariant 
co-ordinate subspace. 
If $(c_1(F)\star_0)$ is represented by an irreducible matrix 
with nonnegative entries, $T$ is a simple eigenvalue of $(c_1(F)\star_0)$ 
by Perron-Frobenius theorem. 
\end{remark} 

\begin{remark}[Fano orbifold] 
For a Fano orbifold, $T$ may have multiplicity bigger than 
one. Consider the $\Z/3\Z$ action on $\P^2$ given by 
$[x,y,z] \mapsto [x, e^{2\pi\iu/3} y, e^{4\pi\iu/3}z]$ 
and let $F$ be the quotient Fano orbifold $\P^2/(\Z/3\Z)$. 
The mirror is $f = x_1^{-1}x_2^{-1} + 
x_1^2  x_2^{-1} + x_1^{-1}x_2^2$ and $T= 3$ has multiplicity three. 
Under homological mirror symmetry, 
three critical points in $f^{-1}(T)$ correspond to three 
flat line bundles which are mutually 
orthogonal in the derived category. 
One could guess that the multiplicity of $T$ equals 
the number of irreducible representations of 
$\pi_1^{\rm orb}(F)$ for a Fano orbifold $F$. 
\end{remark} 

\begin{remark} 
In the rest of the section we assume that our Fano manifold $F$ satisfies 
Property $\O$. 
However, the argument in this section (i.e.~\S \ref{sec:GammaI}) 
except for \S \ref{subsec:Apery} 
works under the following weaker assumption: there exists a 
complex number $T$ such that (1) $T$ is an eigenvalue of $(c_1(F)\star_0)$ of 
multiplicity one and (2) if $\tilde{T}$ is an eigenvalue of $(c_1(F)\star_0)$ 
with $\re(\tilde{T})\ge \re(T)$, then $\tilde{T} = T$. 
In \S \ref{subsec:Apery}, we use the condition part (2) in 
Definition \ref{def:conjO}. 
\end{remark} 

When $F$ satisfies Property $\O$, we write 
\begin{equation}
\label{eq:secondbiggest}
T' := \max \left\{\re(u): u \neq T, \ 
\text{$u$ is an eigenvalue of $(c_1(F)\star_0)$}\right \}<T 
\end{equation}
for the second biggest real part for the eigenvalues of $(c_1(F)\star_0)$. 

\subsection{Asymptotic solutions along positive real line}
\label{subsec:positivereal}
We construct a fundamental solution for $\nabla$ 
with nice asymptotic properties over the positive real line; 
here we do not assume semisimplicity, 
cf.~\S \ref{subsec:asymptotic_solution}. 

Suppose that a Fano manifold $F$ satisfies Property $\O$. 
Let $\lambda_1,\dots,\lambda_k$ be the distinct eigenvalues of 
$(c_1(F)\star_0)$ and let $N_i$ be the multiplicity of $\lambda_i$. 
Note that $N=N_1+ \cdots +N_k$. 
We may assume that $T= \lambda_1$; Property $\O$ implies $N_1=1$. 
Note that the Euler vector field $E$ at $\tau=0$ equals $c_1(F)$. 
We may also assume that the matrix $U$ \eqref{eq:U} of 
eigenvalues of $(E\star_\tau)|_{\tau=0} = (c_1(F) \star_0)$ is of the form: 
\[
U = \begin{pmatrix} 
u_1 &&& \\ 
& u_2 && \\
 && \ddots & \\
 &&& u_N 
\end{pmatrix}
= 
\begin{pmatrix} 
T  & & & \\
& \lambda_2 I_{N_2} & & \\
& & \ddots & \\
& & & \lambda_k I_{N_k} 
\end{pmatrix}  
\]
where $I_{N_i}$ is the identity matrix of size $N_i$. 
Choose a linear isomorphism $\Psi \colon \C^N \to H^\udot(F)$ which 
transforms $(c_1(F)\star_0)$ to the block-diagonal form: 
\[
\Psi^{-1} (c_1(F)\star_0) \Psi = \begin{pmatrix} 
B_1 &  & & \\
& B_2  & & \\ 
& & \ddots & \\ 
& & & B_k 
\end{pmatrix} 
\]
where $B_i$ is a matrix of size $N_i$ and $B_i - \lambda_i I_{N_i}$ 
is nilpotent. 
\begin{proposition} 
\label{prop:asymptoticsol_positivereal}
Suppose that a Fano manifold $F$ satisfies Property $\O$. 
With the notation as above, 
there exists a fundamental matrix solution for the quantum 
connection $\nabla_{z\partial_z}$ \eqref{eq:nabla_tau_zero} 
at $\tau=0$ of the form 
\begin{equation*} 
%% \label{eq:asymptoticsol_positivereal}
P(z) e^{-U/z} 
\begin{pmatrix} 
F_1(z) & & & \\ 
& F_2(z) & &  \\ 
& & \ddots & \\
& & & F_k(z) 
\end{pmatrix} 
\end{equation*}
over a sufficiently small sector 
$\cS = \{z \in\C^\times: |z|\le 1, |\arg(z)|\le \epsilon\}$ 
with $\epsilon>0$ such that 
\begin{enumerate} 
\item $P(z)$ has an asymptotic expansion 
$P(z) \sim \Psi + P_1 z + P_2 z^2 + \cdots$ 
as $z\to 0$ in $\cS$; 

\item $F_i(z)$ is a $GL_{N_i}(\C)$-valued function 
satisfying the estimate $\max(\|F_i(z)\|, \|F_i(z)^{-1}\|) \le C 
\exp(\delta |z|^{-p})$ on $\cS$ for some $C,\delta>0$ 
and $0<p<1$. 

\item $F_1(z) =1$. 
\end{enumerate} 
\end{proposition} 
\begin{proof} 
By a result of Sibuya \cite{Sibuya:simplification} 
and Wasow \cite[Theorem 12.2]{Wasow}, we can 
find a gauge transformation $P(z)$ over $\cS$ satisfying (1) above 
such that the new connection matrix $C(z)$ 
\[
P(z)^{-1} \nabla_{\partial_z} P(z) 
= \partial_z - z^{-2} C(z) 
\]
satisfies: 
\begin{itemize} 
\item $C(z)= \diag[C_1(z),\dots,C_k(z)]$ 
is block-diagonal with $C_i(z)$ being of size $N_i$;
\item $C_i(z)$ has an asymptotic expansion $C_i(z) \sim C_{i,0} + C_{i,1}z + 
C_{i,2}z^2 + \cdots $ as $z\to 0$ in $\cS$; 
\item $C_{i,0} - \lambda_i I_{N_i}$ is nilpotent. 
\end{itemize} 
Now it suffices to find a fundamental solution for 
the $i$th block $\nabla^{(i)}:=\partial_z - z^{-2} C_i(z)$. 
We have $e^{\lambda_i/z} \nabla^{(i)} e^{-\lambda_i/z} 
= \partial_z- z^{-2} (C_i(z) - \lambda_i I_{N_i})$ 
and the leading term of the asymptotic expansion of 
$C_i(z) - \lambda_i I_{N_i}$ is nilpotent. 
By \cite[Lemma 5.4.1]{Sibuya:book}, 
any fundamental solution $F_i(z)$ for 
$\partial_z - z^{-2} (C_i(z) - \lambda_i I_{N_i})$ 
satisfies the estimate in part (2) for $1-1/N_i<p<1$. 

We show that part (3) can be achieved. 
Set $D(z) = C_1(z) - \lambda_1$. Then $F_1(z)$ is a solution 
to the scalar differential equation $(\partial_z - z^{-2} D(z))F_1(z) =0$. 
Arguing as in \cite[Theorem 12.3]{Wasow}, we find that 
$F_1(z) = z^c g(z)$ for some $c\in \C$ and 
a holomorphic function $g(z)$ on $\cS$ which has 
an asymptotic expansion $g(z) \sim a_0 + a_1 z + a_2 z^2 + \cdots$ 
at $z=0$ with $a_0\neq 0$. 
Thus we obtain a $\nabla$-flat section $s_1(z) = e^{-T/z}z^c g(z) P_1(z)$ 
where $P_1(z)$ is the first column of $P(z)$. 
Studying the asymptotics of the equation 
$\nabla s_1(z) =0$ (using e.g.~\cite[Theorem 8.8]{Wasow}), we have 
\[
(\mu + c) a_0 P_1(0) \in \im((c_1(F)\star_0) - T).  
\]
Recall that $P_1(0)$ is a non-zero $T$-eigenvector of $(c_1(F)\star_0)$. 
By Lemma \ref{lem:mu_block} below, we must have $c=0$. 
Thus we can absorb $F_1(z) = g(z)$ into the first column of $P(z)$. 
\end{proof} 

\begin{lemma} 
\label{lem:mu_block} 
Suppose that a Fano manifold $F$ satisfies Property $\O$. 
Let $E(T)\subset H^\udot(F)$ denote the one-dimensional 
$T$-eigenspace of $(c_1(F)\star_0)$ and 
define $H' := \im(T - (c_1(F)\star_0))$ 
which is a complementary subspace of $E(T)$. 
Then $E(T)$ and $H'$ are 
orthogonal with respect to the Poincar\'{e} pairing and 
the endomorphism $\mu$ satisfies $\mu(E(T)) \subset H'$. 
\end{lemma} 
\begin{proof} 
Take $\alpha \in E(T)$ and $\beta \in H'$. 
There exists an element $\gamma \in H^\udot(F)$ 
such that $\beta = (T - c_1(F)\star_0) \gamma$. 
Then we have 
\[
(\alpha,\beta)_F = (\alpha, T\gamma - c_1(F)\star_0\gamma)_F  
= (T\alpha -c_1(F)\star_0\alpha, \gamma)_F= 0. 
\]
Thus $E(T) \perp H'$ with respect to $(\cdot,\cdot)_F$.  
The skew-adjointness of $\mu$ with respect to $(\cdot,\cdot)_F$ 
shows that $(\mu\alpha,\alpha)_F =0$; hence $\mu(\alpha) \in H'$. 
\end{proof}

\subsection{Flat sections with the smallest asymptotics} 
\label{subsec:smallest_asymptotics} 
Under Property $\O$, one can find a one-dimensional 
vector space $\AA$ of flat sections $s(z)$ having 
the most rapid exponential decay $s(z) \sim e^{-T/z}$ 
as $z \to +0$ on the positive real line. 
We refer to such sections as flat sections 
\emph{with the smallest asymptotics}. 
It is easy to deduce the following result from 
Proposition \ref{prop:asymptoticsol_positivereal}. 

\begin{proposition} 
\label{prop:smallest_asymptotics}
Suppose that a Fano manifold $F$ satisfies Property $\O$. 
Consider the quantum connection $\nabla$ in \eqref{eq:nabla_tau_zero} and 
define: 
\begin{align} 
\label{eq:smallest_asymptotics}
\AA & := \left\{ s \colon \R_{>0} \to H^\udot(F) : \nabla s =0, \ 
\|e^{T/z} s(z)\| = O(z^{-m}) \ \text{as $z \to +0$} \ (\exists m)\right \} 
\end{align} 
Then: 
\begin{itemize} 
\item[(1)] 
$\AA$ is a one-dimensional complex vector space.
\item[(2)] 
For $s \in \AA$, 
the limit $\lim_{z\to +0} e^{T/z} s(z)$ exists and lies 
in the $T$-eigenspace $E(T)$ of $(c_1(F)\star_0)$. 
The map $\AA \to E(T)$ defined by this limit is an isomorphism. 
\item[(3)] 
Let $s \colon \R_{>0} \to H^\udot(F)$ be a flat section 
which does not belong to $\AA$. Then 
we have $\lim_{z\to +0} \|e^{\lambda /z} s(z)\| = \infty$ 
for all $\lambda >T'$ where $T'$ is in \eqref{eq:secondbiggest}. 
\end{itemize}  
\end{proposition} 
\begin{proof} 
With the notation as in Proposition \ref{prop:asymptoticsol_positivereal}, 
any $\nabla$-flat section $s(z)$ on the positive real line 
can be written in the form: 
\[
s(z) = P(z) e^{-U/z} (F_1(z) v_1 + \cdots + F_k(z) v_k) 
\]
with $v_i \in \C^{N_i}$ constant vectors. 
We claim that $s(z)$ belongs to $\AA$ if and only if 
$v_2 = v_3 = \cdots = v_k =0$. 
The `if' part of the claim 
is obvious from Proposition \ref{prop:asymptoticsol_positivereal}. 
Suppose that $s(z) \in \AA$. 
Take $i\ge 2$. 
The assumption gives that $e^{(T-\lambda_i)/z} F_i(z) v_i$  
is of polynomial growth as $z \to +0$. 
(The factor $P(z)$ is irrelevant, as it converges to the invertible matrix $\Psi$ 
as $z\to +0$.) 
On the other hand, from $\re(\lambda_i)<T$ and 
the estimate of the norm $\|F_i(z)^{-1}\|$ 
in Proposition \ref{prop:asymptoticsol_positivereal}, 
we obtain an estimate of the form 
$\|e^{(\lambda_i-T)/z} F_i(z)^{-1}\| \le C e^{-\varepsilon/z}$  
on $|z|\le 1$ for some $C, \varepsilon>0$.  
By taking the limit $z\to +0$ in the inequality 
\[
\|v_i\| \le  \left\|e^{(\lambda_i-T)/z} F_i(z)^{-1}\right\|  
\left \|e^{(T-\lambda_i)/z} F_i(z) v_i \right \| 
\le C e^{-\varepsilon/z} \left \|e^{(T-\lambda_i)/z} F_i(z) v_i \right \|, 
\]  
we find that $v_i=0$. The claim follows. 
Parts (1) and (2) of the proposition follows easily from the claim. 
To show part (3), it suffices to show that 
$\|e^{(\lambda -\lambda_i)/z} F_i(z) v_i\| \to \infty$ 
as $z\to +0$ for $i\ge 2$, $\lambda > T'$ and $v_i \neq 0$. 
Using again the estimate for $\|F_i(z)^{-1}\|$, we have an 
estimate of the form 
$\|e^{(\lambda_i-\lambda)/z} F_i(z)^{-1}\| \le C e^{-\varepsilon/z}$ 
on $|z|\le 1$ for some $C, \varepsilon >0$. 
The inequality 
\[
0 < \|v_i\| \le \left \|e^{(\lambda_i-\lambda)/z} F_i(z)^{-1}\right\| 
\left \|e^{(\lambda-\lambda_i)/z} F_i(z) v_i \right \| 
\]
now implies that $\lim_{z \to +0} \| e^{(\lambda -\lambda_i)/z} F_i(z) v_i\| 
= \infty$. 
\end{proof}

\subsection{Gamma Conjecture I: statement} 
\label{subsec:GammaI}
The \emph{Gamma class} \cite{Lib99,Lu,Iritani07,KKP08,Iritani09} 
of a smooth projective variety $F$ is the class 
\[
\Gg_F := \prod_{i=1}^{\dim F} \Gamma(1+ \delta_i) 
\in H^\udot(F) 
\]
where $\delta_1,\dots,\delta_{\dim F}$ are the Chern roots of 
the tangent bundle $TF$ so that 
\mbox{$c(TF) = \prod_{i=1}^{\dim F} (1+ \delta_i)$}. 
The Gamma function $\Gamma(1+ \delta_i)$ 
in the right-hand side should be expanded in the Taylor series 
in $\delta_i$. We have 
\[
\Gg_F = \exp\left (- C_{\rm eu} c_1(F) + \sum_{k=2}^\infty 
 (-1)^{k}(k-1)! \zeta(k)\ch_{k}(TF) \right)
\]
where $C_{\rm eu}=0.57721...$ is Euler's constant and 
$\zeta(k)$ is the special value of Riemann's zeta function. 

The fundamental solution $S(z) z^{-\mu} z^\rho$ from Proposition 
\ref{prop:fundsol} identifies the space of flat sections over the positive 
real line $\R_{>0}$ with the cohomology group: 
\begin{align}
\label{eq:coh_framing}
\begin{split}  
\Phi \colon 
H^\udot(F) & \longrightarrow \left\{ 
s \colon \R_{>0} \to H^\udot(F) : \nabla s =0 \right\} \\
\alpha & \longmapsto (2\pi)^{-\frac{\dim F}{2}} S(z) z^{-\mu}z^\rho \alpha
\end{split}  
\end{align} 
where we use the standard determination for 
$z^{-\mu} z^\rho = \exp(-\mu\log z) \exp(\rho \log z)$ 
such that $\log z\in\R$ for $z\in \R_{>0}$. 

\begin{definition} 
\label{def:principal_asymptotic_class} 
Suppose that a Fano manifold $F$ satisfies Property $\O$. 
The \emph{principal asymptotic class} of $F$ 
is a cohomology class $A_F \in H^\udot(F)$ such that 
$\AA= \C \Phi(A_F)$, where 
$\AA$ is the space \eqref{eq:smallest_asymptotics} 
of flat sections with the smallest asymptotics (recall that 
$\AA$ is of dimension one by Proposition 
\ref{prop:smallest_asymptotics}). 
The class $A_F$ is defined up to a constant; when 
$\langle [\pt], A_F \rangle \neq 0$, we can normalize 
$A_F$ so that $\langle [\pt], A_F \rangle =1$. 
\end{definition} 

\begin{conjecture}[Gamma Conjecture I]  
\label{conj:GammaI}
Let $F$ be a Fano manifold $F$ satisfying Property $\O$. 
The principal asymptotic class $A_F$ of $F$ is given by $\Gg_F$. 
\end{conjecture} 

We explain that the Gamma class should be viewed 
as a square root of the Todd class (or $\hA$-class) 
in the Riemann-Roch theorem. 
Let $H$ denote the vector space of $\nabla$-flat sections 
over the positive real line: 
\[
H := \{ s \colon \R_{>0} \to H^\udot(F) : \nabla s =0\} 
\]
and let $[\cdot,\cdot)$ be a non-symmetric pairing 
on $H$ defined by 
\begin{equation}
\label{eq:pairing_H}
[s_1, s_2) := (s_1(e^{-\iu\pi} z) , s_2(z) )_F 
\end{equation} 
for $s_1,s_2 \in H$, 
where $s_1(e^{-\iu\pi} z)$ denotes the analytic continuation of $s_1(z)$ 
along the semicircle $[0,1]\ni \theta \mapsto e^{-\iu \pi \theta} z$. 
The right-hand side does not depend on $z$ because of the flatness 
\eqref{eq:nabla_preserves_pairing}. 
The non-symmetric pairing $[\cdot,\cdot)$ induces 
a pairing on $H^\udot(F)$ via the map $\Phi$ \eqref{eq:coh_framing}; 
we denote by the same symbol $[\cdot,\cdot)$ the 
corresponding pairing on $H^\udot(F)$. 
Since $S(z)$ preserves the pairing (Proposition \ref{prop:fundsol}), we have 
\begin{equation} 
\label{eq:[)_coh}
[\alpha, \beta) = \frac{1}{(2\pi)^{\dim F}}
(e^{\pi\iu \mu} e^{-\pi\iu \rho} \alpha, \beta)_F 
= \frac{1}{(2\pi)^{\dim F}} 
(e^{\pi\iu\rho} e^{\pi\iu \mu} \alpha, \beta)_F 
\end{equation} 
for $\alpha,\beta \in H^\udot(F)$. 
We write $\Ch(\cdot)$, $\Td(\cdot)$ for the 
following characteristic classes: 
\begin{align*} 
\Ch(V) &= (2\pi\iu)^{\frac{\deg}{2}} \ch(V) 
= \sum_{j=1}^{\rank V} e^{2\pi\iu \delta_j} \\
\Td(V) & = (2\pi\iu)^{\frac{\deg}{2}} \td(V) 
= \prod_{j=1}^{\rank V} 
\frac{2\pi \iu \delta_j}{1- e^{-2\pi\iu \delta_j}} 
\end{align*} 
where $\delta_1,\dots,\delta_{\rank V}$ are the 
Chern roots of a vector bundle $V$. 
We also write $\Td_F = \Td(TF)$. 

\begin{lemma}[\cite{Iritani07, Iritani09}]
\label{lem:Gamma}   
Let $V_1$, $V_2$ be vector bundles over $F$. Then 
\[
\left[ \Gg_{F}\Ch(V_1), \Gg_F\Ch(V_2)\right) = \chi(V_1,V_2) 
\]
where $\chi(V_1,V_2) = \sum_{i=0}^{\dim F} (-1)^i \dim \Ext^i(V_1,V_2)$. 
\end{lemma} 
\begin{proof} 
Recall the Gamma function identity: 
\begin{equation} 
\label{eq:Gamma_identity}
\Gamma(1+z) \Gamma(1-z) = 
\frac{2\pi\iu z}{e^{\iu\pi z} -e^{-\iu\pi z}}. 
\end{equation} 
This immediately implies that 
\begin{equation*} 
% \label{eq:Gamma_Todd} 
e^{\pi\iu \rho} 
\left( e^{\pi\iu \frac{\deg}{2}} \Gg_F \right) \cdot \Gg_F = 
\Td_F. 
\end{equation*} 
Therefore 
\[
\left[ \Gg_{F}\Ch(V_1), \Gg_F\Ch(V_2)\right) = 
\frac{1}{(2\pi \iu)^{\dim F}} 
\int_F \left(e^{\pi\iu \frac{\deg}{2}} \Ch(V_1) \right) \Ch(V_2) 
\Td_F.  
\] 
This equals $\chi(V_1,V_2)$ by the Hirzebruch-Riemann-Roch formula. 
\end{proof} 

Under Gamma Conjecture I, the canonical generator $\Phi(\Gg_F)$ 
of $\AA$ satisfies the following normalization. 
\begin{proposition} 
\label{prop:normalization} 
Suppose that a Fano manifold $F$ satisfies Property $\O$ and 
Gamma Conjecture I. Set $s_1(z) = \Phi(\Gg_F)(z)$. Then we have: 
\begin{enumerate}
\item The limit 
$v:=\lim_{z\to +0} e^{T/z}s_1(z) \in E(T)$ satisfies $(v,v)_F = 1$.  
\item Let $\hnabla_{\partial_\lambda}$ be the Laplace dual \eqref{eq:2nd_conn} 
of the quantum connection $\nabla$ at $\tau=0$. 
There exists a $\hnabla_{\partial_\lambda}$-flat section 
$\varphi(\lambda)$ which is holomorphic near $\lambda=T$ 
such that $\varphi(T)=v$ and that 
$s_1(z) = 
z^{-1}\int_T^\infty \varphi(\lambda) e^{-\lambda/z} d\lambda$.  
\end{enumerate}
\end{proposition} 
\begin{proof} 
We give a construction of the generator $s_1$ of $\AA$ using 
Laplace transformation as in \S \ref{subsec:asymptotic_solution}. 
Recall that once we have a $\hnabla$-flat section $\varphi(\lambda)$ 
holomorphic near $\lambda =T$ such that $\varphi(T)$ 
is a non-zero eigenvector $v$ of $(c_1(F)\star_0)$ with eigenvalue $T$, 
the Laplace transform 
\[
s_1(z) = \frac{1}{z} \int_{T}^\infty \varphi(\lambda) e^{-\lambda/z} 
d\lambda
\]
gives a $\nabla$-flat section such that $\lim_{z\to +0} e^{T/z} s_1(z) =v$ 
(see the discussion around \eqref{eq:yi_Laplace}). 
Thus it suffices to 
construct a $\hnabla$-flat section $\varphi(\lambda)$ as above. 
Property $\O$ ensures that $\hnabla$ is logarithmic at $\lambda = T$, 
and Lemma \ref{lem:mu_block} ensures that the 
residue $R$ of $\hnabla$ at $\lambda=T$ is nilpotent 
and $R v =0$. Thus we have a fundamental matrix solution  
for $\hnabla$ of the form $U(\lambda) \exp(-R \log(\lambda-T))$ 
with $U(\lambda)$ holomorphic near $\lambda=T$ 
and $U(T) = \id$ \cite[Theorem 5.5]{Wasow}. 
Now $\varphi(\lambda) := U(\lambda) \exp(-R\log (\lambda -T)) v 
= U(\lambda) v$ gives the desired $\hnabla$-flat section. 

It now suffices to show that $v = \varphi(T)$ is of unit length. 
By bending the integration path $[T,\infty]$ 
as in Figure \ref{fig:smallest_asymptotics},  
we obtain the analytic continuation 
$s_1(e^{-\pi\iu} z)$ of $s_1(z)$. 
Watson's lemma \cite[\S 6.2.2]{Ablowitz-Fokas} shows that 
$\lim_{z\to +0} e^{-T/z}s_1(e^{-\pi\iu}z) = v$. 
Therefore we have 
\[
\left[\Gg_F, \Gg_F \right) = \left( 
\Phi(\Gg_F)(e^{-\pi \iu} z ), \Phi(\Gg_F)(z) \right)_F 
=(s_1(e^{-\pi\iu}z),s_1(z))_F
\to (v,v)_F 
\]   
as $z\to +0$. By Lemma \ref{lem:Gamma}, 
the left-hand side equals $\chi(\O_F,\O_F) = 1$ since 
$F$ is Fano. 
Thus we have $(v,v)_F =1$ 
and the conclusion follows. 
\end{proof} 

\begin{figure}[htbp]
\begin{center} 
\includegraphics[bb=77 610 523 726]{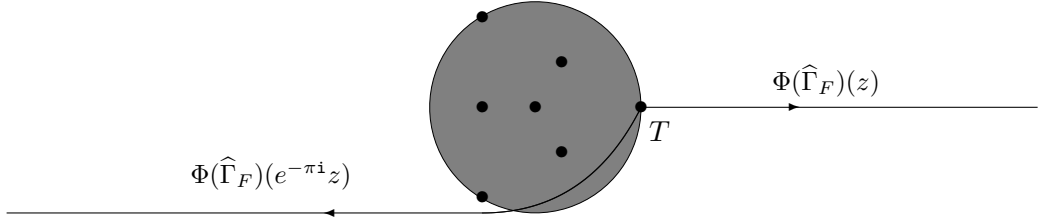} 
\end{center} 
\caption{Spectrum of $(c_1(F)\star_0)$ and 
the paths of integration (on the $\lambda$-plane). }
\label{fig:smallest_asymptotics}
\end{figure}

\subsection{The leading asymptotics of dual flat sections} 
\label{subsec:leading_asymp}
We now study flat sections of the dual quantum 
connection, which we call dual flat sections. 
We show that the limit $\lim_{z\to +0} e^{T/z} f(z)$ exists 
for every dual flat section $f(z)$.  

Let $H_{\ldot}(F)=H_{\rm even}(F;\C)$ denote the even part of 
the homology group with complex coefficients. 
The \emph{dual quantum connection} is a meromorphic flat connection on 
the trivial bundle $H_\ldot(F) \times \P^1 \to \P^1$ 
given by (cf.~\eqref{eq:nabla_tau_zero}) 
\begin{equation} 
\label{eq:dual_qconn}
\nabla^\vee_{z\parfrac{}{z}} = z\parfrac{}{z} + \frac{1}{z} 
(c_1(F)\star_0)^{\rm t} - \mu^{\rm t} 
\end{equation} 
where the superscript ``t" denotes the transpose. 
This is dual to the quantum connection $\nabla$ in the sense that 
we have 
\[
d\langle f(z), s(z) \rangle = \langle \nabla^\vee f(z), s(z) \rangle 
+ \langle f(z), \nabla s(z) \rangle 
\]
where $\langle \cdot, \cdot \rangle$ is the natural pairing between 
homology and cohomology. 
Because of the property \eqref{eq:nabla_preserves_pairing}, 
we can identify the dual flat connection $\nabla^\vee$ 
with the quantum connection $\nabla$ with the opposite sign of $z$. 
In order to avoid possible confusion about signs, 
we shall not use this identification.  
We write $E(T)^*\subset H_\ldot(F)$ for the one-dimensional 
eigenspace of $(c_1(F)\star_0)^{\rm t}$ with eigenvalue $T$ 
and ${H'}^* =\im(T - (c_1(F)\star_0)^{\rm t})$ for a  
complementary subspace. 
They are dual to $E(T)$, $H'$ from Lemma \ref{lem:mu_block}. 

\begin{proposition}
\label{prop:lead} 
Suppose that a Fano manifold $F$ satisfies Property $\O$. 
Then: 
\begin{enumerate} 
\item The limit 
\[ 
\LA(f) := \lim_{z\to +0} e^{-T/z} f(z)  
\]
exists for every $\nabla^\vee$-flat section $f(z)$ and lies in $E(T)^*$. 
In other words, $\LA$ defines a linear functional: 
\[
\LA \colon \left\{ 
f \colon \R_{>0} \to H_{\ldot}(F) : \nabla^\vee f(z) = 0
\right \} 
\longrightarrow E(T)^*.  
\]
\item 
If $\LA(f) = 0$, we have $
\lim_{z\to +0} e^{-\lambda/z} f(z) = 0$ 
for every $\lambda > T'$. 
\end{enumerate}
\end{proposition}  
\begin{proof} 
Note that a fundamental solution for $\nabla^\vee$ is given by 
the inverse transpose of that for $\nabla$. The conclusion 
follows easily by using the fundamental solution in Proposition 
\ref{prop:asymptoticsol_positivereal}. 
\end{proof}

\subsection{A limit formula for the principal asymptotic class} 
The leading asymptotic behaviour of dual flat sections as $z\to +0$ 
is given by the pairing with a flat section with the smallest asymptotics. 
From this we obtain a limit formula for $A_F$. 

The fundamental solution for the dual quantum connection $\nabla^\vee$ 
\eqref{eq:dual_qconn} is given by the inverse transpose 
$(S(z) z^{-\mu} z^\rho)^{-\rm t}$ of the fundamental 
solution for $\nabla$ in Proposition \ref{prop:fundsol}. 
This yields the following identification (cf.~\eqref{eq:coh_framing}) 
\begin{align}
\label{eq:ho_framing} 
\begin{split} 
\Phi^\vee \colon H_\ldot(F) & \longrightarrow \{ f \colon \R_{>0} \to H_\ldot(F) : 
\nabla^\vee f(z) = 0 \} \\ 
\alpha & \longmapsto (2\pi)^{\frac{\dim F}{2}} S(z)^{-\rm t} z^{\mu^{\rm t}} 
z^{-\rho^{\rm t}} \alpha  
\end{split} 
\end{align} 
between dual flat sections and homology classes. 
\begin{proposition} 
\label{prop:LA}
Suppose that $F$ satisfies Property $\O$. 
Let $A_F$ be a principal asymptotic class 
(Definition \ref{def:principal_asymptotic_class}). 
Let $\psi_0^*$ be an element of $E(T)^*$ which 
is dual to $\psi_0 := \lim_{z\to +0}e^{T/z} \Phi(A_F)(z) \in E(T)$ 
in the sense that $\langle \psi_0^*, \psi_0 \rangle =1$. 
Then we have: 
\[
\LA(\Phi^\vee(\alpha)) := \lim_{z\to +0} e^{-T/z}\Phi^\vee(\alpha)(z) 
= \langle \alpha, A_F \rangle \psi_0^*.  
\]
\end{proposition} 
\begin{proof} 
By the definition of $\Phi$ and $\Phi^\vee$ (see \eqref{eq:coh_framing}, 
\eqref{eq:ho_framing}), we have: 
\[
\langle \alpha, A_F \rangle = \langle \Phi^\vee(\alpha)(z), \Phi(A_F)(z) 
\rangle = \langle e^{-T/z} \Phi^\vee(\alpha)(z), 
e^{T/z} \Phi(A_F)(z) \rangle.  
\]
This converges to $\langle \LA(\Phi^\vee(\alpha)), \psi_0 \rangle$ 
as $z\to +0$ by Propositions \ref{prop:smallest_asymptotics}, 
\ref{prop:lead}.  
\end{proof} 

\begin{definition}[\cite{Givental:symplectic}] 
\label{def:J} 
Let $S(z) z^{-\mu}z^\rho$ be the fundamental solution 
from Proposition \ref{prop:fundsol} 
and set $t := z^{-1}$. 
The \emph{$J$-function} $J(t)$ of $F$ is a cohomology-valued 
function defined by: 
\[
J(t)  := z^{\frac{\dim F}{2}} z^{-\rho} z^\mu S(z)^{-1} 1    
\]
where $1\in H^\udot(F)$ is the identity class. 
Alternatively we can define $J(t)$ by the requirement that 
\begin{equation}
\label{eq:pair_J_alpha}
\langle \alpha, J(t) \rangle  = \left(\frac{z}{2\pi}\right)^{\frac{\dim F}{2}}  
\langle \Phi^\vee(\alpha)(z), 1 \rangle 
\end{equation} 
holds for all $\alpha \in H_\ldot(F)$, using the dual flat sections 
$\Phi^\vee(\alpha)$ in \eqref{eq:ho_framing}. 
\end{definition} 

\begin{remark} 
\label{rem:QDE} 
The $J$-function is a solution to the \emph{scalar} differential 
equation associated to the quantum connection 
on the anticanonical line $\tau = c_1(F) \log t$. 
More precisely, we have 
\[
P\left(t, [\nabla_{c_1(F)}]_{z=1} \right) 1 = 0 \ 
\Longleftrightarrow \ 
P\left(t, \textstyle t \parfrac{}{t} \right) J(t) = 0 
\]
for any differential operators $P \in \C\langle t, t\parfrac{}{t} \rangle$, 
where $[\nabla_{c_1(F)}]_{z=1} 
= t\parfrac{}{t} + (c_1(F) \star_{c_1(F) \log t})$. 
Differential equations satisfied by the $J$-function 
are called the \emph{quantum differential equations}. 
See also Remark \ref{rem:anticanonical}. 
\end{remark} 

\begin{remark} 
\label{rem:J_explicit} 
Using the fact that $S(z)^{-1}$ equals the adjoint of $S(-z)$ 
(Proposition \ref{prop:fundsol}) 
and Remark \ref{rem:S_general_tau}, 
we have 
\begin{equation} 
\label{eq:J} 
% J(t) &= z^{-\rho} z^{\frac{\deg}{2}} 
% \left( 1 + \sum_{i=1}^N \sum_{d\in \Eff(F) \setminus \{0\}} 
% \Ang{\frac{\phi_i}{z- \psi}, 1}_{0,2,d}^F \phi^i 
% \right) \\ 
J(t)  = e^{c_1(F) \log t} 
\left( 1 + \sum_{i=1}^N \sum_{d\in \Eff(F) \setminus \{0\}} 
\Ang{\frac{\phi_i}{1-\psi}}_{0,1,d}^F t^{c_1(F) \cdot d} \phi^i 
\right) 
\end{equation} 
where $\{\phi_i\}_{i=1}^N$ and $\{\phi^i\}_{i=1}^N$ are  
bases of $H^\udot(F)$ which are dual with respect to the Poincar\'{e} pairing 
$(\cdot,\cdot)_F$. 
This gives a well-known form of Givental's $J$-function 
\cite{Givental:symplectic} restricted to the anticanonical line 
$\tau = c_1(F) \log t$. 
\end{remark} 

\begin{theorem} 
\label{thm:asymptotic_class}
Suppose that a Fano manifold $F$ satisfies Property $\O$. 
Let $A_F$ be a principal asymptotic class 
(Definition \ref{def:principal_asymptotic_class}) 
and let $J(t)$ be the $J$-function. 
We have 
\[
\lim_{t\to +\infty} \frac{J(t)}{\langle [\pt], J(t)\rangle}
= \frac{A_F}{\langle [\pt],A_F\rangle} 
\]
where the limit is taken over the positive real line and 
exists if and only if $\langle [\pt], A_F \rangle \neq 0$. 
\end{theorem} 
\begin{proof} 
By Proposition \ref{prop:LA}, we have 
\begin{align*} 
\lim_{z\to +0} e^{-T/z} \Phi^\vee(\alpha)(z)  
& = \langle \alpha, A_F \rangle \psi_0^*\\ 
\lim_{z\to +0} e^{-T/z} \Phi^\vee([\pt])(z)  
& = \langle [\pt], A_F \rangle \psi_0^* 
\end{align*} 
for some non-zero $\psi_0^* \in E(T)^*$. 
Therefore by \eqref{eq:pair_J_alpha} we have 
\[
\lim_{t\to +\infty} \frac{\langle \alpha, J(t)\rangle}{\langle [\pt], J(t)\rangle} 
= \lim_{z\to +0} \frac{e^{-T/z} \langle \Phi^\vee(\alpha)(z), 1\rangle }
{e^{-T/z} \langle \Phi^\vee([\pt])(z), 1\rangle }
= \frac{\langle \alpha, A_F \rangle}{\langle [\pt], A_F \rangle}
\]
for all $\alpha \in H_\ldot(F)$. 
Here we used the fact $\langle \psi_0^*, 1\rangle\neq 0$ 
which is proved in Lemma \ref{lem:1_ET} below. 
The limit exists if and only if $\langle [\pt], A_F \rangle \neq 0$. 
The Theorem is proved. 
\end{proof}

\begin{corollary} 
\label{cor:limit_J} 
Suppose that a Fano manifold $F$ satisfies Property $\O$. 
The following statements are equivalent: 
\begin{enumerate} 
\item Gamma Conjecture I 
(Conjecture \ref{conj:GammaI}) holds for $F$. 
\item $\|\Phi(\Gg_F)(z)\|\le C e^{-\lambda/z}$ on $(0,1]$ 
for some $C>0$ and $\lambda > T'$,   
where $T'$ is given in \eqref{eq:secondbiggest} and $\Phi$ is 
given in \eqref{eq:coh_framing}. 
\item We have 
\[
\lim_{t\to +\infty} \frac{J(t)}{\langle [\pt], J(t)\rangle} 
= \Gg_F 
\]
where $J(t)$ is the $J$-function of $F$ given in Definition \ref{def:J} 
and the limit is taken over the positive real line. 
\end{enumerate} 
\end{corollary} 
\begin{proof} 
The equivalence of (1) and (2) is clear from Proposition 
\ref{prop:smallest_asymptotics}. 
The equivalence of (1) and (3) follows from Theorem 
\ref{thm:asymptotic_class}. 
\end{proof} 

\begin{lemma}
\label{lem:1_ET}  
Suppose that $F$ satisfies Property $\O$. 
For a non-zero eigenvector $\psi_0^* \in E(T)^*$, 
we have $\langle \psi_0^*, 1\rangle \neq 0$. 
In other words, the $E(T)$-component of $1$ 
with respect to the decomposition 
$H^\udot(F) = E(T) \oplus H'$ in Lemma \ref{lem:mu_block} 
does not vanish.  
\end{lemma} 
\begin{proof} 
Let $\psi_0\in E(T)$ be a non-zero eigenvector of $(c_1(F)\star_0)$ 
with eigenvalue $T$. We claim that $\psi_0 \star_0 H' =0$. 
Take $\alpha \in H'$ and write $\alpha = T \gamma - c_1(F) \star_0 \gamma$.  
Then we have 
\[
\psi_0 \star \alpha = T \psi_0 \star_0 \gamma - 
\psi_0 \star_0 c_1(F) \star_0 \gamma =0. 
\]
The claim follows. Write $
1 = c \psi_0 + \psi'$ 
with $\psi' \in H'$ and $c\in \C$. 
Applying $(\psi_0\star_0)$, we obtain $\psi_0 = c \psi_0 \star_0 \psi_0$. 
Thus $c\neq 0$ and Lemma follows. 
\end{proof} 

\begin{remark} 
See \cite[\S 7]{CCGGK} for a discussion of the limit formula 
in the classification problem of Fano manifolds. 
\end{remark} 

\subsection{Apery limit} 
\label{subsec:Apery} 
We can replace the continuous limit of the ratio of the $J$-function 
appearing in Theorem \ref{thm:asymptotic_class} with
the (discrete) limit of the ratio of the Taylor coefficients. 
Golyshev \cite{Golyshev08a} considered such limits and 
called them \emph{Apery constants} of Fano manifolds.  
A generalization to \emph{Apery class} was studied by 
Galkin \cite{Galkin:Apery}. 
This limit sees the primitive part of the Gamma class. 

Expand the $J$-function as 
\[
J (t) = e^{c_1(F) \log t} \sum_{n\ge 0} J_n t^{n}. 
\]
The coefficient $J_n\in H^\udot(F)$ is given by: 
\[
J_n := \sum_{i=1}^N 
\sum_{d\in \Eff(F): c_1(F)\cdot d = n} 
\left\langle \phi_i \psi^{n-2- \dim \phi_i} 
\right\rangle^F_{0,1,d} \phi^i 
\]
where we set $\dim \phi_i = \dim F - \frac{1}{2} \deg \phi_i$.  
See Remark \ref{rem:J_explicit}. 
The main theorem in this section is stated as follows: 
\begin{theorem} 
\label{thm:Apery} 
Suppose that $F$ satisfies Property $\O$ and suppose also that 
the principal asymptotic class $A_F$ satisfies $\langle [\pt], A_F \rangle \neq 0$. 
Let $r$ be the Fano index. 
For every $\alpha \in H_\ldot(F)$ such that $c_1(F) \cap \alpha=0$,  
we have 
\[
\liminf_{n\to \infty} \left|
\frac{\langle \alpha,J_{rn} \rangle}{\langle [\pt ],J_{rn}\rangle} 
- \frac{\langle \alpha, A_F \rangle}{\langle [\pt], A_F\rangle} \right| =0. 
\]
\end{theorem} 
\begin{remark} 
Note that $J_n =0$ unless $r$ divides $n$. 
In the left-hand side of the above formula, we set 
$\langle \alpha, J_{rn}\rangle/
\langle [\pt ],J_{rn}\rangle = \infty$ if 
$\langle [\pt], J_{rn} \rangle =0$. 
It is however expected that the following properties 
will hold for Fano manifolds: 
\begin{itemize} 
\item the Gromov--Witten invariants 
$\langle [\pt], J_{rn}\rangle$ are all positive for $n\ge 0$; 
\item $\liminf_{n\to\infty}$ in Theorem \ref{thm:Apery} can 
be replaced with $\lim_{n\to \infty}$. 
\end{itemize} 
\end{remark} 
\begin{remark} 
More generally, if $\alpha, \beta\in H_\ldot(F)$ are homology classes 
such that $\alpha \cap c_1(F) = \beta \cap c_1(F) =0$ and if 
we have $\langle \beta, A_F \rangle \neq 0$, 
we have the limit formula 
\[
\liminf_{n\to \infty} \left |
\frac{\langle \alpha,J_{rn} \rangle}{\langle \beta ,J_{rn}\rangle} 
- \frac{\langle \alpha, A_F \rangle}{\langle \beta , A_F\rangle} 
\right| =0. 
\]
This can be shown by the same argument as below. 
We shall restrict to the case $\beta = [\pt]$ as Gamma Conjecture I implies 
$\langle [\pt], A_F \rangle \neq 0$. 
\end{remark} 

Define the functions $G$ and $\hG$ as: 
\begin{align*} 
G(t) &:= \langle [\pt], J(t) \rangle = 
\sum_{n=0}^\infty G_n  t^n, \quad \text{where} \ 
G_n :=\langle [\pt], J_n \rangle,   \\ 
\hG(\kappa) & := \sum_{n=0}^\infty  n! G_n  \kappa^n 
= \frac{1}{\kappa} \int_0^\infty G(t) e^{-t/\kappa} dt. 
\end{align*} 
These functions $G$, $\hG$ are called (unregularized or regularized) 
\emph{quantum periods} of Fano manifolds \cite{CCGGK}. 
\begin{lemma}
\label{lem:radius_Ghat}
Let $F$ be an arbitrary Fano manifold. 
The convergence radius of $\hG(\kappa)$ is bigger than or equal to $1/T$. 
\end{lemma} 
\begin{proof} 
Recall that we have $G(t) =(2\pi t)^{-\dim F/2}\langle 
\Phi^\vee([\pt])(z), 1 \rangle$ (see \eqref{eq:pair_J_alpha}) 
and $\Phi^\vee([\pt])(z)$ is a $\nabla^\vee$-flat section 
(where $z= t^{-1}$). 
By Proposition \ref{prop:lead}, we have an estimate 
$\|\Phi([\pt])(t^{-1})\| \le C_0 e^{Tt}$ on $t\in [1,\infty)$ 
for some constant $C_0>0$,  
and thus the following Laplace 
transformation converges for $\lambda$ 
with $\re(\lambda)<-T$ and $\nu \gg 0$, $\nu\notin \Z$. 
\[
\varphi(\lambda) = (2\pi)^{-\frac{\dim F}{2}} 
\int_0^\infty t^{\nu-1} e^{\lambda t}\Phi^\vee([\pt])(t^{-1}) dt 
\]
The condition $\nu \gg 0$ ensures that the integral 
converges near $t=0$ (since $\nabla^\vee$ is regular singular 
at $t=0$, $f(t^{-1})$ is of at most polynomial growth 
near $t=0$). 
By the standard argument using integration by parts, we can show 
that $\varphi(\lambda)$ is flat for the connection 
$\tnabla^{(\nu)}$ given by 
\begin{equation*}
% \label{eq:tnabla}
\tnabla^{(\nu)}_{\partial_\lambda} = \partial_\lambda 
+ (\lambda + (c_1(F)\star_0)^{\rm t})^{-1} (-\mu^{\rm t} + \nu). 
\end{equation*}
Then the convergence radius of 
\begin{align} 
\label{eq:Laplace_G}
\begin{split} 
\langle \varphi(-\kappa^{-1}), 1 \rangle 
& = \int_0^\infty t^{\nu+ \frac{\dim F}{2}-1} 
G(t) e^{-t/\kappa} dt \\
& = \kappa^{\nu+\frac{\dim F}{2} -1} 
\sum_{n=0}^\infty  \Gamma(n+\nu+\tfrac{\dim F}{2}) G_n \kappa^n 
\end{split} 
\end{align} 
is bigger than or equal to $1/T$ as 
$\tnabla^{(\nu)}$ has no singularities in $\{|\lambda|>T\}$. 
The function $\hG(\kappa)$ has the same radius of convergence 
as $\langle \varphi(-\kappa^{-1}),1\rangle$ and the conclusion follows. 
\end{proof}

\begin{lemma} 
\label{lem:radius_Ghat_exact} 
Suppose that $F$ satisfies Property $\O$ and 
$\langle [\pt], A_F \rangle \neq 0$. 
Then the convergence radius of $\hG(\kappa)$ is exactly $1/T$. 
\end{lemma} 
\begin{proof} 
In view of the discussion in the previous lemma, 
it suffices to show that the function 
$\langle \varphi(-\kappa^{-1}),1\rangle$ in \eqref{eq:Laplace_G} has 
a singularity at $\kappa = 1/T$. 
By \eqref{eq:pair_J_alpha} we have $G(t) =(2\pi t)^{-\dim F/2}\langle 
\Phi^\vee([\pt])(z), 1 \rangle$. Thus by Proposition \ref{prop:LA} and 
Lemma \ref{lem:1_ET}, we have 
\[
\lim_{t\to + \infty} e^{-t T} (2\pi t)^{\frac{\dim F}{2}} G(t)  = 
\langle [\pt], A_F \rangle \langle \psi_0^*, 1\rangle 
\neq 0. 
\]
Therefore the Laplace transform \eqref{eq:Laplace_G} diverges 
as $\kappa$ approaches $1/T$ from the left for $\nu \gg 0$. 
\end{proof}

\begin{proof}[Proof of Theorem \ref{thm:Apery}]  
Set $\beta:=\alpha - \frac{\langle \alpha, A_F\rangle}
{\langle [\pt],A_F\rangle} [\pt]$. Then 
$\langle \beta, A_F \rangle =0$. 
It suffices to show that 
\begin{equation} 
\label{eq:beta_pt}
\liminf_{n\to \infty} \left| 
\frac{\langle \beta, J_{rn}\rangle}
{\langle [\pt], J_{rn}\rangle} \right| = 0.  
\end{equation} 
Set $b_n := \langle \beta, J_n \rangle$. 
Then we have 
$\sum_{n=0}^\infty b_n t^n = \langle \beta, J(t)\rangle 
= (2\pi t)^{-\dim F/2} \langle \Phi^\vee(\beta)(t^{-1}), 1\rangle$ 
by \eqref{eq:pair_J_alpha}. 
By Proposition \ref{prop:LA}, we have 
$\LA(\Phi^\vee(\beta))= \langle \beta, A_F \rangle \psi_0^*=0$. 
Thus by Proposition \ref{prop:lead} we have 
$\lim_{t\to +\infty} e^{- t \lambda}\Phi^\vee(\beta)(t^{-1})=0$ 
for every $\lambda > T'$.  
Consider again the Laplace transform 
\[
\varphi_\beta(\lambda)= 
(2\pi)^{-\frac{\dim F}{2}} 
\int_0^\infty t^{\nu-1}e^{t\lambda} 
\Phi^\vee(\beta)(t^{-1}) dt 
\]
for $\nu \gg 0$ (as in the proof of Lemma \ref{lem:radius_Ghat}). 
This is a $\tnabla^{(\nu)}$-flat section which is regular on 
$\{\lambda: \re(\lambda)<-T'\}$. Then its component 
\[
\langle \varphi_\beta(-\kappa^{-1}), 1 \rangle 
= \kappa^{\nu+\frac{\dim F}{2} -1} 
\sum_{n=0}^\infty  \Gamma(n+\nu+\tfrac{\dim F}{2}) b_n \kappa^n 
\]
has the radius of convergence strictly bigger than $1/T$. 
This is because the only singularities of $\tnabla^{(\nu)}$ on 
$\{|\lambda|\ge T\}$ are $\{-e^{2\pi \iu k/r} T : k=0,\dots,r-1\}$ 
and  $\langle \varphi_\beta(\lambda), 1 \rangle$  is 
regular at $\lambda= -T$ and satisfies 
$\langle \varphi_\beta(e^{2\pi\iu/r} \lambda), 1\rangle 
= e^{-(\nu+\frac{\dim F}{2} -1)2\pi\iu/r} 
\langle \varphi_\beta(\lambda),1\rangle$.   
(Here we used Part (2) of Property $\O$.) 
This implies that there exist numbers $C, R>0$ such that 
$R>1/T$ and 
\[
|b_n n!| \le  C R^{-n}. 
\]
On the other hand, Lemma \ref{lem:radius_Ghat_exact} implies that 
\[
\limsup_{n\to \infty} |G_n n!|^{1/n}  = T. 
\]
Let $\epsilon>0$ be such that $(1-\epsilon) RT>1$. 
Then we have a number $n_0$ such that for every $n \ge n_0$ 
we have $\sup_{m \ge n} |G_m m!|^{1/m} \ge T(1-\epsilon/2)$. 
Namely we can find a sequence $n_0 \le n_1<n_2<n_3< \cdots$ 
such that $|G_{n_i}| n_i! \ge T^{n_i}(1-\epsilon)^{n_i}$. 
Therefore 
\[
\left |\frac{b_{n_i}}{G_{n_i}} \right | 
= \left |\frac{b_{n_i} n_i !}{G_{n_i}n_i!} \right| 
\le  \frac{C R^{-n_i}}{T^{n_i}(1-\epsilon)^{n_i}}
= \frac{C}{(RT(1-\epsilon))^{n_i}}. 
\]
This implies $\liminf_{n \to \infty} |b_n/G_n|=0$, which 
proves \eqref{eq:beta_pt}. 
\end{proof} 

\subsection{Quantum cohomology central charges} 
\label{subsec:centralcharges} 
For a vector bundle $V$ on $F$, we define 
\begin{equation}
\label{eq:centralcharge}
Z(V) := (2\pi z)^{\frac{\dim F}{2}}
\left (1, \Phi(\Gg_F \Ch(V))(z)\right)_F. 
\end{equation} 
This quantity is called the \emph{quantum cohomology central charge}\footnote
{The central charge $Z(V)$ here is denoted by $(2\pi\iu)^{\dim F} Z(V)$ 
in \cite{Iritani09}.} of $V$ in \cite{Iritani09}.
Using the property $S(z)^* = S(-z)^{-1}$, 
\eqref{eq:coh_framing} and Definition \ref{def:J}, 
we can write $Z(V)$ in terms of the $J$-function: 
\begin{equation}
\label{eq:centralcharge_J} 
Z(V)  = (2\pi\iu)^{\dim F} 
\left[J(e^{\pi\iu} t), \Gg_F \Ch(V) \right) 
% (-1)^{\dim F} \left(
% e^{\pi\iu\rho} J(e^{-\pi\iu} t), \Gg_F^* \Ch(V^*)\right)_F 
\end{equation} 
where $t=z^{-1}$ and $[\cdot, \cdot)$ is the pairing 
in \eqref{eq:[)_coh}. 
% $\Gg_F^* := (-1)^{\frac{\deg}{2}} \Gg_F 
% = \prod_{i=1}^{\dim F} \Gamma(1-\delta_i)$ is the \emph{dual} 
% Gamma class of $F$. 
Therefore, as a component of the $J$-function, $Z(V)$ gives 
a scalar-valued solution to the quantum differential equations 
(see Remark \ref{rem:QDE}). 
If $F$ satisfies Gamma Conjecture I, the 
central charge $Z(\O)$ satisfies the following smallest asymptotics: 
\[
Z(\O) 
\sim  \sqrt{(\psi_0,\psi_0)_F}(2\pi z)^{\frac{\dim F}{2}} e^{-T/z}
\qquad 
\text{as $z\to +0$} 
\]
where $\psi_0 \in E(T)$ is the idempotent for $\star_0$. 
This follows from Proposition \ref{prop:normalization} and 
the proof of Lemma \ref{lem:1_ET}.

\section{Gamma Conjecture II}
\label{sec:GammaII} 

In this section we restrict to a Fano manifold $F$ with 
semisimple quantum cohomology and 
state a refinement of Gamma Conjecture I for such $F$. 
We represent the irregular monodromy of quantum connection 
by a marked reflection system  
given by a basis of asymptotically exponential flat sections. 
Gamma Conjecture II says that the marked reflection system coincides 
with a certain Gamma-basis $\{\Gg_F \Ch(E_i)\}$ for 
some exceptional collection $\{E_i\}$ 
of the derived category $\D^b_{\rm coh}(F)$. 
This also refines Dubrovin's conjecture \cite{Dubrovin98}. 

% \pph {\bf Semiorthonormal bases and Gram matrices. \rm} 
\subsection{Semiorthonormal basis and Gram matrix} 
\label{subsec:sobs}

Let $\V = (V;[\cdot,\cdot);v_1,\dots,v_N)$
be a triple of a vector space $V$, a bilinear form 
$[\cdot , \cdot) = [\cdot,\cdot)_V$,
and a collection of vectors $v_1,\dots,v_N \in V$.
The bilinear form $[\cdot,\cdot)$ is not necessarily symmetric 
nor skew-symmetric. 
We say that $\V$ is a \emph{semiorthonormal collection} 
if the Gram matrix $G_{ij} = [v_i, v_j)$ is uni-uppertriangular:
$\left[ v_i,v_i \right) = 1$
for all $i$, 
and 
$\left[v_i,v_j \right) = 0$
for all $i > j$.
This implies that $G$ is non-degenerate,
hence vectors $v_i$ are linearly independent so $N \leq \dim V$. 
We say that $\V$ is a \emph{semiorthonormal basis} (SOB) 
if $N = \dim V$ i.e.~$v_i$ form a basis of vector space $V$,
this implies that pairing $[\cdot,\cdot)$ is non-degenerate.

% Morphisms of semiorthonormal bases are maps $\phi : V \to V'$
% such that $[u,v)_{V} = [\phi(u),\phi(v))_{V'}$ and $\phi(v_i) = v_i'$.
% Gram matrix gives a one-to-one correspondence between
% isomorphism classes of SOBs and uni-uppertriangular matrices.
% We say that the two SOBs $(V; [\cdot,\cdot)_V; v_1,\dots,v_N)$, 
% $(V'; [\cdot,\cdot)_{V'}; v_1',\dots,v_N')$  
% are \emph{equal} 
% if we have $V= V'$, $[\cdot,\cdot)_V = [\cdot,\cdot)_{V'}$ 
% and $v_i = v_i'$. 

Set of SOBs admits an obvious action of $(\Z/2\Z)^N$: 
if $v_1,\dots,v_N$ is SOB then so is $\pm v_1,\dots,\pm v_N$,
in this case we say that SOBs are the same up to sign. 
For $u,v\in V$, define the \emph{right mutation} 
$R_u v$ and the \emph{left mutation} $L_u v$ by  
\[
R_u v := v - [v,u) u, \qquad 
L_u v := v - [u,v) u.  
\] 
The braid group $B_N$ on $N$ strands acts on the set of SOBs 
of $N$ vectors by:
\begin{align} 
\label{eq:braidaction}
\begin{split} 
\sigma_i (v_1,\dots,v_N) & = 
(v_1,\dots,v_{i-1}, v_{i+1}, R_{v_{i+1}} v_i, v_{i+2},\dots,v_N) \\ 
\sigma_i^{-1}(v_1,\dots,v_N) & = 
(v_1,\dots,v_{i-1}, L_{v_i}v_{i+1}, v_i, v_{i+2},\dots,v_N)
\end{split} 
\end{align} 
where $\sigma_1,\dots,\sigma_{N-1}$ are generators 
of $B_N$ and satisfy the braid relation 
\begin{align*}
% \label{eq:braid_relation} 
\begin{split} 
& \sigma_i \sigma_{i+1} \sigma_i = \sigma_{i+1} \sigma_i \sigma_{i+1}  
\\
& \sigma_i \sigma_j = \sigma_j \sigma_i \quad \text{if $|i-j|\ge 2$}. 
\end{split} 
\end{align*}

\subsection{Marked reflection system} 
\label{subsec:MRS}
A \emph{marked reflection system} (MRS) of phase $\phi \in \R$ is a tuple 
$(V,[\cdot,\cdot), \{v_1,\dots,v_N\}, m, e^{\iu\phi})$ 
consisting of 
\begin{itemize} 
\item a complex vector space $V$; 
\item a bilinear form $[\cdot,\cdot)$ on $V$; 
\item an unordered basis $\{v_1,\dots,v_N\}$ of $V$; 
\item a marking $m \colon \{v_1,\dots,v_N\}\to \C$; 
we write $u_i = m(v_i)$ and $\u = \{u_1,\dots,u_N\}$; 
\item $e^{\iu\phi}\in S^1$; $\phi\in \R$ is called a phase  
\end{itemize} 
such that the vectors $v_i$ are semiorthonormal in the sense 
that: 
\begin{equation} 
\label{eq:semiorth_MRS} 
[v_i,v_j) = \begin{cases} 
1 & \text{if } i=j; \\ 
0 & \text{if } i\neq j \text{ and } 
\im(e^{-\iu\phi} u_i) \le \im(e^{-\iu \phi} u_j). 
\end{cases} 
\end{equation}  
When the basis is ordered appropriately, 
$(V,[\cdot,\cdot), v_1,\dots,v_N)$ gives an SOB. 
We say that $\phi\in \R$ (or $e^{\iu\phi} \in S^1$) 
is \emph{admissible} for a multiset $\u=\{u_1,\dots,u_N\}$ in $\C$ if  
$e^{\iu\phi}$ is not parallel to any non-zero difference $u_i-u_j$. 
An MRS is said to be \emph{admissible} if the phase $\phi$ is 
admissible for $\u$. 
By a small perturbation of the marking, one can always make 
an MRS admissible. 
%, that is, for a given MRS $(V,[\cdot,\cdot), 
%\{v_1,\dots,v_N\}, m, e^{\iu\phi})$ and for every positive 
%number $\epsilon>0$, we can find an admissible MRS 
%$(V,[\cdot,\cdot), \{v_1,\dots,v_N\}, m', e^{\iu\phi})$ 
%with the same $(V,[\cdot,\cdot),\{v_1,\dots,v_N\}, e^{\iu\phi})$ 
% such that $|m'(v_i) - m(v_i)|<\epsilon$ for $i=1,\dots,N$. 
The circle $S^1$ acts on the space of all MRSs 
by rotating simultaneously the phase and all markings:
$\phi \mapsto \alpha+\phi$, $u_i \mapsto  e^{\iu\alpha} u_i$. 

We write $h_\phi \colon \C \to \R$ for the $\R$-linear function 
$h_\phi(z) = \im(e^{-\iu\phi} z)$. 
 
\begin{remark}[\cite{Golyshev01}]
% Some part we just skipped (attic/progon.tex), but part of it I put back as a remark.
The data of admissible marked reflection system may be used to produce 
a polarized local system on $\C \setminus \u$ (cf.~the second 
structure connection \cite{Dubrovin98a,Man99}): 
by identifying the fiber with $V$, 
endowing it with a bilinear form $\{\cdot,\cdot\}$ 
(either with the symmetric form $\{v_1,v_2\} = [v_1,v_2) + [v_2,v_1)$
or with the skew-symmetric form $\{v_1,v_2\} = [v_1,v_2) - [v_2,v_1)$), 
choosing ``$ e^{\iu \phi}\infty$"  for the base point, 
joining it with the points $u_i$ 
with the level rays $u_i + \R_{\ge 0} e^{\iu\phi}$ as paths, 
trivializing the local system outside the union of the paths, 
and requiring that the turn around $u_i$ act 
in the monodromy representation 
by the reflection with respect to $v_i$ : $r_{v_i}(v) = v - \{v,v_i\} v_i$.
In case $u_i$ is a multiple marking all the respective reflections commute,
so the monodromy of the local system is well-defined as a product of such reflections.
The form $\{\cdot,\cdot\}$ could be degenerate,
however construction above also provides local systems with 
fiber the quotient-space $V/\Ker \{\cdot,\cdot\}$ 
(use reflections with respect to images of $v_i$ in the quotient-spaces).

When $u_1,\dots,u_N$ are pairwise distinct, 
this construction could be reversed with a little ambiguity:
if there is such a local system and the paths are given, 
we take $V$ to be a fiber, define $v_i\in V$ 
as a reflection vector of the monodromy at $u_i$ 
(which could be determined up to sign if 
$\{\cdot,\cdot\}$ is non-degenerate), 
and define $[v_i,v_i) = 1$ for all $i$, 
for $i\neq j$ put $[v_i,v_j) = 0$ if $h_\phi(u_i)<h_\phi(u_j)$
and $[v_i,v_j) = \{v_i,v_j\}$ otherwise. 
\end{remark}

\subsection{Mutation of MRSs} 
\label{subsec:MRS_mutation} 
In this section, we explain how an MRS changes by mutation 
as we vary the marking $\u$ and the phase $\phi$. 
First we give an intuitive explanation and then 
give a formal definition. 

The mutation of MRSs is parallel to the mutation of asymptotically 
exponential flat sections in \S \ref{subsec:mutation} 
(see Corollary \ref{cor:mutation}). 
For a given admissible MRS $(V,[\cdot,\cdot),
\{v_1,\dots,v_N\}, m(v_i) = u_i, e^{\iu\phi})$, we draw 
the ray $L_i= u_i + \R_{\ge 0} e^{\iu\phi}$ from 
each marking $u_i$ in the direction of $\phi$ 
to get a picture as in Figure \ref{fig:L_i}. 
We consider a variation of the parameters $\u$ and $\phi$. 
We allow that a multiple marking separates into distinct 
ones, but do \emph{not} allow that two different markings 
$u_i \neq u_j$ collide unless the corresponding vectors 
$v_i$, $v_j$ are orthogonal $[v_i,v_j) = [v_j,v_i) = 0$. 
The vectors $v_1,\dots,v_N$ stay constant 
while each $u_j$ does not hit any other rays $L_i$ with 
$u_j \neq u_i$. When $u_j$ crosses a ray $L_i$ 
from the right side of $L_i$ (resp.~the left side of $L_i$) 
towards the direction $\phi$, 
the vectors marked by $u_i$ undergo 
the right (resp.~ left) mutation. 
In the situation of Figure \ref{fig:rl_mutation}, 
if a vector $v_l$ is marked by $u_i$, it is transformed to 
\begin{align*} 
v_l' = \left(\prod_{k: u_j = u_k} R_{v_k}\right) v_l 
= v_l - \sum_{k:u_j = u_k} [v_l,v_k) v_k   
\end{align*} 
after the move. 
The vectors that are not assigned the marking $u_i$ remain the same. 
Note that the vectors with the same marking $u_j$ 
are orthogonal to each other and hence the above product of 
$R_{v_k}$ does not depend on the order.   
Note also that the vectors $v_l'$ marked by $u_i$ are 
mutually orthogonal even after the mutation. 

\begin{figure}[htbp]
\begin{center} 
\begin{picture}(300,100) 
\put(160,90){\makebox(0,0){$\bullet$}} 
\put(150,90){\makebox(0,0){$u_1$}}

\put(180,80){\makebox(0,0){$\bullet$}} 
\put(170,80){\makebox(0,0){$u_2$}}

\put(90,65){\makebox(0,0){$\bullet$}} 
\put(50,65){\makebox(0,0){$u_3=u_4=u_5$}} 

\put(150,50){\makebox(0,0){$\bullet$}} 
\put(125,50){\makebox(0,0){$u_6=u_7$}}

\put(200,35){\makebox(0,0){$\bullet$}} 
\put(190,37){\makebox(0,0){$u_8$}}

\put(120,20){\makebox(0,0){$\bullet$}} 
\put(110,20){\makebox(0,0){$u_9$}}

\put(160,90){\line(1,0){180}} 
\put(180,80){\line(1,0){160}} 
\put(90,65){\line(1,0){250}} 
\put(150,50){\line(1,0){190}} 
\put(200,35){\line(1,0){140}} 
\put(120,20){\line(1,0){220}} 

\put(345,90){{$L_1$}} 
\put(345,75){$L_2$} 
\put(345,60){$L_3=L_4=L_5$} 
\put(345,45){$L_6 = L_7$} 
\put(345,30){$L_8$} 
\put(345,15){$L_9$} 
\end{picture} 
\end{center} 
\caption{Rays $L_i$ in the admissible direction $e^{\iu\phi}=1$.} 
\label{fig:L_i}
\end{figure}
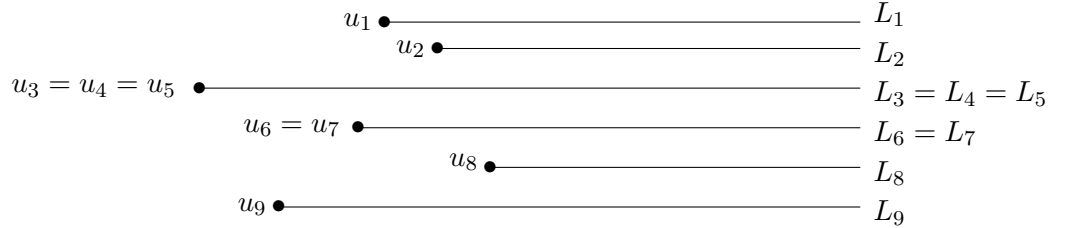

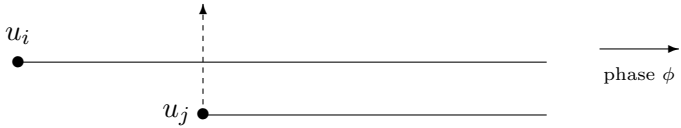
\begin{figure}[htbp]
\begin{center}
\begin{picture}(300,60) 
\put(80,30){\makebox(0,0){$\bullet$}}
\put(80,40){\makebox(0,0){$u_i$}}
\put(80,30){\line(1,0){200}}
\put(150,10){\makebox(0,0){$\bullet$}}
\put(140,10){\makebox(0,0){$u_j$}}
\put(150,10){\line(1,0){130}}

\put(150,10){\line(0,1){2}}
\put(150,14){\line(0,1){2}}
\put(150,18){\line(0,1){2}} 
\put(150,22){\line(0,1){2}} 
\put(150,26){\line(0,1){2}} 
\put(150,30){\line(0,1){2}} 
\put(150,34){\line(0,1){2}} 
\put(150,38){\line(0,1){2}} 
\put(150,42){\line(0,1){2}} 
\put(150,46){\line(0,1){2}} 
\put(150,50){\line(0,1){2}} 
\put(150,52){\vector(0,1){0}}

\put(300,35){\vector(1,0){30}} 
\put(315,25){\makebox(0,0){\tiny phase $\phi$}}
\end{picture} 
\end{center}
\caption{Right mutation} 
\label{fig:rl_mutation}
\end{figure}

\begin{remark} 
\label{rem:adm_MRS} 
Let $(V,[\cdot,\cdot),\{v_1,\dots,v_N\}, m(v_i) = u_i, e^{\iu\phi})$ 
be an admissible MRS with multiple markings. Since we have 
$[v_i,v_j) = [v_j,v_i) = 0$ if $u_i=u_j$ and $i\neq j$ 
(see \eqref{eq:semiorth_MRS}), there exists $\epsilon>0$ 
such that all the admissible MRSs of the form 
$(V,[\cdot,\cdot), \{v_1,\dots,v_N\}, m(v_i) = u_i', e^{\iu\phi})$ 
with pairwise distinct $u_1',\dots,u_N'$ 
and $|u_i - u_i'|<\epsilon$ are related to each other by mutation. 
\end{remark}

Now we give a formal and precise definition 
for the deformation of MRSs. 
We first deal with the case where $\u =(u_1,\dots,u_N)$ 
are pairwise distinct, i.e.~give an element of 
the configuration space $C_N(\C)$ \eqref{eq:config}. 
We construct a local system $\L=\L(V,[\cdot,\cdot))$ of MRSs with 
fixed $(V,[\cdot,\cdot))$ over $C_N(\C) \times S^1$. 
Then we extend the \'etal\'e space (sheaf space) of $\L$ 
to a branched cover $\frM$ over $S^N(\C)\times S^1$ 
(where $S^N(\C)$ is the symmetric power of $\C$); 
this amounts to pushing forward $\L$ 
to $S^N(\C)\times S^1$ as an orbi-sheaf. 
We will see that the fiber $\frM_{(\u,e^{\iu\phi})}$ at 
$(\u,e^{\iu\phi}) \in S^N(\C)\times S^1$ contains the set of 
all MRSs on $(V,[\cdot,\cdot))$ with marking $\u$ and phase $\phi$. 
Then we define a deformation of 
MRSs as a continuous path in $\frM$. 

\medskip 
\noindent 
{\bf Local system on the `distinct' locus.} 
Consider the codimension-one stratum 
$W\subset C_N(\C)\times S^1$ 
consisting of non-admissible parameters $(\u,e^{\iu\phi})$: 
\[
W = \{ (\u, e^{\iu\phi}) \in C_N(\C) \times S^1 : 
h_\phi(u_i) = h_\phi(u_j) \text{ for some $i\neq j$} \}. 
\]
Note that the $S^1$-action $(\u,e^{\iu\phi}) \mapsto 
(e^{\iu\alpha} \u, e^{\iu(\alpha+\phi)})$ on $C_N(\C) 
\times S^1$ preserves $W$. 
The complement $(C_N(\C) \times S^1)\setminus W$ 
is a connected open subset. 
Fix a vector space $V$ with a bilinear form $[\cdot,\cdot)$. 
We define $\L^\circ = \L^\circ(V,[\cdot,\cdot))$ to be the trivial local system on 
the open subset $(C_N(\C) \times S^1)\setminus W$ 
whose fiber at $(\u,e^{\iu\phi})
\in (C_N(\C)\times S^1)\setminus W$ is given by 
\[
\L^\circ_{(\u,e^{\iu\phi})} = \left \{ \text{MRSs on $(V,[\cdot,\cdot))$ 
with marking $\u$ and phase $\phi$} \right\} 
\] 
where by ``MRS on $(V,[\cdot,\cdot))$ with marking $\u$ and 
phase $\phi$'' we mean an MRS of the form 
$(V,[\cdot,\cdot),\{v_1,\dots,v_N\}, m, e^{\iu\phi})$ 
such that $\u = \{m(v_1),\dots,m(v_N)\}$. 
Note that the fiber $\L^\circ_{(\u,e^{\iu\phi})}$ 
is canonically identified with the set of SOBs on $V$ 
via the ordering given by $h_\phi$, i.e.~we can order 
$v_1,\dots,v_N$ in such a way that $h_\phi(u_1) > h_\phi(u_2) 
> \cdots > h_\phi(u_N)$ with $u_i = m(v_i)$. 
This identification defines the structure of a trivial local system 
on $\L^\circ$. 
Next we extend $\L^\circ$ to a local system $\L=\L(V,[\cdot,\cdot))$ on 
$C_N(\C)\times S^1$. 
Let $\jmath \colon (C_N(\C) \times S^1)\setminus W \to 
C_N(\C)\times S^1$ denote the inclusion. 
An extension of $\L^\circ$ is given by an action of 
$\pi_1(C_N(\C)\times S^1)$ on the set of SOBs of $V$ 
which is trivial on $\jmath_*\pi_1((C_N(\C)\times S^1)\setminus W)$. 
Via the $S^1$-action on $C_N(\C)\times S^1$, we have 
the isomorphism 
\[
\{(u_1,\dots,u_N) \in 
\C^N: \im(u_1) > \im(u_2) > \cdots > \im(u_N)\} \times S^1 
\cong 
(C_N(\C)\times S^1) \setminus W  
\]
which sends $((u_1,\dots,u_N),e^{\iu\phi})$ to 
$(\{e^{\iu\phi}u_1,\dots,e^{\iu\phi}u_N\}, e^{\iu\phi})$. 
Since the first factor of the left-hand side is contractible, 
we have $\pi_1((C_N(\C) \times S^1)\setminus W) 
\cong \Z$ and the image 
of $\jmath_* \colon \pi_1((C_N(\C) \times S^1)\setminus W) 
\to \pi_1(C_N(\C) \times S^1)$ is generated by the 
class of an $S^1$-orbit. 
Since $\pi_1(C_N(\C)\times S^1)$ is the direct product 
of $\pi_1(C_N(\C)\times\{1\})$ and $\pi_1$ of an $S^1$-orbit, 
an extension of the local system $\L^\circ$ to 
$C_N(\C) \times S^1$ is given by the action of 
the braid group $B_N \cong \pi_1(C_N(\C)\times \{1\})$ on the set of 
SOBs of $V$. 
We use the $B_N$-action 
on SOBs described in \S\ref{subsec:sobs} to define $\L$. 
More precisely, choosing 
\[
(\u^\circ=\{u_1 = (N-1) \iu, u_2 = (N-2) \iu, \dots, u_N = 0\}, 
e^{\iu\phi} = 1)
\]
as a base point of $(C_N(\C)\times S^1)\setminus W$, 
we define the parallel transport along the closed path 
in Figure \ref{fig:braid} (based at $(\u^\circ,1)$) 
as the action of $\sigma_i\in B_N$ on SOBs given in 
\eqref{eq:braidaction}. 
This defines the local system $\L$ of sets 
over $C_N(\C)\times S^1$. 

\begin{figure}[htbp]
\begin{center} 
\includegraphics[bb=205 592 501 713]{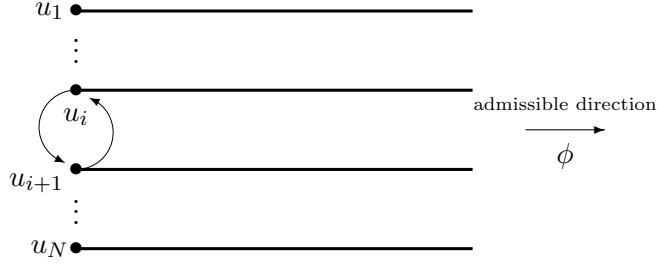} 
\end{center} 
\caption{Braid $\sigma_i$}
\label{fig:braid} 
\end{figure} 

\medskip 
\noindent 
{\bf Extension to the `non-distinct' locus.} 
Consider the following diagram: 
\[
\begin{CD} 
C_N^{\rm ord}(\C)\times S^1 @>{\iota}>> \C^N \times S^1\\ 
@V{p}VV  @V{q}VV\\ 
C_N(\C) \times S^1 @>{\overline{\iota}}>> S^N(\C) \times S^1 
\end{CD} 
\]
where $C_N^{\rm ord}(\C) = \{(u_1,\dots,u_N) \in \C^N: 
\text{$u_i \neq u_j$ for all $i\neq j$}\}$ is the configuration 
space of ordered distinct $N$ points on $\C$, 
$S^N(\C) = \C^N/\frS_N$ is the symmetric product of $\C$, 
$\iota,\overline{\iota}$ are the natural 
inclusions and $p,q$ are the natural projections.  
The local system $p^*\L$ on $C_N^{\rm ord}(\C)\times S^1$ can be extended 
to a constructible sheaf $\iota_* p^*\L$ on $\C^N\times S^1$. 
Since $p$ is an $\frS_N$-covering, $p^*\L$ 
and $\iota_*p^*\L$ are naturally 
$\frS_N$-equivariant. 
We topologize the space of MRSs on $(V,[\cdot,\cdot))$ using 
the \'etal\'e space (also known as the ``sheaf space'') 
of the sheaf $\iota_*p^* \L$. 
We define: 
\[
\frM := \tfrM/\frS_N, \quad 
\tfrM := \text{the \'etal\'e space of the sheaf $\iota_*p^*\L$}.   
\]
As we will see below, $\frM$ is a smooth manifold 
(which is non-separable if $N\ge 2$) 
and is a branched covering of $S^N(\C) \times S^1$. 
We will also see that $\frM$ contains the set of all MRSs on 
$(V,[\cdot,\cdot))$ as an open dense subset. 
We say that two MRSs on $(V,[\cdot,\cdot))$ 
are \emph{related by mutation} (or \emph{mutation-equivalent}) 
if they can be connected by a continuous path in $\frM$. 

\begin{remark} 
Using the language of orbifolds, we can regard $\frM$ as the \'etal\'e 
space of the orbi-sheaf $\overline{\iota}_*\L$ on $S^N(\C)$. 
\end{remark} 

Let $N_\bullet = (N_1\ge N_2\ge \cdots \ge N_k)$ be a 
partition of $N$ (i.e.~$N = \sum_{i=1}^k N_i$). 
Define a locally closed subset $S(N_\bullet)\subset S^N(\C)$ 
to be the set of multisets $\u$ of the form  
\begin{equation} 
\label{eq:u_nondistinct}
\u = 
\Big(\overbrace{\lambda_1,\dots,\lambda_1}^{N_1}, 
\overbrace{\lambda_2,\dots,\lambda_2}^{N_2},  
\cdots, \overbrace{\lambda_k,\dots,\lambda_k}^{N_k} \Big)
\end{equation} 
for some distinct $k$ points $\lambda_1,\dots,\lambda_k \in \C$. 
We have  
\[
S^N(\C) = \bigsqcup_{N_\bullet : \text{partition of $N$}}  S(N_\bullet). 
\]
Define $\tS(N_\bullet) := q^{-1}(S(N_\bullet)) \subset \C^N$. 
We have the following proposition. 
\begin{proposition} 
\label{prop:constructible} 
Let $N_\bullet = (N_1 \ge N_2 \ge \cdots \ge N_k)$ be a  
partition of $N$. 
The restriction of $\iota_*p^*\L$ to the stratum 
$\tS(N_\bullet)\times S^1$ is a local system whose fiber 
is isomorphic to 
the set of ``block semiorthonormal'' bases of $(V,[\cdot,\cdot))$, 
i.e.~bases $(v_{i,j})_{1\le i\le k, 1\le j\le N_i}$ of $V$ 
such that 
\begin{equation} 
\label{eq:block_SOB} 
[v_{i_1,j_1}, v_{i_2,j_2}) =
\begin{cases} 
0 & \text{if $i_1\ge i_2$ and $(i_1,j_1) \neq (i_2,j_2)$;} \\ 
1 & \text{if $(i_1,j_1) = (i_2,j_2)$}. 
\end{cases}  
\end{equation} 
The stabilizer $\Stab(\u) = \frS_{N_1}\times \cdots \times \frS_{N_k}$ 
of the $\frS_N$-action on $\C^N$ at $\u\in \tS(N_\bullet)$ acts on the stalk 
$(\iota^*p_*\L)_{(\u,e^{\iu\phi})}$ by permutation of 
the vectors $(v_{i,j})$ in the same block. 
\end{proposition} 
\begin{proof} 
Let $\u\in \tS(N_\bullet)$ be an element 
of the form \eqref{eq:u_nondistinct}. 
Let $D_i\subset \C$, $i=1,\dots,k$ be pairwise disjoint open discs 
such that $\lambda_i \in D_i$. 
We consider a $\Stab(\u)$-invariant open set $O$ 
containing $(\u,e^{\iu\phi})$ given by: 
\begin{equation} 
\label{eq:openset_O} 
O = D_1^{N_1} \times \cdots \times D_k ^{N_k} 
\times I \subset \C^N \times S^1 
\end{equation} 
where $I\subset S^1$ is an interval containing $e^{\iu\phi}$.  
The stalk $(\iota_*p^*\L)_{(\u,e^{\iu\phi})}$ can be 
identified with the space of sections of $p^*\L$ over 
\[
O \cap (C_N^{\rm ord}(\C) \times S^1) = C^{\rm ord}_{N_1}(D_1) 
\times \cdots \times C^{\rm ord}_{N_k}(D_k) \times I 
\]
where $C_{N_i}^{\rm ord}(D_i)$ denotes 
the set of ordered distinct $N_i$ points on $D_i$. 
The fundamental group of $O \cap (C_N^{\rm ord}(\C)\times S^1)$ is 
the product $P(N_\bullet):=P_{N_1} \times \cdots \times P_{N_k}$ of the 
pure braid groups $P_{N_i} \cong 
\pi_1(C_{N_i}^{\rm ord}(D_i))$. 
We choose a reference point 
\begin{equation} 
\label{eq:reference_point} 
\left(\bx =(x_{i,j})_{1\le i\le k, 1\le j\le N_i}, e^{\iu\theta}
\right)\in O \cap (C_N^{\rm ord}(\C) \times S^1) 
\end{equation}
such that 
\begin{itemize} 
\item $(x_{i,1},\dots,x_{i,N_i}) 
\in C_{N_i}^{\rm ord}(D_i)$ 
for $1\le i\le k$, 
$e^{\iu\theta} \in I$; 
\item $h_\theta(x_{i,1}) > h_\theta(x_{i,2}) > 
\cdots > h_\theta(x_{i,N_i})$ for all $1\le i\le k$;  
\item the intervals $[h_\theta(x_{i,1}),h_\theta(x_{i,N_i})]$, 
$i=1,\dots,k$ are mutually disjoint. 
\end{itemize} 
Renumbering the indices $i$, we may further assume that 
$h_\theta(x_{1,1})>h_\theta(x_{2,1}) > \cdots > h_\theta(x_{k,1})$ 
(then $N_1\ge \cdots \ge N_k$ may no longer hold). 
With this choice, we can identify $(\iota_*p^*\L)_{(\u,e^{\iu\phi})}$ 
with $(p^*\L_{(\bx,e^{\iu\theta})})^{P(N_\bullet)}$, i.e.~
the set of MRSs on $(V,[\cdot,\cdot))$ with marking $\bx$ and 
phase $\theta$ which are $P(N_\bullet)$-invariant. 
The stalks $(\iota_*p^*\L)_{(\u',e^{\iu\phi'})}$ 
at other points $(\u', e^{\iu\phi'})$ in $O \cap (\tS(N_\bullet)\times S^1)$ 
can be also identified with the same set, and therefore $\iota_*p^*\L$ is locally 
constant along the stratum $\tS(N_\bullet) \times S^1$. 
On the other hand, it is easy to check that a semiorthonormal collection  
$(v_1,\dots,v_a)$ is invariant under the pure braid 
group $P_a\subset B_a$ if and only if $v_1,\dots,v_a$ are mutually 
orthogonal, i.e.~$[v_i,v_j) = [v_j,v_i) = 0$ for all $1\le i,j\le a$. 
Therefore $P(N_\bullet)$-invariant MRSs with marking $\bx$ and 
phase $\theta$ corresponds to block SOBs $(v_{i,j})$ in the proposition. 

To see the last statement, note that $\omega = (\omega_1,\dots,\omega_k) 
\in \frS_{N_1}\times \cdots \times \frS_{N_k}$ 
sends a block SOB $(v_{i,j})$ defining an element of 
$(p^*\L_{(\bx,e^{\iu\theta})})^{P(N_\bullet)}$ 
to a block SOB $(v_{i,\omega_i(j)})$ defining an element of 
$(p^*\L_{(\omega\cdot \bx, e^{\iu\theta})})^{P(N_\bullet)}$ 
where $\omega \cdot \bx = ( x_{i,\omega_i(j)})_{i,j}$. 
Since $(v_{i,1},\dots,v_{i,N_i})$ are mutually orthogonal, the 
parallel transport from $p^*\L_{(\omega\cdot \bx, e^{\iu\theta})}$ 
to $p^*\L_{(\bx,e^{\iu\theta})}$ sends the SOB 
$(v_{i,\omega_i(j)})$ to the SOB $(v_{i,\omega_i(j)})$, and the 
conclusion follows. 
\end{proof} 

The above proposition shows that the $\frS_N$-action on 
the \'etal\'e space $\tfrM$ of $\iota_*p^*\L$ is free, and 
thus the quotient $\frM = \tfrM/\frS_N$ is a smooth manifold. 
The natural map $\frM \to S^N(\C) \times S^1$ is a covering map 
restricted to each stratum $S(N_\bullet) \times S^1$. 
We show that MRSs on $(V,[\cdot,\cdot))$ form an open 
dense subset of $\frM$. 

\begin{proposition} 
For $(\u,e^{\iu\phi})\in S^N(\C)\times S^1$, 
we have a canonical inclusion 
\[
\left\{\text{MRSs on $(V,[\cdot,\cdot))$ with marking $\u$ and 
phase $\phi$} \right\} \subset \frM_{(\u,e^{\iu\phi})}  
\]
where $\frM_{(\u,e^{\iu\phi})}$ denotes the fiber of 
$\frM \to S^N(\C) \times S^1$ at $(\u,e^{\iu\phi})$. 
This inclusion is an equality if the phase $\phi$ is admissible for $\u$. 
\end{proposition} 
\begin{proof} 
The fiber of $\frM$ at $(\u,e^{\iu\phi}) \in S^N(\C) \times S^1$ 
is given by $(\iota_*p^*\L)_{(\u,e^{\iu\phi})}/\Stab(\u)$. 
Here, by abuse of notation, $\u$ denotes both a point of $S^N(\C)$ 
and its lift in $\C^N$. 
Suppose that $\u\in S(N_\bullet)$ is of the form \eqref{eq:u_nondistinct}.
Let $(V,[\cdot,\cdot),\{v_{i,j}\}_{1\le i\le k, 1\le j\le N_i})$ be an SOB 
such that the marking $m(v_{i,j}) = \lambda_i$ defines an MRS 
of phase $\phi$. 
We first construct an element of the fiber $\frM_{(\u,e^{\iu\phi})}$ 
from these data. 
Take a sufficiently small 
open neighbourhood $O\subset \C^N \times S^1$ of $(\u,e^{\iu\phi})$ 
of the form \eqref{eq:openset_O} such that, 
for any $(\bx =(x_{i,j})_{1\le i\le k, 1\le j\le N_i},e^{\iu\theta})
\in O \cap (C_N^{\rm ord}(\C) \times S^1)$, 
\[
h_\phi(\lambda_{i_1}) > h_\phi(\lambda_{i_2}) 
\Longrightarrow 
\forall j_1, \forall j_2, \ h_\theta(x_{i_1,j_1}) > h_\theta(x_{i_2,j_2}).   
\]
For $(\bx,e^{\iu\theta}) \in O \cap (C_N^{\rm ord}(\C)\times S^1 
\setminus p^{-1}(W))$, 
we define an MRS $\MRS(\bx,\theta)$ by 
\[
\MRS(\bx,\theta) := 
\left( V,[\cdot,\cdot), \{v_{i,j}\}, m(v_{i,j}) = x_{i,j}, 
e^{\iu\theta} \right).  
\]
This is indeed an MRS because $v_{i,j}$'s with the same value of 
$h_\phi(\lambda_i)$ are orthogonal to each other. 
For the same reason, the family of MRSs $\MRS(\bx,\theta)$ parametrized by 
$(\bx,e^{\iu\theta}) \in O \cap 
(C_N^{\rm ord}(\C)\times S^1 \setminus p^{-1}(W))$ 
are related to each other by mutation, and extends to a locally constant 
section of $p^*\L$ over $O\cap (C_N^{\rm ord}(\C)\times S^1)$. 
Therefore, they define an element of the stalk 
$(\iota_*p^*\L)_{(\u,e^{\iu\phi})}$. 
This construction depends on an auxiliary choice of the 
ordering of each orthogonal collection $\{v_{i,1},\dots,v_{i,N_i}\}$, 
and changing this choice yields elements in a single $\Stab(\u)$-orbit 
in $(\iota_*p^*\L)_{(\u,e^{\iu\phi})}$. Thus we get an injection 
from the set of MRSs with marking $\u$ and phase $\phi$ to 
$\frM_{(\u,e^{\iu\phi})} \cong (\iota_*p^*\L)_{(\u,e^{\iu\phi})}/\Stab(\u)$. 
When $\phi$ is admissible for $\u$, every point $(\bx, e^{\iu\theta})$ 
in $O\cap (C^{\rm ord}_N(\C)\times S^1\setminus p^{-1}(W))$ 
satisfies the conditions for the reference point \eqref{eq:reference_point}
in the proof of Proposition \ref{prop:constructible}. Thus the argument 
and the result there show that this construction is an isomorphism. 
\end{proof}

\subsection{Exceptional collections} 
\label{subsec:exceptional} 
In this section we sketch only the necessary definitions 
of (full) exceptional collections
in triangulated categories and braid group action on them. 
For details we refer the reader to \cite{BP} or to the survey \cite{GK}.

Recall that an object $E$ in a triangulated category $\T$ over $\k$ 
is called exceptional if $\Hom(E,E) = \k$ and 
$\Hom(E,E[j]) = 0$ for all $j\neq 0$. 
An ordered tuple of exceptional objects $E_1,\dots,E_N$ is called 
an exceptional collection
if $\Hom(E_i,E_j[k]) = 0$ for all $i,j,k$ with $i>j$ and $k\in \Z$. 
If $E_1,\dots,E_N$ is an exceptional collection and $k_1,\dots,k_N$ 
is a set of integers 
then $E_1[k_1],\dots,E_N[k_N]$ is also an exceptional collection, 
in this case we say that two collections coincide up to shift. 
An exceptional object $E$ gives a pair of functors
called \emph{left mutation $L_E$} and \emph{right mutation $R_E$} 
which fit into the distinguished triangles\footnote
{What we write as $L_EF$, $R_EF$ here are denoted respectively 
by $L_E F[1]$, $R_E F[-1]$ in \cite{BP,GK}. }:  
\begin{align*} 
L_E F[-1] \to \Hom^\bullet(E,F) \otimes E \to F \to L_E F, \\ 
R_E F  \to F \to \Hom^\bullet(F,E)^* \otimes E \to R_E F[1]. 
\end{align*}  
The set of exceptional collections of $N$ objects in 
$\T$ admits an action of the braid group $B_N$ 
in a way parallel to the $B_N$-action \eqref{eq:braidaction}
on SOBs, i.e.~the generator 
$\sigma_i$ replaces $\{E_i,E_{i+1}\}$ 
with $\{E_{i+1}, R_{E_{i+1}} E_i\}$ 
and $\sigma_i^{-1}$ replaces $\{E_i, E_{i+1}\}$ with 
$\{L_{E_i}E_{i+1}, E_i\}$. 

The subcategory in $\T$ generated by $E_i$ 
(minimal full triangulated subcategory that contains all $E_i$) 
stays invariant under mutations.
An exceptional collection is called \emph{full} if it generates the 
whole category $\T$; clearly braid group and shift action 
respects fullness. 
When a triangulated category $\T$ has a full exceptional collection 
$E_1,\dots,E_N$, 
the Grothendieck $K$-group $K^0(\T)$ 
is the abelian group freely generated by the classes $[E_i]$.
The Gram matrix $G_{ij} = \chi(E_i,E_j)$ of the
Euler pairing $\chi(E,F) = \sum (-1)^i \dim \Hom(E,F[i])$ 
is uni-uppertriangular in the basis $\{[E_i]\}$.
Thus any exceptional collection $E_1,\dots,E_N$ produces an SOB
$(K^0(\T)\otimes \k; \chi(\cdot,\cdot); [E_1],\dots,[E_n])$.
The action of the braid group $B_N$ and shifts $E_i \to E_i [k_i]$ 
described above 
descends to the action of $B_N$ and change of signs on 
the $K$-group described in \S \ref{subsec:sobs}.

\subsection{Asymptotic basis and MRS of a Fano manifold} 
\label{subsec:MRS_Fano} 
% Here we resume the discussion in \S \ref{subsec:isomonodromy} and 
% \S \ref{subsec:repeated}. 
Let $F$ be a Fano manifold such that $\star_\tau$ is convergent 
and semisimple for some $\tau\in H^\udot(F)$. 
Choose a phase $\phi\in\R$ that is admissible for the 
spectrum of $(E \star_\tau)$. 
By Proposition \ref{prop:repeated_eigenvalues}, 
we have a basis $y_1(\tau,z),\dots,y_N(\tau,z)$ 
of asymptotically exponential flat sections for the 
quantum connection $\nabla|_\tau$ defined over the 
sector $\phi - \frac{\pi}{2} -\epsilon < \arg z < \phi + \frac{\pi}{2} 
+ \epsilon$ for some $\epsilon>0$. 
They are canonically defined up to sign and ordering. 
The basis is marked by the spectrum $\u = \{u_1,\dots,u_N\}$ 
of $(E\star_\tau)$. 
We can naturally identify $y_i(\tau,z)$ on the universal cover of 
the $z$-plane $\C^\times$. 
By $\nabla$-parallel transportation of $y_i(\tau,z)$ to 
$\tau=0$, $\arg z =0$ and $|z|\gg 1$, 
we can identify the basis $y_1,\dots,y_N$ of 
flat sections with a basis $A_1,\dots,A_N$ of $H^\udot(F)$ 
via the fundamental solution $S(z) z^{-\mu} z^\rho$ in 
Proposition \ref{prop:fundsol}, i.e.~
\[
y_i(\tau,z) \Big|_{\substack{
\text{\tiny parallel transport} \\ 
\text{\tiny to $\tau=0$, $\arg z =0$}}} =  \Phi(A_i )(z) := 
(2\pi)^{-\frac{\dim F}{2}}
S(z) z^{-\mu}z^\rho A_i
\]
where $\Phi$ was given in \eqref{eq:coh_framing}. 
The basis $\{A_1,\dots,A_N\}$ is defined up to sign and 
ordering and is semiorthonormal with respect to the pairing 
$[\cdot,\cdot)$ in \eqref{eq:pairing_H}--\eqref{eq:[)_coh} 
by Proposition \ref{prop:Stokes}. 
We call it the \emph{asymptotic basis} of $F$ at $\tau$ 
with respect to the phase $\phi$. 

The \emph{marked reflection system} (MRS) of $F$ 
is defined to be the tuple:   
\[ 
\MRS(F,\tau, \phi) := (H^\udot(F), [\cdot,\cdot), \{A_1,\dots,A_N\}, 
m \colon A_i \mapsto u_i, e^{\iu\phi}). 
\]
The marking $m$ assigns to $A_i$ the corresponding 
eigenvalue $u_i$ of $(E\star_\tau)$. 
It is defined up to sign. 
Note that $\MRS(F,\tau,\phi)$ and $\MRS(F,\tau,\phi+2\pi)$ 
are not necessarily the same: they differ by the monodromy around $z=0$. 
By the discussion in \S \ref{subsec:mutation}, when $(\tau,\phi)$ varies, 
the corresponding MRSs change by mutation as described 
in the previous section \S \ref{subsec:MRS_mutation}. 

\begin{remark} 
When $\star_0$ is semisimple, we have the two MRSs 
$\MRS(F, 0, \pm \epsilon)$ 
defined at the canonical position $\tau=0$ 
for a sufficiently small $\epsilon>0$. 
When $\phi=0$ is admissible for the spectrum of $(c_1(F)\star_0)$, 
the two MRSs coincide. 
\end{remark}

% \pph 
\subsection{Dubrovin's (original) conjecture and Gamma conjecture II} 
\label{subsec:GammaII}
We recall \emph{Dubrovin's Conjecture} (see also Remark \ref{rem:Dubrovin}). 
Let $F$ be a Fano manifold. 
Dubrovin \cite[Conjecture 4.2.2]{Dubrovin98} conjectured: 
\begin{enumerate}
\item[(1)] the quantum cohomology of $F$ is semisimple if and only if 
$\D^b_{\rm coh}(F)$ admits a full exceptional collection $E_1,\dots,E_N$; 
\end{enumerate} 
and if the quantum cohomology of $F$ is semisimple, 
there exists a full exceptional collection $E_1,\dots,E_N$ of 
$\D^b_{\rm coh}(F)$ such that: 
\begin{enumerate}%[resume] 
\item[(2)] the Stokes matrix $S=(S_{ij})$ 
is given by $S_{ij}=\chi(E_i,E_j)$, $i, \, j=1, \dots, N$; 
\item[(3)]  
the central connection matrix has the form $C=C' C''$
when the columns of $C''$ are the components of 
$\Ch(E_j) \in H^\udot(F)$ 
and $C':H^\udot(F)\to H^\udot(F)$ 
is some operator satisfying 
$C'(c_1(F)a) = c_1(F) C'(a)$ for any $a\in H^\udot(F)$.
\end{enumerate}
The central connection matrix $C$ above is given by  
\[
C = \frac{1}{(2\pi)^{\dim F/2}}
\begin{pmatrix} 
\vert &  & \vert \\  
A_1 & \cdots& A_N  \\
\vert &  & \vert \\
\end{pmatrix} 
\]
where the column vectors $A_1,\dots,A_N\in H^\udot(F)$ 
are an asymptotic basis in \S \ref{subsec:MRS_Fano}. 
Gamma Conjecture II makes precise the operator $C'$ 
in Part (3) of Dubrovin's conjecture. 

\begin{conjecture}[Gamma Conjecture II] 
\label{conj:conj-g2}
Let $F$ be a Fano manifold such that the quantum product 
$\star_\tau$ is convergent and semisimple for some $\tau \in H^\udot(F)$ 
\emph{and} that $\D^b_{\rm coh}(F)$ has a full exceptional collection. 
There exists a full exceptional collection $\{E_1,\dots,E_N\}$ 
such that the asymptotic basis $A_1,\dots,A_N$ 
of $F$ at $\tau$ with respect to an admissible phase $\phi$ 
(see \S \ref{subsec:MRS_Fano}) 
is given by: 
\begin{equation}
\label{eq:Gamma_basis}
A_i = \Gg_F \Ch(E_i). 
\end{equation} 
In other words, the operator $C'$ in Part (3) of Dubrovin's conjecture 
is given by $C'(\alpha) =(2\pi)^{-\dim F/2} \Gg_F \alpha$. 
\end{conjecture}

\begin{remark}
\label{rem:compatibility_mutation} 
The validity of Gamma Conjecture II does not depend on the choice 
of $\tau$ and $\phi$. Under deformation of $\tau$ and $\phi$, 
the marked reflection system $\MRS(F,\tau,\phi)$ changes by 
mutations (i.e.~defines a continuous path in the space $\frM$ 
of MRSs) 
as described in \S \ref{subsec:mutation} and \S \ref{subsec:MRS_mutation}. 
On the other hand, the braid group acts on full exceptional 
collections by mutations (\S \ref{subsec:exceptional}). 
These mutations are compatible in the $K$-group 
because the Euler pairing $\chi(\cdot,\cdot)$ 
is identified with the pairing $[\cdot,\cdot)$ \eqref{eq:[)_coh} 
by Lemma \ref{lem:Gamma}. 
\end{remark} 

\begin{remark} 
If $F$ satisfies Gamma Conjecture II, the quantum differential 
equations of $F$ specify a distinguished mutation-equivalence class of 
(the images in the $K$-group of) 
full exceptional collections of $F$ 
given by an asymptotic basis.  
It is not known in general whether any two given full exceptional collections 
are connected by mutations and shifts.   
\end{remark} 

\begin{remark} 
\label{rem:Dubrovin} 
While preparing this paper, we are informed that Dubrovin 
\cite{Dubrovin:Strasbourg} himself formulated the same conjecture 
as Gamma Conjecture II. 
\end{remark} 

\begin{proposition} 
Gamma Conjecture II implies Parts (2) and (3) of Dubrovin's conjecture. 
\end{proposition} 
\begin{proof} 
Recall from Proposition \ref{prop:Stokes} that 
the Stokes matrix is given by the pairing $[y_i, y_j)$ 
of asymptotically exponential flat sections. 
Under Gamma Conjecture II, we have 
$[y_i,y_j) = [A_i,A_j) = \chi(E_i,E_j)$, 
where we used the fact that the pairing $[\cdot,\cdot)$ is 
identified with the Euler pairing (Lemma \ref{lem:Gamma}). 
\end{proof} 
% By Riemann--Roch-Hirzebruch formula 
% \[ \chi(E_i,E_j) = \int_{[F]} \ch(E_i^\vee) \cup \ch(E_j) \cup td_F . \]
% Note that 
% \[\Gamma(1+x) \Gamma(1-x) = \frac{\pi x}{\sin \pi x} 
% = \frac{2 \pi i x}{e^{\pi i x}-e^{-\pi i x}}
% = e^{-\pi i x} \frac{2 \pi i x}{1 -  e^{-2 \pi i x}}\]
% hence $\gamma(x) \gamma(-x) = e^{-\frac{x}{2}} \frac{x}{1-e^{-x}}$,
% so 
% \[ td_F = e^{\frac{c_1(X)}{2}} \gamma_F \cup (-1)^{Gr} \gamma_F \]

A relationship between Gamma Conjectures I and II is explained 
as follows. 

\begin{proposition} 
\label{prop:Gamma_I_II} 
Suppose that a Fano manifold $F$ satisfies Property $\O$ 
and that the quantum product $\star_0$ is semisimple. 
Suppose also that $F$ satisfies Gamma Conjecture II. 
If the exceptional object corresponding to the eigenvalue 
$T$ and an admissible phase $\phi$ with $|\phi|<\pi/2$ is 
the structure sheaf $\O$ (or its shift), then $F$ satisfies Gamma Conjecture I. 
\end{proposition} 
\begin{proof} 
Under the assumptions, the flat section $\Phi(\Gg_F)$ corresponding to 
$\O$ satisfies the asymptotics $\Phi(\Gg_F)(z) \sim e^{-T/z}\psi_0$ 
for some $\psi_0 \in E(T)$, as $z\to 0$ in the sector 
$\phi -\frac{\pi}{2} -\epsilon 
< \arg z < \phi + \frac{\pi}{2} + \epsilon$ 
(see Proposition \ref{prop:repeated_eigenvalues}). 
Therefore $\Phi(\Gg_F)$ belongs to the space $\AA$
\eqref{eq:smallest_asymptotics}. Gamma Conjecture I 
holds (see Conjecture \ref{conj:GammaI}, Theorem \ref{thm:asymptotic_class}). 
\end{proof} 

\subsection{Quantum cohomology central charges} 
Suppose that $F$ satisfies Gamma Conjecture II 
and that $\star_0$ is semisimple. 
Let $\psi_1,\dots,\psi_N$ be the idempotent basis 
for $\star_0$ and $u_1,\dots,u_N$ be the corresponding 
eigenvalues of $(c_1(F)\star_0)$  i.e.~$c_1(F) \star_0 \psi_i = u_i \psi_i$. 
Let $\phi\in \R$ be an admissible phase for $\{u_1,\dots,u_N\}$ 
and let $E_1,\dots,E_N$ be an exceptional collection corresponding 
(via \eqref{eq:Gamma_basis}) to the asymptotic basis $A_1,\dots,A_N$ 
at $\tau=0$ with respect to the phase $\phi$.  
Then the quantum cohomology central charges of $E_1,\dots,E_N$ 
(see \eqref{eq:centralcharge}, \eqref{eq:centralcharge_J}) 
have the following asymptotics: 
\[
Z(E_i) \sim \sqrt{(\psi_i,\psi_i)_F} (2\pi z)^{\dim F/2} e^{-u_i/z} 
\]
as $z\to 0$ in the sector $\phi-\frac{\pi}{2} -\epsilon <\arg z < 
\phi+\frac{\pi}{2} + \epsilon$ for a sufficiently small $\epsilon>0$. 
This follows from the definition \eqref{eq:centralcharge} of $Z(E_i)$ 
and the asymptotics $\Phi(A_i)(z) \sim e^{-u_i/z} \psi_i/\sqrt{(\psi_i,\psi_i)_F}$. 
See also \S \ref{subsec:centralcharges}.

\section{Gamma conjectures for projective spaces} 
\label{sec:Gamma_P} 

In this section we prove the Gamma Conjectures for projective spaces  
following the method of Dubrovin \cite[Example 4.4]{Dubrovin98a}. 
The Gamma conjectures for projective spaces also follow from 
the computations in \cite{Iritani07,Iritani09, KKP08}. 
% A more general result for toric varieties was explained 
% in the previous section, but we include a detailed proof for the 
% projective spaces for the convenience of the reader. 
\begin{theorem} 
\label{thm:Gamma_P}
Gamma Conjectures I and II hold for the projective space 
$\P = \P^{N-1}$. An asymptotic basis is formed by mutations 
of the Gamma basis $\Gg_\P \Ch(\O(i))$ associated to 
Beilinson's exceptional collection 
$\{\O(i) : 0\le i\le N-1\}$. 
\end{theorem} 
\begin{corollary}[Guzzetti \cite{Guzzetti99}, Tanab\'{e} \cite{Tanabe}]
Dubrovin's conjecture holds for $\P=\P^{N-1}$. 
\end{corollary} 

\subsection{Quantum connection of projective spaces}
Let $\P$ denote the projective space $\P^{N-1}$. 
Let $h = c_1(\O(1))\in H^2(\P)$ denote the hyperplane class. 
The quantum product by $h$ is given by:  
\[
h\star_0 h^i = 
\begin{cases} 
h^{i+1} & \text{if $0\le i\le N-2$;} \\
1  & \text{if $i=N-1$.} 
\end{cases} 
\]
It follows that the quantum multiplication $(c_1(\P)\star_0) 
= N (h\star_0)$ is a semisimple operator with pairwise distinct eigenvalues. 
Also Property $\O$ (Definition \ref{def:conjO}) holds for $\P$ with $T=N$. 
Let $\nabla$ be the quantum connection \eqref{eq:nabla_tau_zero} 
at $\tau=0$. 
The differential equation $\nabla f(z) =0$ for a cohomology-valued 
function $f(z) = \sum_{i=0}^{N-1} f_i(z) z^{i-\frac{N-1}{2}} h^{N-1-i}$ 
reads: 
\begin{align*}
f_i(z) & = N^{-i} ( z\partial_z )^i f_0(z) \qquad 1\le i\le N-1, 
\\ 
f_0(z) & = z^N N^{-N} (z \partial_z)^N f_0(z). 
\end{align*} 
Therefore a flat section $f(z)$ for $\nabla$ is determined by 
its top-degree component $z^{-\frac{N-1}{2}}
f_0(z)$ which satisfies the scalar differential equation: 
\begin{equation} 
\label{eq:QDE_P} 
\left(D^N - N^N (-t)^N \right) f_0 =0
\end{equation} 
where we set $t:= z^{-1}$ and write $D := t\partial_t = -z \partial_z$. 
\begin{remark} 
Equation \eqref{eq:QDE_P} is the 
quantum differential equation of $\P$ 
with $t$ replaced with $-t$ (see Remark \ref{rem:QDE}). 
\end{remark} 

\subsection{Frobenius solutions and Mellin solutions} 
Solving the differential equation \eqref{eq:QDE_P} 
by the Frobenius method, we obtain a series solution 
\begin{equation} 
\label{eq:Pi}
\Pi(t;h) := e^{-N h \log t} \sum_{n=0}^\infty 
\frac{\Gamma(h -n)^N}{\Gamma(h)^N} t^{Nn} 
\end{equation} 
taking values in the ring $\C[h]/(h^N)$. 
Expanding $\Pi(t;h)$ in the nilpotent element $h$, we obtain 
a basis $\{\Pi_k(t): 0\le k\le N-1\}$ of solutions as follows:  
\begin{align*} 
\Pi(t;h) = \Pi_0(t) + \Pi_1(t) h + \cdots 
+ \Pi_{N-1}(t) h^{N-1}. 
\end{align*} 
Since $h^N=0$, we may identify $h$ with 
the hyperplane class of $\P$. 
The basis $\{\Pi_k(t)\}$ yields the fundamental 
solution $S(z) z^{-\mu} z^\rho$ of $\nabla$ 
near the regular singular point $t=z^{-1} = 0$ 
in Proposition \ref{prop:fundsol}. More precisely, we have: 
\begin{equation}
\label{eq:fundsol_P}
S(z) z^{-\mu} z^\rho h^{N-1-k} =\begin{pmatrix}
t^{-\frac{N-1}{2}}(-\frac{1}{N} D)^{N-1} \Pi_k(t) \\
 \vdots \\
t^{\frac{N-3}{2}} (-\frac{1}{N}D) \Pi_k(t)   \\ 
t^{\frac{N-1}{2}}\Pi_k(t) 
\end{pmatrix}
\end{equation} 
where the right-hand side is presented in the basis $\{1,h,\dots,h^{N-1}\}$. 

Another solution to equation \eqref{eq:QDE_P} 
is given by the Mellin transform of $\Gamma(s)^N$: 
\begin{equation}
\label{eq:Psi} 
\Psi(t) := \frac{1}{2\pi\iu} \int_{c- \iu \infty}^{c + \iu \infty} 
\Gamma(s)^N t^{-Ns} ds 
\end{equation} 
with $c>0$. The integral does not depend on $c>0$. 
One can easily check that $\Psi(t)$ satisfies equation \eqref{eq:QDE_P} 
by using the identity $s\Gamma(s) = \Gamma(s+1)$. 
Comparing \eqref{eq:Pi} and \eqref{eq:Psi}, we can think of 
$\Psi(t)$ as being obtained from $\Pi(t;h)$ by 
multiplying $\Gamma(h)^N$ and replacing the 
discrete sum over $n$ with a ``continuous" sum 
(i.e.~integral) over $s$.

Note that we have a standard determination for $\Pi_k(t)$ and $\Psi(t)$ 
along the positive real line by requiring that $\log t \in \R$ 
and $t^{-Ns} = e^{-Ns \log t}$ for $t\in \R_{>0}$. 
We see that $\Psi(t)$ gives rise to a flat section 
with the smallest asymptotics along the positive real line 
(see \S \ref{subsec:smallest_asymptotics}). 

\begin{proposition}
\label{prop:Laplacemethod_P} 
The solution $\Psi(t)$ \eqref{eq:Psi} of \eqref{eq:QDE_P} satisfies the asymptotics 
\[
\Psi(t) \sim  C  t^{-\frac{N-1}{2}}e^{-Nt} (1 +O(t^{-1}))  
\]
as $t\to \infty$ in the sector $-\frac{\pi}{2} < \arg t < \frac{\pi}{2}$, 
where $C = N^{-1/2} (2\pi)^{(N-1)/2}$. 
Write the Gamma class of $\P=\P^{N-1}$ as 
$\Gg_\P = \Gamma(1+h)^N = \sum_{k=0}^{N-1} c_k h^{N-1-k}$. 
Then we have the following connection formula: 
\[
\Psi(t) = \sum_{k=0}^{N-1} c_k \Pi_k(t) = \int_{\P} 
\Gg_\P \cup \Pi(t;h) 
\]
under the analytic continuation along the positive real line. 
\end{proposition}

The connection formula in this proposition 
and equation \eqref{eq:fundsol_P} show that: 
\begin{equation} 
\label{eq:y_0_P} 
S(z) z^{-\mu} z^\rho \Gg_\P = 
\begin{pmatrix}
t^{-\frac{N-1}{2}}(-\frac{1}{N} D)^{N-1} \Psi(t) \\
 \vdots \\
t^{\frac{N-3}{2}} (-\frac{1}{N}D) \Psi(t)   \\ 
t^{\frac{N-1}{2}}\Psi(t) 
\end{pmatrix}. 
\end{equation} 
This flat section has the smallest asymptotics 
$\sim e^{-Nt}$ as $t\to +\infty$ 
(see \eqref{eq:smallest_asymptotics}). 
Therefore we conclude (recall Conjecture \ref{conj:GammaI}): 

\begin{corollary} 
Gamma Conjecture I holds for $\P$. 
\end{corollary}

\begin{proof}[Proof of Proposition \ref{prop:Laplacemethod_P}]
The function $\Psi(t)$ is one of Meijer's $G$-functions 
\cite{Meijer46} and the large $t$ asymptotics here was obtained 
by Meijer, see equations (22), (23), (24) \emph{ibid}. 
Here we follow Dubrovin \cite[Example 4.4]{Dubrovin98a} and use 
the method of stationary phase to obtain the asymptotics of $\Psi(t)$. 
We write 
\[
\Psi(t) = \frac{1}{2\pi\iu} \int_{c-\iu \infty}^{c+ \iu \infty} 
e^{f_t(s)}  ds 
\]
with $f_t(s) = N (\log \Gamma(s) - s \log t)$. The integral 
is approximated by a contribution around the critical point of 
$f_t(s)$. Using Stirling's formula 
\[
\log \Gamma(s) \sim \left(s-\frac{1}{2}\right) 
\log s - s + \frac{1}{2} \log (2\pi) + O(s^{-1}) 
\]
we find a critical point $s_0$ of $f_t(s)$ such that 
$s_0 \sim t + \frac{1}{2} + O(t^{-1}) $ as $t \to \infty$. Then we have 
\begin{align*} 
f_t(s_0) & \sim -Nt + \frac{N}{2} \log (2\pi/t) + O(t^{-1}), \\
f_t''(s_0) &\sim \frac{N}{t}  + O(t^{-2}). 
\end{align*} 
From these we obtain the asymptotics: 
\[
\Psi(t) \sim \frac{1}{\sqrt{2\pi f_t''(s_0)}}e^{f_t(s_0)} 
= \frac{1}{\sqrt{N}} \left( \frac{2\pi}{t} \right)^{\frac{N-1}{2}} 
e^{-Nt}. 
\]
To show the connection formula, we close the integration contour 
in \eqref{eq:Psi} to the left and express $\Psi(t)$ as the sum of 
residues at $s=0, -1,-2,-3,\dots$. We have for $n\in \Z_{\ge 0}$ 
\begin{align*} 
\Res_{s=-n} \Gamma(s)^N t^{-Ns} ds & = 
\Res_{h=0} \Gamma(h - n)^N t^{Nn- Nh} dh
\\ 
& = \int_{\P} \Gamma(1+h)^N \frac{\Gamma(h -n)^N}{\Gamma(h)^N} 
t^{Nn - Nh}. 
\end{align*} 
In the second line we used the fact that $\Gamma(1+h)/\Gamma(h) = h$ and 
$\int_\P g(h) = \Res_{h=0} (g(h)/h^N) dh$ for 
any $g(h) \in \C[\![h]\!]$. 
Therefore we have 
\[
\Psi(t) = \sum_{n=0}^\infty \Res_{s=-n} \Gamma(s)^N t^{-Ns} ds 
= \int_{\P} \Gg_\P \cup \Pi(t;h) 
\]
as required. 
\end{proof}

\subsection{Monodromy transformation and mutation} 
We use monodromy transformation 
and mutation to deduce Gamma Conjecture II for $\P$ 
from the truth of Gamma Conjecture I for $\P$. 

The differential equation \eqref{eq:QDE_P} is invariant under 
the rotation $t \to e^{2\pi\iu/N}t$. 
Therefore the functions $\Psi^{(j)}(t)= \Psi(e^{-2\pi\iu j /N}t)$, $j\in \Z$ 
are also solutions to \eqref{eq:QDE_P}. 
By changing co-ordinates $t \to e^{-2\pi\iu j/N} t$ in 
Proposition \ref{prop:Laplacemethod_P}, we find that 
$\Psi^{(j)}$ satisfies the asymptotic condition
\begin{equation} 
\label{eq:asymp_Psi_j}
\Psi^{(j)}(t) \sim C_j t^{-\frac{N-1}{2}} e^{-N \zeta^{-j}} 
\end{equation} 
in the sector $-\frac{\pi}{2}+ \frac{2\pi j}{N}< \arg t 
<\frac{\pi}{2} + \frac{2\pi j}{N}$ 
where $\zeta = \exp(2\pi\iu/N)$ and $C_j = C e^{(N-1)\pi\iu j/N}$. 
Using $\Pi(e^{-2\pi\iu j/N} t;h)=\Ch(\O(i))\cup \Pi(t;h)$, 
we also find the connection formula 
\[ 
\Psi^{(j)}(t) = \int_\P \Gg_\P \Ch(\O(j)) \cup \Pi(t;h).  
\] 
Let $y_j(z)$ be the flat section corresponding to $\Psi^{(j)}$ 
(cf.~\eqref{eq:y_0_P}): 
\[
y_j(z) := (2\pi)^{-\frac{N-1}{2}} 
\begin{pmatrix}
t^{-\frac{N-1}{2}}(-\frac{1}{N} D)^{N-1} \Psi^{(j)}(t) \\
 \vdots \\
t^{\frac{N-3}{2}} (-\frac{1}{N}D) \Psi^{(j)}(t)   \\ 
t^{\frac{N-1}{2}}\Psi^{(j)}(t) 
\end{pmatrix}. 
\]
The above connection formula and \eqref{eq:fundsol_P} show that 
\begin{equation} 
\label{eq:y_j_O(j)}
y_j(z) = (2\pi)^{-\frac{N-1}{2}} 
S(z) z^{-\mu} z^\rho \left( \Gg_\P  \Ch(\O(i)) \right) 
= \Phi\left(\Gg_\P \Ch(\O(i))\right)  
\end{equation} 
where we recall that $\Phi$ was defined in \eqref{eq:coh_framing}. 
Since the sectors where the asymptotics \eqref{eq:asymp_Psi_j} 
of $\Psi^{(j)}$ hold depend on $j$, 
the flat sections $y_j(z)$ do not quite form the asymptotically exponential 
fundamental solution in the sense of Proposition \ref{prop:repeated_eigenvalues}.
We will see however that they give it after a sequence of mutations. 

Recall from Proposition \ref{prop:normalization} 
that the flat section $y_0(z)$ with the smallest asymptotics (along $\R_{>0}$) 
can be written as the Laplace integral: 
\[
y_0(z) = \frac{1}{z} \int_T^\infty \varphi(\lambda) e^{-\lambda/z} d\lambda
\]
for some $\hnabla$-flat section $\varphi(\lambda)$ holomorphic 
near $\lambda =T(=N)$. 
Recall that $\hnabla$ is the Laplace dual \eqref{eq:2nd_conn} 
of $\nabla$ at $\tau=0$. 
This integral representation is valid when $\re(z) >0$. We have 
\begin{align*} 
y_j(z) & = e^{2\pi\iu j \mu/N} y_0(e^{2\pi\iu j/N} z) 
= \frac{\zeta^{-j}}{z}  \int_T^\infty e^{2\pi\iu j\mu/N} \varphi(\lambda) 
e^{-\zeta^{-j} \lambda/z}d\lambda \\ 
& = \int_{T \zeta^{-j} + \R_{\ge 0} \zeta^{-j}} 
e^{2\pi\iu j \mu/N} 
\varphi(\zeta^j \lambda) e^{-\lambda/z} d \lambda. 
\end{align*} 
Set $\varphi_j (\lambda) = e^{2\pi\iu j \mu/N} 
\varphi(\zeta^j \lambda)$. 
Then $\varphi_j(\lambda)$ is holomorphic near $\lambda = \zeta^{-j} T$ 
and is flat for $\hnabla$. The latter fact follows easily 
from \eqref{eq:symmetry_c1star}. 
Thus we obtain an integral representation of $y_j(z)$:  
\[
y_j(z) = \frac{1}{z} \int_{T \zeta^{-j} + \R_{\ge 0} \zeta^{-j}} 
\varphi_j(\lambda) e^{-\lambda/z} d\lambda 
\]
which is valid when $-\frac{\pi}{2} - \frac{2\pi j}{N} 
<\arg z < \frac{\pi}{2} - \frac{2\pi j}{N}$. 
By bending the radial integration path, 
we can analytically continue $y_j(z)$ for arbitrary $\arg z$. 
When $\arg z$ is close to zero, for example, we bend the paths as  
shown in Figure \ref{fig:bendingpaths_P}. 
Here we choose a range $[j_0,j_0+N-1]$ of length $N$ 
and consider a system of integration paths for $y_j(z)$, $j\in [j_0,j_0+N-1]$ 
when $\arg z$ is close to zero. 

\begin{figure}[htbp]
\begin{center} 
\rotatebox{-100}{
\includegraphics[height=150pt, bb=0 0 168 193]{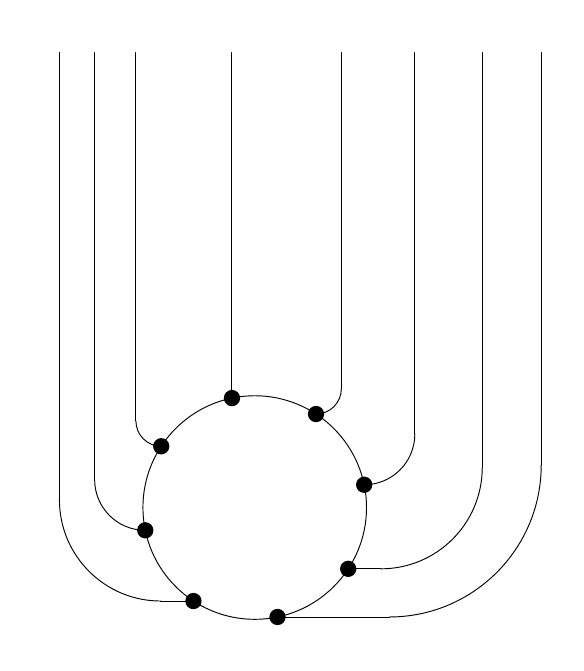}}
\end{center} 
\caption{Bent paths starting from $T \zeta^{-j}$, $j=-3,\cdots,4$ 
($N=8$).} 
\label{fig:bendingpaths_P}
\end{figure} 
\begin{figure}[htbp] 
\begin{center} 
\rotatebox{-100}{
\includegraphics[bb=0 0 91 186]{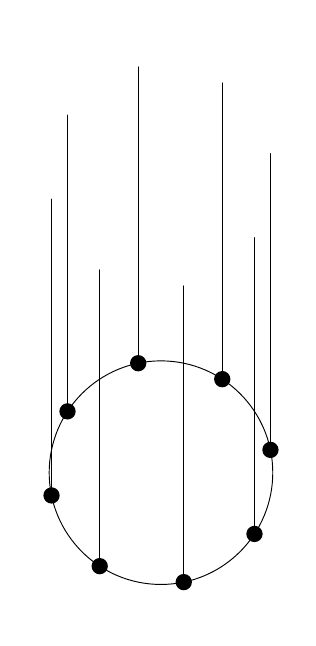}} 
\end{center} 
\caption{Straightened paths in the admissible direction $e^{\iu\phi}$.}
\label{fig:straightpaths_P} 
\end{figure} 

We can construct an asymptotically exponential fundamental solution 
in the sense of Proposition \ref{prop:repeated_eigenvalues} 
by straightening the paths  (see Figure \ref{fig:straightpaths_P}).  
Note that $\lim_{z\to +0}e^{T/z}y_0(z) = \varphi(T)$ is a 
$T$-eigenvector of $(c_1(\P)\star_0)$ of unit length 
by Proposition \ref{prop:normalization}. 
Therefore $\varphi_j(\zeta^{-j}T) = e^{2\pi\iu j \mu/N} \varphi(T)$ 
is also of unit length. It is a $\zeta^{-j} T$-eigenvector 
of $(c_1(\P)\star_0)$ by \eqref{eq:symmetry_c1star}. 
Therefore $\varphi_j(\zeta^{-j} T)$, $j\in [j_0,j_0+N-1]$ form 
a normalized idempotent basis for $\star_0$. 
Let $\phi$ be an admissible phase for the spectrum of $(c_1(\P)\star_0)$ 
which is close to zero. 
By the discussion in \S \ref{subsec:asymptotic_solution}, 
the $\nabla$-flat sections 
\[
x_j(z) = \frac{1}{z} \int_{T \zeta^{-j} + \R_{\ge 0} e^{\iu\phi}}  
\varphi_j(\lambda) e^{-\lambda/z} d\lambda 
\sim e^{-\zeta^{-j} T/z} \varphi_j(\zeta^{-j} T) 
\]
with $j\in [j_0, j_0+N-1]$ 
give the asymptotically exponential fundamental solution 
associated to $e^{\iu\phi}$. 
By the argument in \S \ref{subsec:mutation}, the two bases 
$\{y_j\}$ and $\{x_j\}$ are related by a sequence of mutations. 
Because $\{y_j\}$ corresponds to the Beilinson collection 
$\{\O(j)\}$ (see \eqref{eq:y_j_O(j)}), 
$\{x_j\}$ corresponds to a mutation of $\{\O(j)\}$  
(see \S \ref{subsec:mutation}
and Remark \ref{rem:compatibility_mutation}). 
This shows that the asymptotic basis at $\tau=0$ with respect to 
phase $\phi$ is given by a mutation 
of $\Gg_\P \Ch(\O(j))$, $j_0 \le j \le j_0+N-1$.

The proof of Theorem \ref{thm:Gamma_P} is now complete. 

\begin{figure}[htbp]
\begin{center}
\rotatebox{-90}{
\includegraphics[height=100pt,bb=0 0 220 186]{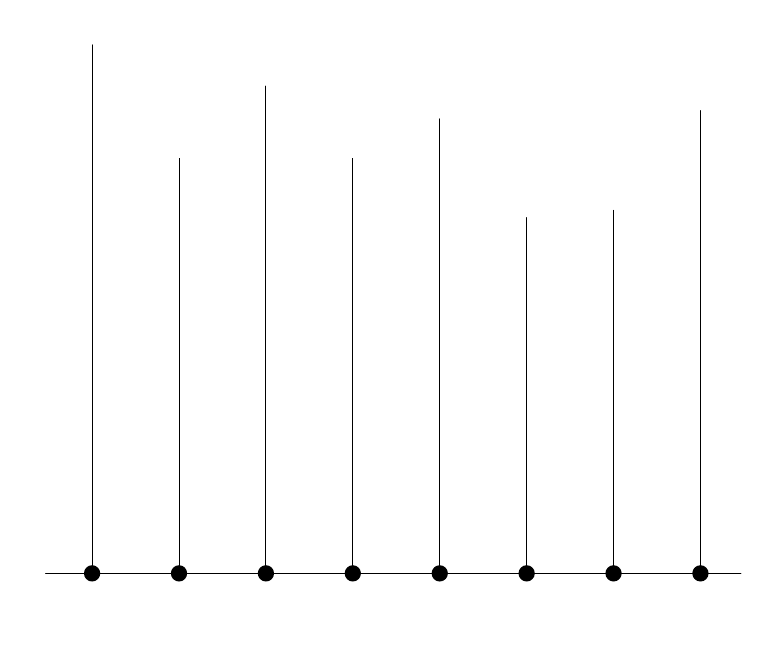}}
\end{center} 
\caption{A system of paths after isomonodromic deformation. } 
\label{fig:vertical}
\end{figure} 

\begin{remark} 
In \S \ref{subsec:mutation} and \S \ref{subsec:MRS_mutation}, 
we considered mutations for straight paths. 
Although our paths in Figure \ref{fig:bendingpaths_P} are not 
straight, we can make them straight \emph{after} 
isomonodromic deformation (see Figure \ref{fig:vertical}). 
We can therefore apply our mutation argument to such straight paths. 
When the eigenvalues of $E\star_\tau$ align as in Figure \ref{fig:vertical}, 
the corresponding flat sections $y_j$ form an asymptotically exponential 
fundamental solution and the Beilinson collection  
gives the asymptotic basis $\{\Gg_\P \Ch(\O(j))\}$ at such a point 
$\tau$.  
\end{remark}

\section{Gamma conjectures for Grassmannians} 
\label{sec:Gamma_G}
We derive Gamma Conjectures for Grassmannians of type A 
from the truth of Gamma Conjectures for projective spaces. 
Let $\G=G(r,N)$ denote the Grassmannian of $r$-dimensional 
subspaces of $\C^N$ and let 
$\P = \P^{N-1}=G(1,N)$ denote the projective space of 
dimension $N-1$. 

\subsection{Statement} 
Let $S^\nu$ denote the Schur functor for a partition $\nu
=(\nu_1\ge \nu_2 \ge \cdots\ge \nu_r)$.
Denote by $V$ the tautological bundle\footnote
{The dual $V^*$ can be identified with the universal quotient bundle 
on the dual Grassmannian $G(N-r,N)$.} over $\G=G(r,N)$. 
Kapranov \cite{Kapranov83,Kapranov} showed that 
the vector bundles $S^\nu V^*$ 
form a full exceptional collection of $\D^b_{\rm coh}(\G)$, 
where $\nu$ ranges over all partitions such that 
the corresponding Young diagrams are contained in 
an $r \times (N-r)$ rectangle, 
i.e.~$\nu_i \le N-r$ for all $i$. 

% \pph 
\begin{theorem} 
\label{thm:Gamma_G} 
Gamma Conjectures I and II hold for Grassmannians $\G$. 
An asymptotic basis of $\G$ is formed by mutations of the Gamma basis 
$\Gg_\G \Ch(S^\nu V^*)$ associated to 
Kapranov's exceptional collection 
$\{ S^\nu V^*: \nu \text{ is contained in an $r\times (N-r)$ rectangle }\}$. 
\end{theorem} 

%\pph 
\begin{corollary}[Ueda \cite{Ue05}]
Dubrovin's Conjecture holds for $\G$. 
\end{corollary} 

The proof of Gamma Conjecture II for $\G$ will be 
completed in \S \ref{subsec:wedge_Gamma}; 
the proof of Gamma Conjecture I for $\G$ 
will be completed in \S \ref{subsec:GammaI_G}. 

\subsection{Quantum Pieri and quantum Satake} 
For background material on classical cohomology rings of 
Grassmannians, we refer the reader to \cite{Fulton}. 
Let $x_1,\dots,x_r$ be the Chern roots of the dual $V^*$ 
of the tautological bundle over $\G= G(r,N)$. 
Every cohomology class of $\G$ can be written as a symmetric 
polynomial in $x_1,\dots,x_r$. We have: 
\[
H^{\udot}(\G) \cong \C[x_1,\dots,x_r]^{\frS_r} \big/
\langle h_{N- r+1}, \dots, h_N \rangle 
\]
where $h_i = h_i(x_1,\dots,x_r)$ is the $i$th complete symmetric polynomial 
of $x_1,\dots,x_r$. 
An additive basis of $H^\udot(\G)$ is given by 
the Schur polynomials\footnote
{The Satake identification $H^\udot(\G)\cong \wedge^r H^\udot(\P)$ 
is also indicated by this expression.} 
\begin{equation} 
\label{eq:Schur} 
\sigma_\lambda = \sigma_\lambda(x_1,\dots,x_r) = 
\frac{\det  (x_i^{\lambda_j+r-j})_{1\le i,j \le r}}
{\det (x_i^{r-j})_{1\le i,j \le r}}
\end{equation} 
with partition $\lambda$ in an $r\times (N-r)$ rectangle. 
The class $\sigma_\lambda$ is Poincar\'{e} dual to 
the Schubert cycle $\Omega_\lambda$ and $\deg \sigma_\lambda 
= 2|\lambda| = 2 \sum_{i=1}^r \lambda_i$. 
For $1\le k\le N-r$, we write $\sigma_k =h_k(x_1,\dots,x_r) =c_k(Q)$ 
for the Schubert class corresponding to the partition $(k\ge 0\ge \cdots\ge 0)$. 
Here $Q$ is the universal quotient bundle. 
They are called the \emph{special} Schubert classes. 
The first Chern class  is given by $c_1(\G) =N\sigma_1$. 
The classical Pieri formula describes 
the multiplication by the special Schubert classes: 
\begin{equation}
\label{eq:cl_Pieri}
\sigma_k \cup \sigma_\lambda = \sum \sigma_\nu 
\end{equation} 
where the sum ranges over all partitions $\nu =(\nu_1\ge \cdots 
\ge \nu_r)$ in an $r \times (N-r)$ rectangle such that $|\nu| = |\lambda|+k$ and 
$\nu_1 \ge \lambda_1 \ge \nu_2 \ge \lambda_2 \ge 
\cdots \ge \nu_r \ge \lambda_r$. 
Bertram's quantum Pieri formula gives the quantum multiplication 
by special Schubert classes: 
\begin{proposition}[Quantum Pieri \cite{Bertram}, see also \cite{Buch,FW04}] 
\label{prop:Pieri} 
We have 
\[
\sigma_k \star_{\sigma_1 \log q} \sigma_\lambda = \sum \sigma_\nu 
+q \sum \sigma_\mu 
\]
where the first sum is the same as the classical Pieri formula \eqref{eq:cl_Pieri}
and the second sum ranges over all $\mu$ in an $r\times (N-r)$ rectangles 
such that $|\mu| = |\lambda| + k - N$ and 
$\lambda_1 - 1 \ge \mu_1 \ge \lambda_2 -1 \ge \mu_2 \ge 
\cdots \ge \lambda_r -1 \ge \mu_r\ge 0$. 
\end{proposition} 

From the quantum Pieri formula, one can deduce 
that the quantum connection of $\G$ is the wedge power 
of the quantum connection of $\P$. 
This is an instance of the \emph{quantum Satake principle} 
of Golyshev-Manivel \cite{GMa}. 
The geometric Satake correspondence of Ginzburg 
\cite{Ginzburg95} implies that the intersection 
cohomology of an affine Schubert variety $X_\lambda$ 
in the affine Grassmannian of a complex Lie group $G$ 
becomes an irreducible representation of the 
Langlands dual Lie group $G^L$.  
Our target spaces $\P$, $\G$ arise as certain minuscule 
Schubert varieties in the affine Grassmannian of $GL_N(\C)$. 
Therefore the cohomology groups of $\P$ and $\G$ are representations 
of $GL_N(\C)$: 
$H^\udot(\P)$ is the standard representation of $GL(H^\udot(\P))\cong GL_N(\C)$ 
and $H^\udot(\G)$ is the $r$-th wedge power of the standard representation. 
In fact, by the `take the span' rational map 
$\P\times \cdots  \times \P \text{ ($r$ factors)}\dashrightarrow \G$, 
we obtain the \emph{Satake identification}  
\begin{align} 
\label{eq:Satake_id} 
\Sat \colon \wedge^r H^\udot(\P)  \overset{\cong}{\longrightarrow} 
H^\udot(\G),  
\qquad
\sigma_{\lambda_1 +r-1} \wedge \sigma_{\lambda_2 + r -2} 
\wedge\cdots \wedge \sigma_{\lambda_r}  \longmapsto \sigma_\lambda 
\end{align} 
where $\sigma_{\lambda_i + r - i}= (\sigma_1)^{\lambda_i + r-i}$ 
is the special Schubert class of $\P =G(1,N)$. 
The Satake identification endows $H^\udot(\G)$ with the structure of a 
$\mathfrak{gl}(H^\udot(\P))$-module. 
\begin{proposition}[{\cite[Theorem 2.5]{BCFK}, \cite[\S 1.6]{GMa}}]
\label{prop:q_Satake} 
The quantum product $(c_1(\G)\star_0)$ on $H^\udot(\G)$ 
coincides, upon the Satake identification $\Sat$, with the action of 
$(c_1(\P) \star_{\pi\iu(r-1)\sigma_1}) \in \mathfrak{gl}(H^\udot(\P))$ 
on the $r$-th wedge representation $\wedge^r H^\udot(\P)$. 
\end{proposition} 
\begin{proof} 
This follows from quantum Pieri in Proposition \ref{prop:Pieri}. 
Recall that $c_1(\P) = N\sigma_1$ and $c_1(\G) = N \sigma_1$. 
Thus it suffices to examine the quantum product by $\sigma_1$.  
The action of $(\sigma_1 \star_{\pi\iu(r-1) \sigma_1})$ on the basis element 
$\sigma_{\lambda_1 +r-1} \wedge \sigma_{\lambda_2 + r-2} \wedge 
\cdots \wedge \sigma_{\lambda_r}$ of the $r$-th wedge 
$\wedge^r H^\udot(\P)$ 
is given by 
\begin{multline*} 
\sum_{i=1}^r \sigma_{\lambda_1 + r-1} \wedge 
\cdots \wedge (\sigma_1 \star_{\pi\iu(r-1) \sigma_1}\sigma_{\lambda_i + r-i}) 
\wedge \cdots \wedge \sigma_{\lambda_r} \\ 
= \left( \sum_{i=1}^r \sigma_{\lambda_1 + r-1} \wedge 
\cdots \wedge \overset{{}^{\text{$i$th}}}
{\sigma_{\lambda_i + r- i +1}} \wedge 
\cdots \wedge \sigma_{\lambda_r} \right)  + (-1)^{r-1} 
\delta_{N-1, \lambda_1 + r-1} \sigma_0 \wedge \sigma_{\lambda_2 + r-2} 
\wedge \cdots \wedge \sigma_{\lambda_r} 
\end{multline*} 
Under the Satake identification, 
the first term corresponds to the classical Pieri formula 
and the second term corresponds to the quantum correction. 
The sign $(-1)^{r-1}$ cancels the sign coming from 
the permutation of $\sigma_0$ and $\sigma_{\lambda_2 + r-2} \wedge \cdots 
\wedge \sigma_{\lambda_r}$. 
\end{proof} 

\begin{remark}[\cite{GMa}] 
More generally, the action of 
the power sum $(x_1^k + \cdots + x_r^k)\star_0$ on $H^\udot(\G)$ 
coincides, upon the Satake identification with the action of 
$(\sigma_k\star_{(r-1)\pi\iu\sigma_1})\in \mathfrak{gl}(H^\udot(\P))$ 
on the $r$-th wedge representation 
$\wedge^r H^\udot(\P)$, 
cf.~Theorem \ref{thm:wedge_conn}. 
\end{remark} 

We have a similar result for the grading operator. 
The following lemma together with Proposition \ref{prop:q_Satake} 
implies that the quantum connection of $\G$ is the wedge product of the 
quantum connection of $\P$. 

\begin{lemma} 
\label{lem:Satake_grading} 
Let $\mu^\P$ and $\mu^\G$ be the grading operators of $\P$ and 
$\G$ respectively (see \S \ref{subsec:q_conn}). 
The action of $\mu^\G$ coincides, upon the Satake identification 
with the action of $\mu^\P\in \mathfrak{gl}(H^\udot(\P))$ 
on the $r$-th wedge representation $\wedge^r H^\udot(\P)$. 
\end{lemma} 
\begin{proof} 
This is immediate from the definition: we use 
$\dim \G + r(r-1) = r \dim \P$. 
\end{proof}

The Satake identification also respects the Poincar\'{e} pairing 
up to sign. 
The Poincar\'{e} pairing on $H^\udot(\P)$ induces the pairing on 
the $r$-th wedge $\wedge^r H^\udot(\P)$: 
\[
(\alpha_1\wedge \cdots \wedge \alpha_r, 
\beta_1\wedge \cdots \wedge \beta_r)_{\wedge \P} 
:= \det((\alpha_i,\beta_j)_{\P})_{1\le i,j\le r}. 
\]
\begin{lemma} 
\label{lem:Satake_pairing} 
We have $(\alpha,\beta)_{\wedge \P} = (-1)^{r(r-1)/2} 
(\Sat(\alpha), \Sat(\beta))_\G$ for 
$\alpha,\beta \in \wedge^r H^\udot(\P)$. 
\end{lemma} 
\begin{proof} 
For two partitions $\lambda$, $\mu$ in an $r\times (N-r)$ rectangle, 
we have $(\sigma_\lambda,\sigma_\mu)_\G = 1$ if 
$\lambda_i + \mu_{r-i+1} = N-r$ for all $i$, 
and $(\sigma_\lambda, \sigma_\mu)_\G =0$ otherwise. 
On the other hand, if $\lambda_i + \mu_{r-i+1} = N-r$ for all $i$, we have 
\begin{multline*} 
(\sigma_{\lambda_1+r-1} \wedge \cdots \wedge \sigma_{\lambda_r}, 
\sigma_{\mu_1+r-1} \wedge \cdots \wedge \sigma_{\mu_r}
)_{\wedge \P} \\
= (\sigma_{\lambda_1 + r-1} \wedge \cdots \wedge \sigma_{\lambda_r}, 
\sigma_{N-1-\lambda_r} \wedge \cdots \wedge 
\sigma_{N-r-\lambda_1})_{\wedge \P}  
= (-1)^{r(r-1)/2}. 
\end{multline*} 
Otherwise, the pairing can be easily seen to be zero. 
\end{proof}

\begin{remark} 
\label{rem:eigenvalues_G} 
Proposition \ref{prop:q_Satake} implies 
the well-known formula % \cite{Witten?, BCFS, Galkin-Golyshev, Chaput-Manivel-Perrin,Stroppel}
that the spectrum of 
$(c_1(\G)\star_0)$ consists of sums of $r$ distinct eigenvalues 
of $(c_1(\P)\star_{(r-1)\pi\iu \sigma_1})$, i.e.
\[
\Spec(c_1(\G)\star_0) = 
\left\{ 
N e^{(r-1)\pi\iu/N} (\zeta_N^{i_1} + \cdots + \zeta_N^{i_r}) 
: 0 \le i_1< i_2 < \cdots < i_r \le N-1\right\}
\] 
where $\zeta_N = e^{2\pi\iu/N}$. 
In particular $\G$ satisfies Property $\O$ (Definition \ref{def:conjO}) 
with $T = N \sin(\pi r/N)/\sin(\pi/N)$. 
The quantum product $(c_1(\G)\star_0)$ has pairwise distinct eigenvalues 
if $k!$ and $N$ are coprime, where $k=\min(r,N-r)$. 
\end{remark} 

\subsection{The wedge product of the big quantum connection of $\P$} 
\label{subsec:wedge_conn} 
The quantum Satake principle 
implies that the quantum connection for $\G$ (at $\tau=0$) 
is the $r$-th wedge 
product of the quantum connection for $\P$ (at $\tau = (r-1)\pi\iu$). 
By using the abelian/non-abelian correspondence of 
Bertram--Ciocan-Fontanine--Kim--Sabbah \cite{BCFK, CFKS, KimSab}, 
we observe that this is also true for the isomonodromic deformation 
corresponding to the big quantum cohomology of $\P$ and $\G$. 
When the eigenvalues of $(c_1(\G)\star_0)$ are pairwise distinct, 
this can be also deduced from the quantum Pieri (Proposition 
\ref{prop:Pieri}) and Dubrovin's 
reconstruction theorem \cite{Dubrovin94}; however it is not 
always true that $(c_1(\G)\star_0)$ has pairwise distinct 
eigenvalues (see Remark \ref{rem:eigenvalues_G}). 

The quantum product $\star_{(r-1)\pi\iu\sigma_1}$ of $\P=\P^{N-1}$ 
is semisimple and $(c_1(\P)\star_{(r-1)\pi\iu\sigma_1})$ 
has pairwise distinct eigenvalues 
$\u^\circ = \{N e^{\pi\iu (r-1)/N}e^{2\pi\iu k/N}
: 0\le k \le N-1\}$. 
As explained in \S \ref{subsec:isomonodromy}, the quantum 
connection of $\P$ has an isomonodromic deformation over the 
universal cover $C_N(\C)\sptilde$ of the configuration space \eqref{eq:config} 
of distinct $N$ points in $\C$. 
Let $\nabla^\P$ be the connection on the trivial bundle $H^\udot(\P) \times 
(C_N(\C)\sptilde \times \P^1) \to (C_N(\C)\sptilde \times \P^1)$ 
which gives the isomonodromic deformation 
(see Proposition \ref{prop:isomonodromic_deformation}). 
Here the germ $(C_N(\C)\sptilde, \u^\circ)$ at $\u^\circ$ 
is identified with the germ 
$(H^\udot(\P), (r-1)\pi\iu \sigma_1)$ by the eigenvalues of $(E\star_\tau)$ 
and $\nabla^\P$ 
is identified with the big quantum connection \eqref{eq:big_q_conn} 
of $\P$ near $\u^\circ$. 
Consider the $r$-th wedge product of this bundle: 
\[
\left( \wedge^r H^\udot(\P) \right) \times 
(C_N(\C)\sptilde \times \P^1) \to (C_N(\C)\sptilde \times \P^1)
\]
equipped with the meromorphic flat connection $\nabla^{\wedge \P}$: 
\[
\nabla^{\wedge \P} (\xi_1\wedge \xi_2 \wedge \cdots \wedge \xi_r) 
:= \sum_{i=1}^r \xi_1 \wedge \cdots \wedge (\nabla^{\P} \xi_i) 
\wedge  \cdots \wedge \xi_r. 
\] 
\begin{theorem} 
\label{thm:wedge_conn}
There exists an embedding 
$f\colon (H^\udot(\P), (r-1)\pi\iu\sigma) \cong 
(C_N(\C)\sptilde,\u^\circ) \to (H^\udot(\G), 0)$ 
between germs of complex manifolds such that the big quantum 
connection $\nabla^\G$ of $\G$ pulls back (via $f$) to 
the $r$-th wedge $\nabla^{\wedge \P}$ of the big quantum 
connection of $\P$ under the Satake identification $\Sat$ \eqref{eq:Satake_id}. 
More precisely, the bundle map 
\begin{align*} 
\wedge^r H^\udot(\P) \times (C_N(\C)\sptilde \times \P^1) 
& \longrightarrow H^\udot(\G) \times (C_N(\C)\sptilde \times \P^1)  \\ 
(\alpha, (\u,z)) &\longmapsto (\iu^{r(r-1)/2}\Sat(\alpha), (\u,z)) 
\end{align*} 
intertwines the connections $\nabla^{\wedge \P}$, 
$f^* \nabla^{\G}$ and the pairings $(\cdot,\cdot)_{\wedge \P}$, 
$(\cdot,\cdot)_\G$. 
\end{theorem} 
\begin{proof} 
The proposition follows by unpacking the definition of the
``alternate product of Frobenius manifolds" in \cite{KimSab}. 
The key ingredient is a universal property of the 
universal deformation (unfolding) of meromorphic connections 
due to Hertling and Manin \cite{HerMan}. 
The argument is straightforward, but technical. 
We recommend the reader unfamiliar with Hertling-Manin's 
universal deformation to begin by reading \cite[\S 1]{KimSab}. 

We first review the construction of the alternate product 
for the big quantum cohomology Frobenius manifold of $\P$. 
A \emph{pre-Saito structure} \cite[\S 1.1]{KimSab} with base $M$ 
is a meromorphic flat connection $\nabla$ on a trivial vector bundle 
$E^0\times (M\times \P^1) \to M\times \P^1$ 
of the form: 
\[
\nabla = d + \frac{1}{z} \CC + \left(-\frac{1}{z}U + V\right) \frac{dz}{z} 
\]
for some $\CC \in \End(E^0)\otimes \Omega^1_M$ 
and $U, V \in \End(E^0) \otimes \O_M$ (it follows from the 
flatness of $\nabla$ that 
$V$ is constant). Here $E^0$ is a finite dimensional 
complex vector space and the base space $M$ is a complex 
manifold. 
The big quantum connection \eqref{eq:big_q_conn} 
is an example of a pre-Saito structure. 
Consider the quantum connection of $\P$ 
restricted to $H^2(\P)\times \P^1$. This gives a pre-Saito structure. 
The external tensor product of this pre-Saito structure yields 
a pre-Saito structure $\nabla^{\times r}$ on the bundle: 
\[
\otimes^r H^\udot(\P) \times ((H^2(\P))^r \times \P^1) 
\to (H^2(\P))^{r} \times \P^1.  
\]
% We take the point $t^\circ = (r-1) \pi\iu \sigma_1$ as a base point on 
% $H^2(\P)$ and set 
% $\Delta t^\circ :=(t^\circ,\cdots,t^\circ)\in (H^2(\P))^r$.   
We choose a base point $t^\circ\in H^2(\P)$.  
Let $\Delta \colon H^\udot(\P) \to (H^\udot(\P))^r$ 
denote the diagonal map $\Delta(\tau) = (\tau,\dots,\tau)$. 
For notational convenience, for a point $x$ on a manifold $M$, 
we write $M_x$ for the germ $(M,x)$ of $M$ at $x$. 
We restrict the pre-Saito structure $\nabla^{\times r}$ to 
the germ $(H^2(\P))^r_{\Delta(t^\circ)}$. 
By Hertling-Manin's reconstruction theorem \cite{HerMan}, 
\cite[Corollary 1.7]{KimSab} 
and \cite[Lemma 2.1]{KimSab}, the pre-Saito structure $\nabla^{\times r}$ 
admits a universal deformation over the base $\otimes^r H^\udot(\P)$. 
Here we embed the base $(H^2(\P))^r_{\Delta(t^\circ)}$ into 
$\otimes^r H^\udot(\P)_{\chi_1(\Delta(t^\circ))}$ by the map 
\[
\chi_1: (H^2(\P))^r \hookrightarrow \otimes^r H^\udot(\P), \quad 
\chi_1 (\tau_1,\dots,\tau_r) 
= \sum_{i=1}^r 1\otimes \cdots \otimes 1 \otimes 
\overset{\text{$i$th}}{\tau_i} \otimes 1\otimes 
\cdots \otimes 1   
\]
and the universal deformation is a pre-Saito structure 
$\nabla^{\otimes r}$ on the bundle 
\[
\otimes^r H^\udot(\P) \times 
\left( \otimes^r H^\udot(\P)_{\chi_1(\Delta(t^\circ))} \times \P^1\right) 
\to \otimes^r H^\udot(\P)_{\chi_1(\Delta(t^\circ))} \times \P^1
\] 
such that $\chi_1^*\nabla^{\otimes r} = \nabla^{\times r}$. 
The construction of $\nabla^{\otimes r}$ is given in 
\cite[\S 2.2, p.234]{KimSab}; the embedding $\chi_1$ is a primitive 
of the infinitesimal period mapping (see \cite[\S 1.2]{KimSab})  
of the pre-Saito structure $\nabla^{\times r}$ attached 
to $1\otimes 1\otimes \cdots \otimes 1 \in \otimes^r H^\udot(\P)$. 
Geometrically this is the big quantum connection of 
$\P \times \cdots \times \P$ ($r$ factors).  
The pre-Saito structure $\nabla^{\otimes r}$ is equivariant with respect to 
the natural $W:=\frS_r$-action. 
We restrict the pre-Saito structure $\nabla^{\otimes r}$ to the 
$W$-invariant base 
$\Sym^r H^\udot(\P) \subset \otimes^r H^\udot(\P)$; 
fibers of the restriction are representations of $W$. 
By taking the anti-symmetric part, 
we obtain a pre-Saito structure $\nabla^{\Sym}$ on 
the bundle 
\[
\wedge^r H^\udot(\P) 
\times \left(\Sym^r H^\udot(\P)_{\chi_1(\Delta(t^\circ))} \times \P^1\right)
\to \Sym^r H^\udot(\P)_{\chi_1(\Delta(t^\circ))} \times \P^1. 
\]
Here the superscript `$\Sym$' of $\nabla^{\Sym}$ 
signifies the base space, not the fiber.  
We further restrict this pre-Saito structure to the 
subspace $\Elem^r H^\udot(\P)\subset \Sym^r H^\udot(\P)$ 
spanned by ``elementary symmetric" vectors
\[
\Elem^r H^\udot(\P) = \bigoplus_{k=1}^r 
\C \sum_{\substack{i_1,\dots,i_r \in \{0,1\} \\ 
\sum_{a=1}^r i_a = k}} 
\sigma_{i_1}\otimes \cdots \otimes \sigma_{i_r} 
\subset \Sym^r H^\udot(\P) 
\]
which contains $\chi_1(\Delta(H^2(X))$. 
By Hertling-Manin's reconstruction theorem and \cite[Lemma 2.9]{KimSab}, 
we obtain a universal deformation of the pre-Saito structure 
$\nabla^{\Sym}|_{\Elem^r H^\udot(\P)}$ over the base 
$\wedge^r H^\udot(\P)$. 
More precisely, we have an embedding
\[
\chi_2 \colon \Elem^r H^\udot(\P)_{\chi_1(\Delta(t^\circ))} 
\hookrightarrow \wedge^r H^\udot(\P)_{\chi_2(\chi_1(\Delta(t^\circ)))}  
\]
and a pre-Saito structure $\nabla^{\wedge r}$ on 
the bundle 
\begin{equation}
\label{eq:preSaito_wedge}
\wedge^r H^\udot(\P) \times 
\left(\wedge^r H^\udot(\P)_{\chi_2(\chi_1(\Delta(t^\circ)))} 
\times \P^1\right) \to 
\wedge^r H^\udot(\P)_{\chi_2(\chi_1(\Delta(t^\circ)))}  \times \P^1 
\end{equation} 
which is a universal deformation of $\nabla^{\Sym}|_{\Elem^r H^\udot(\P)}$ 
via the embedding $\chi_2$. 
The construction of $\nabla^{\wedge r}$ is given in 
\cite[\S 2.2, p.235--236]{KimSab}. 
By the construction in the proof of \cite[Corollary 1.7]{KimSab}, 
the embedding $\chi_2$ is a primitive of the infinitesimal period mapping 
of the pre-Saito structure $\nabla^{\Sym}|_{\Elem^r H^\udot(\P)}$ 
attached to $\sigma_{r-1} \wedge \sigma_{r-2} \wedge 
\cdots \wedge \sigma_0$, i.e.~
\[
d \chi_2 = \left( 
z\nabla^{\Sym} \sigma_{r-2} \wedge \sigma_{r-1} \wedge \cdots \wedge \sigma_0
\right)\Bigr|_{\Elem^r H^\udot(\P)\times \{z=0\}}. 
\]
We normalize $\chi_2$ by the initial condition 
$\chi_2(\chi_1(\Delta(t^\circ))) =t^\circ \sigma_{r-1}\wedge \sigma_{r-2} 
\wedge \cdots \wedge \sigma_0$.  
Since the restriction of $\nabla^{\Sym}$ to the diagonal 
$\chi_1 \circ \Delta \colon H^2(\P) \hookrightarrow \Elem^r H^\udot(\P)$ 
can be identified with the $r$-th wedge product of the small quantum connection 
of $\P$, we have 
\[
\chi_2(\chi_1(\Delta(\tau))) = \tau\sigma_{r-1} \wedge 
\sigma_{r-2} \wedge \cdots \wedge \sigma_0, 
\qquad \text{for $\tau \in H^2(\P)$}. 
\]
The pre-Saito structure $\nabla^{\wedge r}$ together with a 
primitive section $c \sigma_{r-1} \wedge \cdots \wedge \sigma_0$ 
(for some $c\in \C^\times$) 
and the metric $(\cdot,\cdot)_{\wedge \P}$ 
endows the base space $\wedge^r H^\udot(\P)$ with 
a Frobenius manifold structure \cite[Corollary 2.10]{KimSab}. 

The main result of Ciocan-Fontanine--Kim--Sabbah 
\cite[Theorem 2.13]{KimSab}, \cite[Theorem 4.1.1]{CFKS} 
implies that the germ of the pre-Saito structure $\nabla^{\wedge r}$ 
at $\chi_2(\chi_1(\Delta(t^\circ)))$ with 
$t^\circ = (r-1)\pi\iu \sigma_1$ is isomorphic to the 
pre-Saito structure on the bundle 
\begin{equation}
\label{eq:preSaito_G} 
H^\udot(\G) \times ((H^\udot(\G),0)\times \P^1) \to 
(H^\udot(\G),0)\times \P^1
\end{equation} 
defined by the big quantum connection of $\G$. 
By the construction in \cite[\S 3]{CFKS}, 
the isomorphism between the above two pre-Saito 
bundles \eqref{eq:preSaito_wedge}, \eqref{eq:preSaito_G} 
is induced by the Satake identification $\Sat$ 
between the fibers (up to a scalar multiple). 

By the universal property 
(see \cite[Definition 2.3]{HerMan}, \cite[\S 1.2]{KimSab})  
of the pre-Saito structures $\nabla^{\otimes r}$ 
and $\nabla^{\wedge r}$, we obtain maps 
$\tchi_1$, $\tchi_2$ extending $\chi_1$ and $\chi_2$ 
which fit into the following commutative diagram. 
\[
\xymatrix{
(H^2(\P))^r_{\Delta(t^\circ)}  \ar@{^{(}->}[r]^{\chi_1}  
\ar@{^{(}->}[d]
& \otimes^r H^\udot(\P)_{\chi_1(\Delta(t^\circ))}  
& \ar@{_{(}->}[l] \Sym^r H^\udot(\P)_{\chi_1(\Delta(t^\circ))}  
\ar[dr]_{\tchi_2} 
& \ar@{_{(}->}[l] \Elem^r H^\udot(\P)_{\chi_1(\Delta(t^\circ))} 
\ar@{^{(}->}[d]^{\chi_2} 
\\
(H^\udot(\P))^r_{\Delta(t^\circ)}  \ar[ru]_{\tchi_1} & & & 
\wedge^r H^\udot(\P)_{\chi_2(\chi_1(\Delta(t^\circ)))} 
\\ 
H^\udot(\P)_{t^\circ} \ar@{_{(}->}[u]^{\Delta} 
\ar@{^{(}->}[rruu]_{\psi} \ar[rrru]_{f} 
& & & 
}
\]
The maps $\tchi_1$, $\tchi_2$ are such that 
\begin{itemize} 
\item the pre-Saito structure over $(H^\udot(\P))^r_{\Delta(t^\circ)}$ 
defined by the $r$-fold external tensor product of the big quantum connection 
of $\P$ equals the pull-back of $\nabla^{\otimes r}$ by $\tchi_1$; 
% the isomorphism here restricts to the given one along $(H^2(\P))^r$ 
% (i.e.~the fibers of the two pre-Saito bundles are identified 
% by the identity of $\otimes^r H^\udot(\P)$); 
\item the pre-Saito structure $\nabla^{\Sym}$ 
over $\Sym^r H^\udot(\P)_{\chi_1(\Delta(t^\circ))}$ 
equals the pull-back of $\nabla^{\wedge r}$ by $\tchi_2$. 
% the isomorphism here restricts to the given one along $\Elem^r H^\udot(\P)$ 
% (i.e.~the fibers of the two pre-Saito bundles are identified 
% by the identity of $\wedge^r H^\udot(\P)$). 
\end{itemize} 
Again by the proof of \cite[Corollary 1.7]{KimSab}, the 
(non-injective) map $\tchi_1$ is a primitive 
of the infinitesimal period mapping of the pre-Saito structure 
on $(H^\udot(\P))^r_{\Delta(t^\circ)}$ (which is the $r$-fold 
external tensor product of the big quantum connection 
of $\P$) attached to $1\otimes 1\otimes \cdots \otimes 1$, 
and thus given by: 
\[
\tchi_1(\tau_1,\dots,\tau_r) = \sum_{i=1}^r 1 \otimes \cdots \otimes
\overset{\text{$i$th}}{\tau_i} \otimes \cdots \otimes 1. 
\]
Pre-composed with the diagonal map 
$\Delta\colon H^\udot(\P) \to (H^\udot(\P))^r$, 
$\tchi_1$ induces a map $\psi \colon H^\udot(\P)_{t^\circ} 
\to \Sym^r H^\udot(\P)_{\chi_1(\Delta(t^\circ))}$. 
We have that the pre-Saito structure on $H^\udot(\P)_{t^\circ}$ 
defined as the $r$-fold tensor product of the big quantum connection of 
$\P$ is isomorphic to the pull-back of $\nabla^{\otimes r}|_{\Sym^r H^\udot(\P)}$ 
by $\psi$, where the identification of fibers is the identity 
of $\otimes^r H^\udot(\P)$. 
Therefore the pull-back of the pre-Saito structure $\nabla^{\Sym}$ (defined 
as the anti-symmetric part of $\nabla^{\otimes r}|_{\Sym^r H^\udot(\P)}$) 
by $\psi$ is naturally identified with the $r$-th wedge power $\nabla^{\wedge \P}$ 
of the big quantum connection of $\P$. 
Hence the composition $f = \tchi_2 \circ \psi$ pulls back the 
pre-Saito structure $\nabla^{\wedge r}$ to the 
pre-Saito structure $\nabla^{\wedge \P}$. 
Since the pre-Saito structure $\nabla^{\wedge r}$ is identified  
with the quantum connection $\nabla^{\G}$ of $\G$ near $f(t^\circ)$, 
we have $f^* \nabla^\G \cong \nabla^{\wedge \P}$ 
under the Satake identification. 
The scalar factor $\iu^{r(r-1)/2}$ is put to make the pairings 
match (see Lemma \ref{lem:Satake_pairing}). 

Finally we show that the map $f$ 
is an embedding of germs. 
It suffices to show that the differential of $f$ at the base point 
$t^\circ = (r-1)\pi\iu\sigma_1$ is injective. 
Since we already know that $\nabla^\G$ is pulled back to 
$\nabla^{\wedge \P}$, it suffices to check that 
$z \nabla^{\wedge \P}_{\sigma_k} 
(\sigma_{r-1} \wedge \cdots \wedge \sigma_0)|_{t^\circ}$, 
$k=0,\dots,N-1$ are linearly independent, as they 
correspond to $z\nabla^\G_{df(\sigma_k)} 1 = df(\sigma_k)$. 
This follows from a straightforward computation. 
\end{proof}

\begin{corollary} 
\label{cor:semisimple_G}
Let $\psi_1,\dots,\psi_N$ be the idempotent basis of the quantum 
cohomology of $\P$ at $\u \in C_N(\C)\sptilde$ near $\u^\circ$ 
and write $\Delta_i = (\psi_i,\psi_i)_\P^{-1}$. 
Let $f$ be the embedding in Theorem \ref{thm:wedge_conn}. 
Then we have: 
\begin{enumerate}
\item the quantum product $\star_\tau$ of $\G$ is semisimple 
near $\tau=0$; 
\item the idempotent basis of $\G$ at $\tau=f(\u)$ 
is given by 
\[
\left( {\textstyle\prod_{a=1}^r \Delta_{i_a}}\right) 
\det\left((\psi_{i_a}, \sigma_{r-b})_{\P}\right)_{1\le a,b\le r}
\Sat(\psi_{i_1}\wedge \cdots \wedge \psi_{i_r}) 
\]
with $1\le i_1<i_2<\cdots<i_r\le N$; 
\item the eigenvalues of the Euler multiplication $(E^\G\star_\tau)$ 
of $\G$ at  $\tau = f(\u)$ are given by $u_{i_1} + \cdots  + u_{i_r}$ with 
$1\le i_1< i_2 < \cdots < i_r \le N$. 
\end{enumerate} 
\end{corollary} 
\begin{proof} 
The quantum product $\star_\tau$ for $\G$ is semisimple if and only 
if there exists a class $v \in H^\udot(\G)$ such that the endomorphism 
$(v \star_\tau)$ is semisimple with pairwise distinct eigenvalues. 
In this case, each eigenspace of $(v\star_\tau)$ contains a unique 
idempotent basis vector. 
On the other hand, since the quantum product of $\P$ is semisimple, 
we can find $w \in H^\udot(\P)$ such that the action 
of $(w\star_\u)$ on $\wedge^r H^\udot(\P)$ is semisimple with 
pairwise distinct eigenvalues. 
Theorem \ref{thm:wedge_conn} implies that $w\star_\u$ is 
conjugate to $df_\u(w) \star_{f(\u)}$ under $\Sat$. 
This proves Part (1). 
Moreover the eigenspace $\C \Sat(\psi_{i_1} \wedge 
\cdots \wedge \psi_{i_r})$, $i_1<i_2<\cdots<i_r$ 
of $df_\u(w) \star_{f(\u)}$ contains a unique idempotent basis vector 
$\psi_{i_1,\dots,i_r}\in H^\udot(\G)$. 
Set $\psi_{i_1,\dots,i_r} = c \Sat(\psi_{i_1} 
\wedge \cdots \wedge \psi_{i_r})$. 
Then we have 
\[
(\psi_{i_1,\dots,i_r}, \psi_{i_1,\dots,i_r})_\G 
= (\psi_{i_1,\dots, i_r} \star_{f(\u)} \psi_{i_1,\dots,i_r}, 1)_\G 
= (\psi_{i_1,\dots,i_r},1)_\G.  
\]
This implies by Lemma \ref{lem:Satake_pairing} that: 
\[
c^2 (\psi_{i_1}\wedge \cdots \wedge \psi_{i_r}, 
\psi_{i_1}\wedge \cdots \wedge \psi_{i_r})_{\wedge \P} 
= c (\psi_{i_1}\wedge \cdots \wedge \psi_{i_r}, 
\sigma_{r-1} \wedge \cdots \wedge \sigma_0)_{\wedge \P}. 
\]
Part (2) follows from this. 
Part (3) follows from the fact that the Euler multiplication 
$(E^\P\star_\tau)$ on $\wedge^r H^\udot(\P)$ is conjugate 
to $(E^\G\star_{f(\u)})$ on $H^\udot(\G)$ by $\Sat$. 
\end{proof} 

\subsection{The wedge product of MRS}
\label{subsec:wedge_MRS} 
Let $M= (V,[\cdot,\cdot), \{v_1,\dots,v_N\}, m: v_i \mapsto u_i, e^{\iu\phi})$ 
be an MRS (see \S \ref{subsec:MRS}). 
The $r$-th wedge product $\wedge^r M$ is defined by the data: 
\begin{itemize} 
\item the vector space $\wedge^r V$; 
\item the pairing $[\alpha_1\wedge \cdots \wedge \alpha_r, 
\beta_1 \wedge \cdots \wedge \beta_r) := 
\det([\alpha_i, \beta_j))_{1\le i,j\le r}$; 
\item the basis $\{v_{i_1}\wedge \cdots \wedge v_{i_r} :
1\le i_1<\cdots<i_r\le N\}$; 
\item the marking $m: v_{i_1} \wedge \cdots \wedge v_{i_r} 
\mapsto u_{i_1}+ \cdots + u_{i_r}$; 
\item the same phase $e^{\iu\phi}$. 
\end{itemize} 
Note that the basis $\{v_{i_1}\wedge \cdots \wedge v_{i_r}\}$ 
is determined up to sign because $\{v_1,\dots,v_N\}$ is an 
unordered basis. In other words, $\wedge^r M$ is \emph{defined up to sign}. 
The following lemma shows that the above data is indeed an MRS.  
Recall that $h_\phi(u) = \im(e^{-\iu\phi} u)$. 
\begin{lemma} 
Let $1\le i_1 < \cdots < i_r \le N$, $1\le j_1 < \cdots < j_r \le N$ 
be increasing sequences of integers. 
If $h_{\phi}(u_{i_1} + \cdots + u_{i_r}) = h_{\phi}(u_{j_1} + \cdots 
+ u_{j_r})$, we have 
\[
[v_{i_1}\wedge \cdots \wedge v_{i_r}, v_{j_1} \wedge \cdots \wedge 
v_{j_r}) = 
\begin{cases} 
1 & \text{ if } i_a = j_a \text{ for all $a=1,\dots,r$;} \\ 
0 & \text{ otherwise}. 
\end{cases}
\]
\end{lemma} 
\begin{proof} 
Suppose that $\prod_{a=1}^r [v_{i_a},v_{j_{\sigma(a)}}) \neq 0$ 
for some permutation $\sigma \in \frS_r$. 
Then for each $a$, we have either $i_a = j_{\sigma(a)}$ 
or $h_\phi(u_{i_a})>h_\phi(u_{j_{\sigma(a)}})$ 
by \eqref{eq:semiorth_MRS}. 
The assumption implies that $i_a = j_{\sigma(a)}$ for all $a$; 
this happens only when $\sigma=\id$. 
The lemma follows. 
\end{proof} 

The wedge product of an admissible MRS is not necessarily admissible. 
\begin{remark} 
We can define the tensor product of two MRSs similarly. 
\end{remark} 

\begin{remark} 
We can show that, when two MRSs $M_1$ and $M_2$ are related by 
mutations (see \S \ref{subsec:MRS_mutation}), 
$\wedge^r M_1$ and $\wedge^r M_2$ are also related by 
mutations. We omit a proof of this fact since we do not use 
it in this paper; the details are left to the reader. 
\end{remark}

\subsection{MRS of Grassmannian} 
Using the result from \S \ref{subsec:wedge_conn}, 
we show that the MRS of $\G$ is isomorphic to the wedge product 
of the MRS of $\P$. 
By the results in \S \ref{subsec:wedge_conn}, the flat connections 
$\nabla^{\P}$ and $\nabla^{\wedge\P}$ on the base $C_N(\C)\sptilde$ 
give isomonodromic deformations of the quantum connections 
of $\P$ and $\G$ respectively. 
Therefore we can define the MRSs of $\P$ or $\G$ 
at a point $\u\in C_N(\C)\sptilde$ with respect to 
an admissible phase $\phi$ 
following \S \ref{subsec:isomonodromy} 
(and \S\ref{subsec:MRS_Fano}). 

\begin{proposition} 
\label{prop:MRS_G}
Let $(H^\udot(\P), [\cdot,\cdot), \{A_1,\dots,A_N\}, 
A_i \mapsto u_i, e^{\iu\phi})$ 
be the MRS of $\P$ 
at $\u \in C_N(\C)\sptilde$ with respect to phase $\phi$. 
Suppose that the phase $\phi$ is admissible for the $r$-th wedge 
$\{u_{i_1}+\cdots + u_{i_r} : i_1<i_2<\cdots<i_r\}$ of the spectrum 
of $(E\star_\u)$. 
Then the MRS of $\G$ at $\u$ with respect to $\phi$ 
is given by the asymptotic basis 
\[
\frac{1}{(2\pi \iu)^{r(r-1)/2}} 
e^{-(r-1) \pi\iu \sigma_1}  
\Sat(A_{i_1}\wedge \cdots \wedge A_{i_r})
\]
marked by $u_{i_1} + \cdots + u_{i_r}$ 
with $1\le i_1<i_2<\cdots<i_r\le N$, where 
$\Sat$ is the Satake identification \eqref{eq:Satake_id}.   
In other words, one has $\MRS(\G,\u,\phi)\cong 
\wedge^r \MRS(\P,\u,\phi)$ (see \S \ref{subsec:wedge_MRS}). 
\end{proposition} 
\begin{proof} 
Let $y_1(z),\dots,y_N(z)$ be the basis of asymptotically 
exponential flat sections for $\nabla^\P|_{\u}$ 
corresponding to $A_1,\dots,A_N$. 
They are characterized by the asymptotic condition 
$y_i(z) \sim e^{-u_i/z} (\Psi_i + O(z))$ in the sector 
$|\arg z - \phi|<\frac{\pi}{2}+ \epsilon$ for some $\epsilon>0$, 
where $\Psi_1,\dots,\Psi_N$ are normalized idempotent basis 
for $\P$. 
For $1\le i_1<i_2<\cdots<i_r\le N$, 
$y_{i_1,\dots,i_r}(z) := y_{i_1}(z) \wedge \cdots \wedge y_{i_r}(z)$ 
gives a flat section 
for $\nabla^{\wedge \P}|_\u$ and satisfies 
the asymptotic condition 
\[
y_{i_1,\dots,i_r}(z) \sim 
e^{-(u_{i_1} + \cdots +u_{i_r})/z} 
(\Psi_{i_1} \wedge \cdots \wedge \Psi_{i_r} + O(z))
\]
in the same sector. 
Note that $\iu^{r(r-1)/2} \Sat(\Psi_{i_1} \wedge \cdots \wedge \Psi_{i_r})$ 
give (analytic continuation of) the normalized idempotent basis 
for $\G$ by Theorem \ref{thm:wedge_conn} and 
Corollary \ref{cor:semisimple_G}. 
Therefore the asymptotically exponential flat sections
$\iu^{r(r-1)/2}\Sat(y_{i_1,\dots,i_r}(z))$ 
give rise to the asymptotic basis of $\G$ for $\u$ and $\phi$. 

Let $S^\P(\tau,z) z^{-\mu^\P} z^{\rho^\P}$ with 
$\tau\in H^\udot(\P)$ 
denote the fundamental solution for the big quantum connection 
of $\P$ as in Remark \ref{rem:S_general_tau}. 
The group elements $S^\P(\tau,z), z^{-\mu^\P}, z^{\rho^\P} 
\in GL(H^\udot(\P))$ naturally act on the 
$r$-th wedge representation $\wedge^r H^\udot(\P)$; 
we denote these actions by the same symbols. 
Define the $\End(H^\udot(\G))$-valued function $S^\G(z)$ by 
\[
S^\G(z) \Sat(\alpha) 
= e^{(t^\circ \cup)/z} \Sat(S^\P(t^\circ,z) \alpha) 
\qquad \text{with }
t^\circ := (r-1)\pi\iu \sigma_1
\]
for all $\alpha \in \wedge^r H^\udot(\P)$. 
This satisfies $S^\G(z=\infty) =\id_{H^\udot(\G)}$. 
Using Lemma \ref{lem:Satake_grading} and the `classical' Satake 
for the cup product by $c_1$ 
(cf.~Proposition \ref{prop:q_Satake}), we find: 
\begin{align*}
z^{\rho^\G} \Sat(\alpha) 
=  \Sat( z^{\rho^\P} \alpha), \qquad 
z^{\mu^\G} \Sat(\alpha) 
= \Sat( z^{\mu^\P}\alpha)
\end{align*}
where $\mu^\G$ is the grading operator of $\G$ and 
$\rho^\G = (c_1(\G) \cup)$. 
Therefore: 
\begin{equation}
\label{eq:SG_Sat}
S^\G(z) z^{-\mu^\G} z^{\rho^\G} 
\Sat(\alpha) 
= e^{(t^\circ \cup)/z} \Sat( 
S^\P(t^\circ,z)  z^{-\mu^\P} z^{\rho^\P} \alpha). 
\end{equation} 
These sections are flat for $\nabla^\G|_{\tau=0}$ by the `quantum' Satake 
(or Theorem \ref{thm:wedge_conn}). 
Note that we have: 
\[
z^{\mu^\G} S^\G(z) z^{-\mu^\G} 
\Sat(\alpha) \\
= e^{(t^\circ \cup)} \Sat(z^{\mu^\P}S^\P(t^\circ,z) z^{-\mu^\P} 
\alpha). 
\] 
By Lemma \ref{lem:S_shift} below, 
we have $[z^{\mu^\G} S^\G(z) z^{-\mu^\G}]_{z=\infty} 
=\id_{H^\udot(\G)}$.  
Hence $S^\G(z) z^{-\mu^\G} z^{\rho^\G}$ coincides 
with the fundamental solution of $\G$ from Proposition \ref{prop:fundsol}. 

The asymptotic basis $A_1,\dots,A_N$ of $\P$ are related to 
$y_1(z),\dots,y_N(z)$ as (see \S \ref{subsec:MRS_Fano}) 
\[
y_i(z)\Big|_{\substack{\text{\tiny parallel transport} \\ 
\text{\tiny to $\u^\circ =\{\tau_\P=t^\circ\}$}}} = 
\frac{1}{(2\pi)^{\dim \P/2}}S^\P(t^\circ,z) z^{-\mu^\P} z^{\rho^\P} A_i. 
\]
This together with \eqref{eq:SG_Sat} and the definition of 
$y_{i_1,\dots,i_r}(z)$ implies that
\begin{multline*} 
\iu^{r(r-1)/2}\Sat( y_{i_1,\dots,i_r}(z)) 
\Big|_{\text{\tiny parallel transport to $\u^\circ$}} \\  
= \frac{1}{(2\pi)^{\dim \G/2}}  
S^\G(z) z^{-\mu^\G} z^{\rho^\G} 
\left[ \frac{1}{(-2\pi\iu)^{r(r-1)/2}} 
e^{-(t^\circ \cup)} 
\Sat(A_{i_1}\wedge \cdots \wedge A_{i_r})\right]. 
\end{multline*}
Recall that the base point $\u^\circ \in C_N(\C)\sptilde$ 
corresponds to $0\in H^\udot(\G)$ for $\G$ and 
to $t^\circ\in H^\udot(\P)$ for $\P$.  
The conclusion follows from this. 
(Note also that the asymptotic basis is defined only up to sign.) 
\end{proof} 

\begin{lemma}
\label{lem:S_shift} 
The fundamental solution $S(\tau,z)z^{-\mu} z^\rho$ 
in Remark \ref{rem:S_general_tau} satisfies 
$[z^{\mu} S(\tau,z)z^{-\mu}]_{z=\infty} = e^{-(\tau\cup)}$ 
for $\tau\in H^2(F)$.   
\end{lemma} 
\begin{proof} 
The differential equation for $T(\tau,z) = z^{\mu} S(\tau,z) z^{-\mu}$ in 
the $\tau$-direction reads 
$\partial_\alpha T(\tau,z) + z^{-1} 
(z^{\mu} (\alpha\star_\tau) z^{-\mu}) T(\tau,z) =0$ 
for $\alpha \in H^\udot(F)$. 
If $\tau,\alpha \in H^2(F)$, we have that 
$z^{-1} (z^{\mu} (\alpha\star_\tau) z^{-\mu})$ 
is regular at $z=\infty$ and equals $(\alpha \cup)$ there.  
Here we use the fact that $F$ is Fano and the 
divisor axiom (see Remark \ref{rem:qc_formalseries}). 
The conclusion follows by solving the differential 
equation along $H^2(F)$. 
\end{proof} 

\begin{remark} 
When we identify the MRS of $\G$ with the $r$-th wedge of the MRS 
of $\P$, we should use the identification 
$(2\pi\iu)^{-r(r-1)/2} e^{-(r-1)\pi\iu \sigma_1} 
\Sat \colon \wedge^r H^\udot(\P) \cong H^\udot(\G)$ 
that respects the pairing $[\cdot,\cdot)$ in \eqref{eq:[)_coh}. 
\end{remark}

\subsection{The wedge product of Gamma basis}
\label{subsec:wedge_Gamma}

We show that the $r$-th wedge of the Gamma basis 
$\{\Gg_\P \Ch(\O(i))\}$ given by 
Beilinson's exceptional collection for $\P$ 
matches up with the 
Gamma basis $\{\Gg_\G \Ch(S^\nu V^*)\}$ 
given by Kapranov's exceptional collection for $\G$. 
In view of the truth of Gamma Conjecture II for $\P$ 
(Theorem \ref{thm:Gamma_P}) 
and Proposition \ref{prop:MRS_G},  
the following proposition completes the proof of  
Gamma Conjecture II for $\G$. 

\begin{proposition}
\label{prop:wedge-basis} 
Let $N-r \ge \nu_1 \ge \nu_2 \ge \cdots \ge \nu_r \ge 0$ 
be a partition in an $r \times (N-r)$ rectangle. 
We have 
\[
\Gg_\G \Ch(S^\nu V^*) = (2\pi\iu)^{-\binom{r}{2}}
e^{-(r-1)\pi\iu \sigma_1} 
\Sat \left(\Gg_\P \Ch(\O(\nu_1+r-1)) \wedge 
\cdots \wedge \Gg_\P \Ch(\O(\nu_r)) \right).  
\]
\end{proposition} 

We give an elementary algebraic proof 
of Proposition \ref{prop:wedge-basis} in this section. 
A geometric proof will be discussed in \S \ref{subsec:abnonab}. 
% In this section we prove this proposition. 
Let $x_1,\dots,x_r$ denote the Chern roots of $V^*$ as before. 
(Recall that $V$ is the tautological bundle on $\G$.)
Let $h= c_1(\O(1))$ denote the hyperplane class on $\P = \P^{N-1}$.

\begin{lemma} 
\label{lem:Satake_det}
Let $f_1(z), \dots, f_r(z)$ be power series in $\C[\![z]\!]$. One has 
\[ 
\Sat (f_1(h) \wedge f_2(h)  
\wedge \dots \wedge f_r(h) )=
\frac{\det (f_j(x_i))_{1\le i,j\le r}}{\prod_{i<j} (x_i - x_j)}. 
\]
\end{lemma} 
\begin{proof} 
For monomial $f_i(z)$'s this follows by the definition of 
the Schur polynomials \eqref{eq:Schur} and 
the Satake identification $\Sat$ \eqref{eq:Satake_id}. 
The general case follows by linearity. 
\end{proof} 

\begin{lemma} 
\label{lem:ChSQ} 
One has 
\[
\Ch(S^\nu V^*) = \frac{
\det( e^{2\pi\iu x_i (\nu_j+ r-j)})_{1\le i,j\le r}}
{\prod_{i<j} (e^{2\pi\iu x_i} - e^{2\pi\iu x_j})}
\]
\end{lemma} 
\begin{proof} 
Let $L_1,\dots,L_r$ be the $K$-theoretic Chern roots of $V^*$ 
so that $[V^*] = L_1 + \cdots + L_r$. 
The $K$-class $[S^\nu V^*]$ can be expressed as the 
Schur polynomial $\sigma_\nu(L_1,\dots, L_r)$ 
in $L_1,\dots,L_r$. 
The lemma follows from the definition of the Schur polynomial 
\eqref{eq:Schur} and $\Ch(L_i) = e^{2\pi\iu x_i}$. 
\end{proof} 

\begin{lemma} 
\label{lem:Gamma_G} 
The Gamma class of $\G$ is given by 
\[
\Gg_\G =
(2\pi\iu)^{-\binom{r}{2}} e^{-(r-1) \pi\iu \sigma_1} 
\prod_{i<j}  \frac{e^{2\pi\iu x_i} - e^{2\pi\iu x_j}}
{x_i - x_j}  \prod_{i=1}^r \Gamma(1+x_i)^N 
\]
\end{lemma} 
\begin{proof} 
The tangent bundle $T\G$ of $\G$ is isomorphic to 
$\Hom(V, Q)$, where $Q$ is the universal quotient bundle. 
The exact sequence $0 \to V \to \O_\G^{\oplus N} \to Q \to 0$ 
implies that $[T\G] = [\Hom(V, \O)^{\oplus N}] - [\Hom(V,V)]
=N[V^*] - [V^* \otimes V]$ in the $K$-group. 
Thus we have: 
\[  
\Gg_\G = \frac{\prod_{i=1}^r \Gamma(1+x_i)^N}
{\prod_{1\le i,j\le r} \Gamma(1 + x_i -x_j)}. 
\]
The denominator of the right-hand side equals 
\begin{align*} 
\prod_{i<j} \bigl( \Gamma(1 + x_i - x_j) \Gamma(1 - x_i + x_j) \bigr) 
& = \prod_{i<j} \frac{2\pi\iu(x_i-x_j)}
{ e^{\pi\iu (x_i - x_j)} - e^{-\pi\iu (x_i-x_j)}} \\ 
& = (2\pi\iu)^{\binom{r}{2}} 
\prod_{i<j} (x_i - x_j) 
\frac{e^{(r-1)\pi \iu \sigma_1}}
{\prod_{i<j} (e^{2\pi\iu x_i} - e^{2\pi\iu x_j})}
\end{align*} 
where we used the Gamma function identity \eqref{eq:Gamma_identity} and 
$e^{\pi\iu(x_i-x_j)} - e^{-\pi\iu(x_i-x_j)}
= (e^{2\pi\iu x_i }- e^{2\pi\iu x_j}) e^{-\pi\iu (x_i + x_j)}$. 
The Lemma follows. 
\end{proof} 

\begin{proof}[Proof of Proposition \ref{prop:wedge-basis}] 
By Lemmata \ref{lem:ChSQ} and \ref{lem:Gamma_G}, we have 
\[
\Gg_\G \Ch(S^\nu V^*) = (2\pi\iu)^{-\binom{r}{2}} 
e^{-(r-1) \pi \iu \sigma_1} 
\frac{\det(e^{2\pi\iu x_i (\nu_j + r -j)})_{1\le i,j \le r}} 
{\prod_{i<j} (x_i - x_j)}
\prod_{i=1}^r \Gamma(1+x_i)^N. 
\]
By Lemma \ref{lem:Satake_det}, we have 
\[
\Sat \left( \Gg_\P \Ch(\O(\nu_1+ r-1)) \wedge \cdots \wedge 
\Gg_\P \Ch(\O(\nu_r)) \right) 
= \frac{\det \left(\Gamma(1+x_i)^N 
e^{2\pi\iu x_i(\nu_j + r-j)}\right)_{1\le i,j \le r}}
{\prod_{i<j} (x_i - x_j)}. 
\]
The conclusion follows from these formulas. 
\end{proof} 

We have now completed the proof of Gamma Conjecture II 
for $\G$. 

\subsection{Gamma Conjecture I for Grassmannians} 
\label{subsec:GammaI_G}
Here we prove that $\G$ satisfies Gamma Conjecture I. 
We may assume that $r\le N/2$ by replacing $r$ with $N-r$ if necessary. 
Recall that $\G$ satisfies Property $\O$ (Remark \ref{rem:eigenvalues_G}) 
and that the quantum product $\star_\tau$ of $\G$ is semisimple 
near $\tau=0$ (Corollary \ref{cor:semisimple_G}). 
Let $\phi$ be an admissible phase for the spectrum of $(c_1(\G)\star_0)$, 
sufficiently close to zero. 
Let $T_\G = N \sin(\pi r/N)/\sin(\pi/N)$ denote the biggest 
eigenvalue of $(c_1(\G)\star_0)$. 
In view of Proposition \ref{prop:Gamma_I_II}, it suffices to 
show that the member $A_\G$ of the asymptotic basis (at $\tau=0$ with respect 
to $\phi$) corresponding to $T_\G$ is $\pm \Gg_\G$. 
By Proposition \ref{prop:MRS_G}, we have 
\[
A_\G = (2\pi\iu)^{-\binom{r}{2}} e^{-(r-1)\pi\iu \sigma_1} 
\Sat(A_0 \wedge \cdots \wedge A_{r-1}) 
\]
where $A_k \in H^\udot(\P)$ is the member of the 
asymptotic basis for $\P$ at $\tau=(r-1)\pi\iu\sigma_1$ with respect to 
$\phi$, corresponding to the eigenvalue  
$N e^{\pi\iu(r-1)/N} e^{-2\pi\iu k/N}$. 
By the discussion in \S \ref{sec:Gamma_P}, we know that 
$A_k = \Gg_\P \Ch(\O(k))$. 
Here we use the condition that $\phi$ is close to zero and 
$r\le N/2$. 
Thus Proposition \ref{prop:wedge-basis} implies the equality 
$A_\G =\pm \Gg_\G$. 
Gamma Conjecture I for $\G$ is proved. 

The proof of Theorem \ref{thm:Gamma_G} is now complete.

\subsection{The abelian/non-abelian correspondence and the Gamma basis} 
\label{subsec:abnonab} 
We give an alternative proof of Proposition \ref{prop:wedge-basis} 
in the spirit of the abelian/non-abelian correspondence \cite{BCFK,CFKS}.  
The product $\P^{\times r} = \P^{N-1}\times \cdots \times \P^{N-1}$ 
($r$ factors) of projective spaces and the Grassmannian 
$\G = G(r,N)$ arise respectively as the GIT quotients  
$\Hom(\C^r,\C^N)/\!/ (\C^\times)^r$ and 
$\Hom(\C^r, \C^N)/\!/ GL(r,\C)$ 
of the same vector space $\Hom(\C^r, \C^N)$. 
We relate them by the following diagram: 
\[
\begin{CD} 
\F := Fl(1,2,\dots,r,N) @>{p}>> \G \\ 
@V{\iota}VV \\
\P^{\times r}  
\end{CD} 
\]
where $\F$ is the partial flag variety parameterizing 
flags $0= V_0 \subset V_1 \subset V_2 \subset \cdots 
\subset V_r \subset \C^N$ 
with $\dim V_i = i$, $p$ is the projection forgetting 
the intermediate flags 
$V_1,\dots,V_{r-1}$ and $\iota$ is the real-analytic inclusion 
sending a flag $\{V_1\subset \cdots \subset V_r\}$ to a collection 
$(V_1, V_2 \ominus V_1, \dots, V_r \ominus V_{r-1})$ of lines, 
where $V_i \ominus V_{i-1}$ is the orthogonal complement of 
$V_{i-1}$ in $V_i$. 
The projection $p$ is a fiber bundle with fiber $Fl(1,2,\dots,r)$. 
The normal bundle $\N_\iota$ of the inclusion $\iota$ is isomorphic 
to the conjugate $\overline{T_p}$ of the relative tangent bundle 
$T_p$ of $p$, as a topological vector bundle. 

Let $L_i = (V_i/V_{i-1})^*$ be the line bundle on $\F$ 
and set $x_i := c_1(L_i)$. 
We have $\Euler(T_p) = \prod_{i<j} (x_i-x_j) 
= (-1)^{\binom{r}{2}} \Euler(\N_\iota)$ and: 
\begin{align*} 
H^\udot(\G) & \cong \C[x_1,\dots,x_r]^{\frS_r}/ 
\langle h_{N-r+1},\dots,h_N \rangle; \\
H^\udot(\F) & \cong \C[x_1,\dots,x_r]/\langle h_{N-r+1},\dots,h_N \rangle; \\ 
H^\udot(\P^{\times r}) & \cong \C[x_1,\dots,x_r]/
\langle x_1^N, \dots, x_r^N \rangle. 
\end{align*} 
The injective map 
\[
\iota_* p^* \colon H^\udot(\G) \longrightarrow 
H^\udot(\P^{\times r}) \cong \otimes^r H^\udot(\P) 
\]
identifies $H^\udot(\G)$ with the anti-symmetric part 
$\wedge^r H^\udot(\P) \subset \otimes^r H^\udot(\P)$, 
where we embed $\wedge^r H^\udot(\P)$ into $\otimes^r H^\udot(\P)$ 
by $\alpha_1 \wedge \cdots \wedge \alpha_r 
\mapsto \sum_{\sigma \in \frS_r} \sgn(\sigma) 
\alpha_{\sigma(1)} \otimes \cdots \otimes \alpha_{\sigma(r)}$. 
This is inverse to the Satake identification 
$(-1)^{\binom{r}{2}} \Sat$ \eqref{eq:Satake_id}. 
The Kapranov exceptional bundle $S^\nu V^*$ is given by 
\[
S^\nu V^* = p_*(L^\nu) 
\]
with $L^\nu = L_1^{\otimes \nu_1} \otimes \cdots \otimes L_r^{\otimes \nu_r}$. 
Because $\iota^* T(\P^{\times r}) \cong 
p^* T\G \oplus T_p \oplus \overline{T_p}$ 
as topological bundles, we find  
\begin{equation} 
\label{eq:Gamma_abnonab} 
\iota^* \Gg_{\P^{\times r}} = p^* \Gg_\G \cup 
\Gg(T_p) \cup \Gg(\overline{T_p}) = p^* \Gg_\G \cup 
\Td(T_p) e^{-\pi\iu c_1(T_p)}. 
\end{equation} 
By the Grothendieck-Riemann--Roch theorem, we have  
\begin{equation} 
\label{eq:GRR}
\Ch(S^\nu V^*) = 
\Ch(p_*(L^\nu)) = (2\pi\iu)^{-\binom{r}{2}}p_*(\Ch(L^\nu) \Td(T_p) ). 
\end{equation} 
Combining \eqref{eq:Gamma_abnonab} and \eqref{eq:GRR}, 
we have: 
\begin{align*} 
\iota_* p^* & (\Gg_\G \Ch(S^\nu V^*)) 
= (2\pi\iu)^{-\binom{r}{2}} 
\iota_* p^* p_* \left( 
\Ch(L^\nu)\Td(T_p) \cup p^*\Gg_\G    \right)  \\
& = (2\pi\iu)^{-\binom{r}{2}} 
\iota_* p^* p_* \left( 
\Ch(L^\nu) e^{\pi\iu c_1(T_p)} \iota^* \Gg_{\P^{\times r}}\right) \\
& = (2\pi\iu)^{-\binom{r}{2}} 
\iota_* p^* p_*  \iota^* \left( 
\Ch\left(\O(\nu_1+ r-1) \boxtimes 
\cdots \boxtimes \O(\nu_r)\right) 
e^{-(r-1)\pi\iu \sigma_1} \Gg_{\P^{\times r}}\right). 
\end{align*} 
In the last line we used 
$\pi \iu c_1(T_p) = \pi \iu ((r-1) x_1 + (r-3) x_2 + \cdots  -(r-1)x_r)  
= - (r-1) \pi \iu \sigma_1 + 2\pi\iu ((r-1)x_1+ (r-2) x_2+ \cdots + x_{r-1})$. 
The map $\iota_*p^* p_* \iota^*\colon \otimes^r H^\udot(\P) 
\to \otimes^r H^\udot(\P)$ is the anti-symmetrization map 
and the quantity in the last line can be identified an element 
\[
(-2\pi\iu)^{-\binom{r}{2}} 
e^{-(r-1)\pi\iu \sigma_1} 
\left(\Gg_\P \Ch(\O(\nu_1+r-1)) \wedge \cdots \wedge \Gg_\P \Ch(\O(\nu_r))
\right). 
\]
of the wedge product $\wedge^r H^\udot(\P)$. 
Since $\iota_* p^*$ is inverse to $(-1)^{\binom{r}{2}}\Sat$, 
the conclusion of Proposition \ref{prop:wedge-basis} follows. 

\begin{remark} 
Most of the above discussions can be applied to  
a general abelian/non-abelian correspondence. 
The bundle $T_p$ corresponds to the sum of positive roots. 
\end{remark}

\appendix
\section{$\zeta$-function regularization}
\label{app:zeta}
The $\zeta$-function regularization \eqref{eq:regularization} 
of the $S^1$-equivariant Euler class $e_{S^1}(\N_+)$ 
has been computed by Lu \cite[Proposition 3.5]{Lu}. 
In this appendix, we recall the definition of the $\zeta$-function 
regularization and explain the meaning of \eqref{eq:regularization}. 

For a sequence $\{\lambda_n\}_{n=1}^\infty$ of complex numbers, 
the associated $\zeta$-function is defined to be $f(s) = \sum_{n=1}^\infty 
\lambda_n^{-s}$. If $f(s)$ can be analytically continued to a holomorphic 
function around $s=0$, we define the \emph{$\zeta$-regularized product} 
of $\{\lambda_n\}_{n=1}^\infty$ to be $\exp(-f'(0))$ and 
write $\prod_{n=1}^\infty\lambda_n \sim \exp(-f'(0))$. 
This is called the \emph{$\zeta$-function regularization}. 

Let $\delta_1,\dots,\delta_{\dim X}$ denotes the Chern roots of $TX$. 
The equation \eqref{eq:regularization} follows by regarding 
$\delta_i$, $z$ as positive real numbers and 
applying the $\zeta$-function regularization to the infinite product 
\[
\frac{1}{e_{S^1}(\N_+)} = 
\prod_{i=1}^{\dim X} \prod_{n=1}^\infty \frac{1}{\delta_i + n z}. 
\]
Indeed, the $\zeta$-function regularization gives 
\[
\prod_{n=1}^\infty 
\frac{1}{\delta_i +nz} 
\sim 
\sqrt{\frac{z}{2\pi}} z^{\delta_i/z}\Gamma(1+\delta_i/z) 
\] 
where we note that the associated $\zeta$-function is 
the Hurwitz zeta function $z^{s} \zeta(-s;\delta/z+1)$. 

\bigskip 

\noindent 
\bf Acknowledgments: \rm 
We thank Boris Dubrovin, Ionut Ciocan-Fontanine, Kohei Iwaki, 
Bumsig Kim, Etienne Mann, Anton Mellit, Kaoru Ono for 
many useful suggestions and simplifications. 
We also thank anonymous referees for valuable comments 
and suggestions. 

\medskip
\noindent
\bf Funding Sources: \rm
S.G.~was supported by grant MK-1297.2014.1; 
AG Laboratory NRU-HSE, RF government grant, ag.~11.G34.31.0023; 
Grant of Leading Scientific Schools (N.Sh.~2998.2014.1); 
World Premier International Research Center Initiative (WPI Initiative), 
MEXT, Japan; and JSPS KAKENHI Grant Number 10554503. 
H.I.~was supported by JSPS KAKENHI (Kinban-C)  
Grant Number 25400069; 
JSPS KAKENHI (Wakate-B) Grant Number 19740039;  
and EPSRC (EP/E022162/1). 

\nocite{Golyshev08a}
\nocite{Guzzetti99}
\nocite{Iritani07}
\nocite{Iritani09}
\nocite{KKP08}

\bibliographystyle{amsplain}
\providecommand{\bysame}{\leavevmode\hbox to3em{\hrulefill}\thinspace}
\providecommand{\MR}{\relax\ifhmode\unskip\space\fi MR }
\providecommand{\MRhref}[2]{\href{http://www.ams.org/mathscinet-getitem?mr=#1}{#2}}
\providecommand{\arxiv}[1]{\href{http://arxiv.org/abs/#1}{arXiv:#1}}

\end{document}